\documentclass[12pt]{amsart}
\usepackage{epsf,overpic,amssymb,amscd,times,stmaryrd,euscript}

\newtheorem{thm}{Theorem}[section]
\newtheorem{theorem}[thm]{Theorem}
\newtheorem{claim}[thm]{Claim}

\newtheorem{prop}[thm]{Proposition}
\newtheorem{proposition}[thm]{Proposition}

\newtheorem{lemma}[thm]{Lemma}
\newtheorem{cor}[thm]{Corollary}
\newtheorem{conj}[thm]{Conjecture}
\newtheorem{q}[thm]{Question}
\newtheorem{rmk}[thm]{Remark}

\numberwithin{equation}{subsection}

\newcommand{\C}{\mathbb{C}}
\newcommand{\R}{\mathbb{R}}
\newcommand{\Z}{\mathbb{Z}}
\newcommand{\Q}{\mathbb{Q}}
\newcommand{\N}{\mathbb{N}}

\newcommand{\F}{\mathcal{F}}

\renewcommand{\S}{\mathcal{S}}

\newcommand{\A}{\EuScript{A}}
\newcommand{\bdry}{\partial}
\newcommand{\s}{\vskip.1in}
\newcommand{\n}{\noindent}
\newcommand{\wt}{\widetilde}

\newcommand{\be}{\begin{enumerate}}
\newcommand{\ee}{\end{enumerate}}

\begin{document}

\title{Reeb vector fields and open book decompositions}

\author{Vincent Colin}
\address{Universit\'e de Nantes, UMR 6629 du CNRS, 44322 Nantes, France}
\email{Vincent.Colin@math.univ-nantes.fr}

\author{Ko Honda}
\address{University of Southern California, Los Angeles, CA 90089}
\email{khonda@usc.edu} \urladdr{http://rcf.usc.edu/\char126 khonda}

\date{This version: September 29, 2008.}



\thanks{VC supported by the Institut
Universitaire de France and the ANR Symplexe. KH supported by an
Alfred P.\ Sloan Fellowship and an NSF CAREER Award (DMS-0237386)
and NSF Grant DMS-0805352.}

\begin{abstract}
We determine parts of the contact homology of certain contact
3-manifolds in the framework of open book decompositions, due to
Giroux. We study two cases: when the monodromy map of the compatible
open book is periodic and when it is pseudo-Anosov.  For an open
book with periodic monodromy, we verify the Weinstein conjecture. In
the case of an open book with pseudo-Anosov monodromy, suppose the
boundary of a page of the open book is connected and the fractional
Dehn twist coefficient $c={k\over n}$, where $n$ is the number of
prongs along the boundary. If $k\geq 2$, then there is a
well-defined linearized contact homology group.  If $k\geq 3$, then
the linearized contact homology is exponentially growing with
respect to the action, and every Reeb vector field of the
corresponding contact structure admits an infinite number of simple
periodic orbits.
\end{abstract}

\maketitle

\tableofcontents

\section{Background and introduction}

About ten years ago, Emmanuel Giroux \cite{Gi} described a 1-1
correspondence between isotopy classes of contact structures and
equivalence classes of open book decompositions (in any odd
dimension).  This point of view has been extremely fruitful,
particularly in dimension three.  Open book decompositions were the
conduit for defining the contact invariant in Heegaard-Floer
homology (due to Ozsv\'ath-Szab\'o \cite{OSz}).  This contact
invariant has been studied by Lisca-Stipsicz \cite{LS} and others
with great success, and has contributed considerably to the
understanding of tight contact structures on Seifert fibered spaces.
It was also open book decompositions that enabled the construction
of concave symplectic fillings for any contact 3-manifold (due to
Eliashberg~\cite{El} and Etnyre~\cite{Et1}); this in turn was the
missing ingredient in Kronheimer-Mrowka's proof of Property P for
knots \cite{KM}.  In higher dimensions, the full potential of the
open book framework is certainly not yet realized, but we mention
Bourgeois' existence theorem for contact structures on any
odd-dimensional torus $T^{2n+1}$ \cite{Bo2}.

The goal of this paper is to use the open book framework to
calculate parts of the {\em contact homology} $HC(M,\xi)$ of a
contact manifold $(M,\xi)$ adapted to an open book decomposition, in
dimension three.  Giroux had already indicated that there exists a
Reeb vector field $R$ which is in a particularly nice form with
respect to the open book: $R$ is transverse to the interior of each
page $S$, and is tangent to and agrees with the orientation of the
binding $\bdry S$ of the open book. (Here the orientation of $\bdry
S$ is induced from $S$.)  The difficulty that we encounter is that
this Reeb vector field is not {\em nice enough} in general, e.g., it
is not easy to see whether the contact homology is cylindrical, and
boundary maps are difficult to determine.  (Some results towards
understanding $HC(M,\xi)$ were obtained by Yau~\cite{Y2,Y3}.) In
this paper we prove that, for large classes of tight contact
3-manifolds, $HC(M,\xi)$ is cylindrical, and moreover that
$HC(M,\xi)\not=0$.  What enables us to get a handle on the contact
homology is a better understanding of tightness in the open book
framework.  The second author, together with Kazez and Mati\'c
\cite{HKM}, showed a contact manifold $(M,\xi)$ is tight if and only
if all its compatible open books have {\em right-veering} monodromy.
We will see that there is a distinct advantage to restricting our
attention to right-veering monodromy maps.

In this section we review some notions around open book
decompositions in dimension three.

\subsection{Fractional Dehn twist coefficients}
Let $S$ be a compact oriented surface with nonempty boundary $\bdry
S$.  Fix a reference hyperbolic metric on $S$ so that $\bdry S$ is
geodesic.  (This excludes the cases where $S$ is a disk or an
annulus, which we understand well.) Suppose that $\bdry S$ is
connected. Let $h:S\rightarrow S$ be a diffeomorphism for which
$h|_{\bdry S}=id$. If $h$ is not reducible, then $h$ is freely
homotopic to homeomorphism $\psi$ of one of the following two types:

\be
\item A {\em periodic} diffeomorphism, i.e., there is an integer
$n>0$ such that $\psi^n=id$.
\item A {\em pseudo-Anosov} homeomorphism.
\ee

Let $H: S\times[0,1]\rightarrow S$ be the free isotopy from
$h(x)=H(x,0)$ to its periodic or pseudo-Anosov representative
$\psi(x)=H(x,1)$.  We can then define $\beta: \bdry S \times [0,1]
\to \bdry S \times [0,1]$ by sending $(x,t)\mapsto (H(x,t),t)$,
i.e., $\beta$ is the trace of the isotopy $H$ along $\bdry S$. Form
the union of $\bdry S \times [0,1]$ and $S$ by gluing $\bdry S
\times \{1\}$ and $\bdry S$. By identifying this union with $S$, we
construct the homeomorphism $\beta\cup\psi$ on $S$ which is isotopic
to $h$ relative to $\bdry S$. We will assume that $h=\beta\cup\psi$,
although $\psi$ is usually just a homeomorphism in the pseudo-Anosov
case.  (More precisely, $\psi$ is smooth away from the singularities
of the stable/unstable foliations.)

If we choose an oriented identification $\bdry S\simeq \R/\Z$, then
we can define an orientation-preserving homeomorphism $f:
\R\rightarrow \R$ as follows: lift
$\beta:\R/\Z\times[0,1]\rightarrow \R/\Z\times[0,1]$ to
$\tilde\beta: \R\times[0,1]\rightarrow \R\times [0,1]$ and set
$f(x)= \tilde\beta(x,1)-\tilde\beta(x,0)+x$. We then call $\beta$ a
{\it fractional Dehn twist} by an amount $c\in \Q$, where $c$ is the
{\em rotation number} of $f$, i.e., $c=\lim_{n\rightarrow \infty}
{f^n(x)-x\over n}$ for any $x$.  In the case $\psi$ is periodic, $c$
is simply $f(x)-x$ for any $x$.  In the pseudo-Anosov case, $c$ can
be described as in the next paragraphs.

A pseudo-Anosov homeomorphism $\psi$ is equipped with a pair of
laminations --- the {\em stable and unstable measured geodesic
laminations} $(\Lambda^s,\mu^s)$ and $(\Lambda^u,\mu^u)$ ---  which
satisfy $\psi(\Lambda^s,\mu^s)=(\Lambda^s,\tau \mu^s)$ and
$\psi(\Lambda^u,\mu^u)=(\Lambda^u,\tau^{-1}\mu^u)$ for some
$\tau>1$.  (Here $\Lambda^s$ and $\Lambda^u$ are the laminations and
$\mu^s$ and $\mu^u$ are the transverse measures.)  The lamination
$\Lambda$ ($=$ $\Lambda^s$ or $\Lambda^u$) is {\em minimal} (i.e.,
does not contain any sublaminations), does not have closed or
isolated leaves, is disjoint from the boundary $\bdry S$, and every
component of $S-\Lambda$ is either an open disk or a semi-open
annulus containing a component of $\bdry S$. In particular, every
leaf of $\Lambda$ is dense in $\Lambda$.

Now the connected component of $S-\Lambda^s$ containing $\bdry S$ is
a semi-open annulus $A$ whose metric completion $\hat{A}$ has
geodesic boundary consisting of $n$ infinite geodesics
$\lambda_1,\dots,\lambda_n$. Suppose that the $\lambda_i$ are
numbered so that $i$ increases (modulo $n$) in the direction given
by the orientation on $\bdry S$.  Now let $P_i\subset A$ be a
semi-infinite geodesic which begins on $\bdry S$, is perpendicular
to $\bdry S$, and runs parallel to $\gamma_i$ and $\gamma_{i+1}$
(modulo $n$) along the ``spike'' that is ``bounded'' by $\gamma_i$
and $\gamma_{i+1}$. These $P_i$ will be referred to as the {\em
prongs}. Let $x_i=P_i\cap \bdry S$ be the endpoint of $P_i$ on
$\bdry S$.   We may assume that $\psi$ permutes (rotates) the prongs
and, in particular, there exists an integer $k$ so that $\psi:
x_i\mapsto x_{i+k}$ for all $i$.  It then follows that $c$ is a lift
of ${k\over n}\in \R/\Z$ to $\Q$.

If $\bdry S$ is not connected, then one can similarly define a
fractional Dehn twist coefficient $c_i$ for the $i$th boundary
component of $S$.

\subsection{Open book decompositions and tightness}
In this paper, the ambient 3-manifold $M$ is oriented and the
contact structure $\xi$ is cooriented.

Let $(S,h)$ be a pair consisting of a compact oriented surface $S$
and a diffeomorphism $h:S \stackrel\sim\rightarrow S$ which
restricts to the identity on $\partial S$, and let $K$ be a link in
a closed oriented 3-manifold $M$. An {\em open book decomposition}
for $M$ with {\em binding} $K$ is a homeomorphism between $((S
\times[0,1]) /_{\sim _h}, (\partial S \times[0,1]) /_{\sim_h})$ and
$(M,K)$. The equivalence relation $\sim_h$ is generated by $(x,1)
\sim_h (h(x),0)$ for all $x\in S$ and $(y,t) \sim_h (y,t')$ for all
$y \in
\partial S$ and $t,t'\in[0,1]$. We will often identify $M$ with $(S \times[0,1]) /
_{\sim _h}$; with this identification $S_t= S \times \{t\}, t\in
[0,1]$, is called a {\em page} of the open book decomposition and
$h$ is called the {\em monodromy map}. Two open book decompositions
are {\em equivalent} if there is an ambient isotopy taking binding
to binding and pages to pages. We will denote an open book
decomposition by $(S,h)$, although, strictly speaking, an open book
decomposition is determined by the triple $(S,h,K)$. There is a
slight difference between the two --- if we do not specify $K\subset
M$, we are referring to isomorphism classes of open books instead of
isotopy classes.

Every closed 3-manifold has an open book decomposition, but it is
not unique. One way of obtaining inequivalent open book
decompositions is to perform a {\em positive} or {\em negative
stabilization}: $(S',h')$ is a {\em stabilization} of $(S,h)$ if
$S'$ is the union of the surface $S$ and a band $B$ attached along
the boundary of $S$ (i.e., $S'$ is obtained from $S$ by attaching a
1-handle along $\bdry S$), and $h'$ is defined as follows. Let
$\gamma$ be a simple closed curve in $S'$ ``dual'' to the cocore of
$B$ (i.e., $\gamma$ intersects the cocore of $B$ at exactly one
point) and let $id_B \cup h$ be the extension of $h$ by the identity
map to $B \cup S$.  Also let $R_\gamma$ be the positive (or
right-handed) Dehn twist about $\gamma$.  Then for a {\em positive}
stabilization $h'=R_\gamma \circ (id_B \cup h)$, and for a {\em
negative} stabilization $h'=R_\gamma^{-1} \circ (id_B \cup h)$. It
is well-known that if $(S',h')$ is a positive (negative)
stabilization of $(S,h)$, and $(S,h)$ is an open book decomposition
of $(M,K)$, then $(S',h')$ is an open book decomposition of $(M,
K')$ where $K'$ is obtained by a Murasugi sum of $K$ (also called
the {\em plumbing} of $K$) with a positive (negative) Hopf link.

According to Giroux~\cite{Gi}, a contact structure $\xi$ is {\em
supported} by the open book decomposition $(S,h,K)$ if there is a
contact 1-form $\alpha$ which: \be
\item induces a symplectic form $d\alpha$ on each page $S_t$;
\item $K$ is transverse to $\xi$, and the orientation on $K$ given by
$\alpha$ is the same as the boundary orientation induced from $S$
coming from the symplectic structure. \ee In the 1970's, Thurston
and Winkelnkemper~\cite{TW} showed that (in Giroux's terminology)
any open book decomposition $(S,h,K)$ of $M$ supports a contact
structure $\xi$. Moreover, the contact planes can be made
arbitrarily close to the tangent planes of the pages, away from the
binding.

The following result is the converse (and more), due to
Giroux~\cite{Gi}.

\begin{thm}[Giroux]
Every contact structure $(M,\xi)$ on a closed 3-manifold $M$ is
supported by some open book decomposition $(S,h,K)$. Moreover, two
open book decompositions $(S,h,K)$ and $(S',h',K')$ which support
the same contact structure $(M,\xi)$ become equivalent after
applying a sequence of positive stabilizations to each.
\end{thm}

Akbulut-Ozbagci~\cite{AO} and Giroux (independently) also clarified
the role of Stein fillability, inspired by the work of
Loi-Piergallini \cite{LP}:

\begin{cor}[Loi-Piergallini, Akbulut-Ozbagci, Giroux] \label{holo}
A contact structure $\xi$ on $M$ is holomorphically fillable if and
only if  $\xi$ is supported by some open book $(S,h,K)$ with $h$ a
product of positive Dehn twists.
\end{cor}

The second author, together with Kazez and Mati\'c \cite{HKM},
partially clarified the role of tightness in the open book
framework.   In particular, the following theorem was obtained:

\begin{thm}\label{veer}
A contact structure $(M,\xi)$ is tight if and only if all of its
open book decompositions $(S,h)$ have {\em right-veering} $h$.
\end{thm}

We will briefly describe the notion of {\em right-veering}. Let
$\alpha$ and $\beta$ be isotopy classes, rel endpoints, of properly
embedded oriented arcs $[0,1]\rightarrow S$ with a common initial
point $\alpha(0)=\beta(0)=x\in \partial S$. Assume
$\alpha\not=\beta$.  Choose representatives $a, b$ of $\alpha,
\beta$ so that they intersect transversely (this include the
endpoints) and {\em efficiently}, i.e., with the fewest possible
number of intersections.  Then we say $\beta$ is {\em strictly to
the right of} $\alpha$ if the tangent vectors $(\dot b(0), \dot
a(0))$ define the orientation on $S$ at $x$. A monodromy map $h$ is
{\em right-veering} if for every choice of basepoint $x \in \partial
S$ and every choice of arc $\alpha$ based at $x$, $h(\alpha)=\alpha$
or is strictly to the right of $\alpha$.

\begin{thm}  \label{ot}
Suppose $h$ is freely homotopic to $\psi$ which is periodic or
pseudo-Anosov, and $c_i$ is the fractional Dehn twist coefficient
corresponding to the $i$th boundary component of $S$. \be \item If
$\psi$ is periodic, then $h$ is right-veering if and only if all
$c_i\geq 0$.  Hence $(M,\xi)$ is overtwisted if some $c_i<0$. \item
If $\psi$ is pseudo-Anosov, then $h$ is right-veering if and only if
all $c_i>0$. Hence $(M,\xi)$ is overtwisted if some $c_i\leq 0$. \ee
\end{thm}

\section{Main results}
In this article, we prove the existence and nontriviality of
cylindrical contact homology for a contact structure $(M,\xi)$ given
by an open book decomposition $(S,h)$ with periodic or pseudo-Anosov
monodromy, under favorable conditions. Here $S$ is a compact,
oriented surface with nonempty boundary $\bdry S$ (often called a
``bordered surface''), and $h:S\stackrel\sim\rightarrow S$ is an
orientation-preserving diffeomorphism which restricts to the
identity on the boundary.

One of the motivating problems in 3-dimensional contact geometry is
the following Weinstein conjecture:

\begin{conj}[Weinstein conjecture]
Let $(M,\xi)$ be a contact 3-manifold.  Then for any contact form
$\alpha$ with $\ker \alpha=\xi$, the corresponding Reeb vector field
$R=R_\alpha$ admits a periodic orbit.
\end{conj}

During the preparation of this paper, Taubes~\cite{Ta} gave a
complete proof of the Weinstein conjecture in dimension three.  Our
methods are completely different from those of Taubes, who uses
Seiberg-Witten Floer homology instead of contact homology.  In some
situations (i.e., Theorem~\ref{main} and Corollary~\ref{cor:
infinitely many simple orbits}), we prove a better result which
guarantees an infinite number of simple periodic orbits.

Prior to the work of Taubes, the Weinstein conjecture in dimension
three was verified for contact structures which admit planar open
book decompositions \cite{ACH} (also see related work of
Etnyre~\cite{Et2}), for certain Stein fillable contact structures
\cite{Ch,Ze}, and for certain universally tight contact structures
on toroidal manifolds \cite{BC}.  We also refer the reader to the
survey article by Hofer~\cite{H2}.

\subsection{The periodic case}
Our first result is the following:

\begin{thm}
The Weinstein conjecture holds when $(S,h)$ has periodic monodromy.
\end{thm}

\begin{proof}
By the work of Hofer~\cite{H1}, the Weinstein conjecture holds for
overtwisted contact structures, on $S^3$ and manifolds which are
covered by $S^3$, and on manifolds for which $\pi_2(M)\not=0$.  If
any of the fractional Dehn twist coefficients $c_i$ are negative,
then $(M,\xi)$ is overtwisted by Theorem~\ref{ot}.  If any $c_i=0$,
then $h=id$.  In this case, $M$ is a connected sum of $(S^1\times
S^2)$'s, and has $\pi_2(M)\not=0$.

When all the $c_i>0$ and the universal cover of $M$ is $\R^3$, then
the cylindrical contact homology is well-defined and nontrivial by
Theorem~\ref{thm:cylindrical}.
\end{proof}

In the periodic case, we also prove Theorem~\ref{thm:right-veering
is exact}, which states that $(M,\xi)$ is tight if and only if $h$
is right-veering.  Moreover, the tight contact structures are
$S^1$-invariant and also Stein fillable.

\subsection{The pseudo-Anosov case}
We now turn our attention to the pseudo-Anosov case.  For
simplicity, suppose $S$ has only one boundary component. Then $h$
gives rise to two invariants: the fractional Dehn twist coefficient
$c$ and the pseudo-Anosov homeomorphism $\psi$ which is freely
homotopic to $h$. If $c\leq 0$, then the contact manifold $(M,\xi)$
is overtwisted by Theorem~\ref{ot}.  Hence we restrict our attention
to $c>0$. Let us write $c={k\over n}$, where $n$ is the number of
prongs about the unique boundary component $\bdry S$.  Our main
theorem is the following:

\begin{thm}  \label{main}
Suppose $\bdry S$ is connected and $c={k\over n}$ is the fractional
Dehn twist coefficient.
\begin{enumerate}
\item If $k\geq 2$, then any chain group $(\A(\alpha,J),\bdry)$ of the
full contact homology admits an augmentation $\varepsilon$. Hence
there is a well-defined linearized contact homology group
$HC^\varepsilon(M,\alpha,J)$.
\item If $k\geq 3$, then the linearized contact homology group $HC^\varepsilon(M,\alpha,J)$
has {\em exponential growth with respect to the action}.  (In
particular, $HC^\varepsilon(M,\alpha,J)$ is nontrivial.)
\end{enumerate}
\end{thm}

For the notions of {\em full contact homology}, {\em linearized
contact homology}, and {\em augmentations}, see
Section~\ref{section:contact homology}. The {\em action}
$A_\alpha(\gamma)$ of a closed orbit $\gamma$ with respect to a
contact $1$-form $\alpha$ is $\int_\gamma \alpha$. The linearized
contact homology group $HC^\varepsilon(M,\alpha,J)$ with respect to
the contact $1$-form $\alpha$ and adapted almost complex structure
$J$ on the symplectization is said to have {\em exponential growth
with respect to the action}, if there exist constants $c_1,c_2>0$ so
that the number of linearly independent generators in
$HC^\varepsilon(M,\alpha,J)$ which are represented by $\sum_i
a_i(\gamma_i-\varepsilon(\gamma_i))$, $a_i\in \Q$, with
$A_\alpha(\gamma_i)<L$ for all $i$ is greater than $c_1 e^{c_2L}$.
Here the contact homology groups are defined over $\Q$. The notion
of exponential growth with respect to the action is independent of
the choice of contact $1$-form $\alpha$ and almost complex structure
$J$ in the following sense: Given another $\A(\alpha',J')$, there is
a chain map $\Phi:\A(\alpha',J')\rightarrow \A(\alpha,J)$ which
pulls back the augmentation $\varepsilon$ on $\A(\alpha,J)$ to
$\Phi^*\varepsilon$ on $\A(\alpha',J')$, so that
$HC^{\Phi^*\varepsilon}(M,\alpha',J')\simeq
HC^\varepsilon(M,\alpha,J)$ and
$HC^{\Phi^*\varepsilon}(M,\alpha',J')$ has exponential growth if and
only if $HC^\varepsilon(M,\alpha,J)$ does.

Theorem~\ref{main}, together with Theorem~\ref{ot}, implies the
following:

\begin{cor}\label{thm: Weinstein}
The Weinstein conjecture holds for $(M,\xi)$ which admits an open
book with pseudo-Anosov monodromy if either $c\leq 0$ or $c\geq
{3\over n}$.
\end{cor}

In the paper~\cite{CH2}, we prove that every open book $(S,h)$ can
be stabilized (after a finite number of stabilizations) to $(S',h')$
so that $h'$ is freely homotopic to a pseudo-Anosov homeomorphism
and $\bdry S'$ is connected. This proves that ``almost all'' contact
3-manifolds satisfy the Weinstein conjecture.

\begin{rmk}
With our approach, it remains to prove the Weinstein conjecture for
$c={1\over n}$ and ${2\over n}$. The $c={1\over n}$ case is
fundamentally different, and requires a different strategy; the
$c={2\over n}$ case might be possible by a more careful analysis of
Conley-Zehnder indices.
\end{rmk}

We also have the following:

\begin{cor} \label{cor: infinitely many simple orbits}
Let $\alpha$ be a contact $1$-form for $(M,\xi)$ which admits an
open book with pseudo-Anosov monodromy and $c\geq {3\over n}$. Then
the corresponding Reeb vector field $R_\alpha$ admits an infinite
number of simple periodic orbits.
\end{cor}

In the corollary we do not require that the contact $1$-form
$\alpha$ be nondegenerate. The proof of Corollary~\ref{cor:
infinitely many simple orbits} will be given in
Section~\ref{subsection: proof of corollary}.

The main guiding philosophy of the paper is that a Reeb flow is not
too unlike a pseudo-Anosov flow on a 3-manifold, since both types of
flows are transversely area-preserving.  The difference between the
two will be summarized briefly in Section~\ref{subsection: flux} and
discussed more thoroughly in the companion paper~\cite{CHL}.


We are also guided by the fundamental work of Gabai-Oertel~\cite{GO}
on essential laminations, which we now describe as it pertains to
our situation. Suppose $S$ is hyperbolic with geodesic boundary.
Suspending the stable geodesic lamination $\Lambda^s$ of $\psi$, for
example, we obtain a codimension 1 lamination $\mathcal{L}$ on $M$,
which easily satisfies the conditions of an {\em essential
lamination}, provided $k>1$.  In particular, the universal cover
$\widetilde M$ of $M$ is $\R^3$ and each leaf of $\mathcal{L}$ has
fundamental group which injects into $\pi_1(M)$.

The following is an immediate corollary of the proof of
Theorem~\ref{main}:

\begin{cor}
A contact structure $(M,\xi)$ supported by an open book with
pseudo-Anosov monodromy with $k>1$ is universally tight with
universal cover $\R^3$.
\end{cor}

\subsection{Growth rates of contact homology}

Theorem~\ref{main} opens the door to questions about the growth
rates of linearized contact homology groups on various contact
manifolds.

\s\n {\bf Example 1.} The standard tight contact structure on $S^3$.
Modulo taking direct limits, there is a contact $1$-form with two
simple periodic orbits, both of elliptic type. The two simple
orbits, together with their multiple covers, generate the
cylindrical contact homology group. Hence the growth is linear with
respect to the action.

\s\n {\bf Example 2.} The unique Stein fillable tight contact
structure $(T^3,\xi)$, given by $\alpha= \sin (2\pi z)dx -\cos (2\pi
z)dy$ on $T^3=\R^3/\Z^3$ with coordinates $(x,y,z)$. Modulo direct
limits, the closed orbits are in $2-1$ correspondence with
$\Z^2-\{(0,0)\}$. Hence the cylindrical contact homology grows
quadratically with respect to the action.

\s\n {\bf Example 3.} The set of periodic orbits of the geodesic
flow on the unit cotangent bundle of a closed hyperbolic surface
$\Sigma$ is in 1-1 correspondence with the set of closed geodesics
of $\Sigma$. Hence the cylindrical contact homology of the
corresponding contact structure grows exponentially with respect to
the action.

\begin{q}
Which contact manifolds $(M,\xi)$ have linearized contact homology
with exponential growth? Are there contact manifolds which have
linearized contact homology with polynomial growth, where the degree
of the polynomial is greater than $2$?
\end{q}

A special case of the question is:

\begin{q}
What happens to contact structures on circle bundles over closed
hyperbolic surfaces $\Sigma$ with Euler number between $0$ and
$2g(\Sigma)-3$, which are transverse to the fibers?  Here
$g(\Sigma)$ is the genus of $\Sigma$.
\end{q}

Euler number $2g(\Sigma)-2$ corresponds to the unit cotangent bundle
case; on the other hand, Euler number $\leq-1$ corresponds to the
$S^1$-invariant case, and has linear growth.

We also conjecture the following:

\begin{conj}
Universally tight contact structures on hyperbolic manifolds have
exponential growth.
\end{conj}

\s\n {\em Organization of the paper.} The notions of {\em contact
homology} will be described in Section~\ref{section:contact
homology}. In particular we quickly review the notions of
augmentations and linearizations. Section~\ref{section:periodic} is
devoted to the periodic case. In particular, we show that a periodic
$(S,h)$ is tight if and only if $h$ is right-veering
(Theorem~\ref{thm:right-veering is exact}); moreover, a tight
$(S,h)$ with periodic monodromy is Stein fillable. In
Section~\ref{section: Rademacher}, we present the Rademacher
function and its generalizations, adapted to periodic and
pseudo-Anosov homeomorphisms. In Section~\ref{section: construction}
we construct the desired Reeb vector field $R$ which closely hews to
the suspension lamination.  This section is the technical heart of
the paper, and is unfortunately rather involved.
Section~\ref{section: genericity} is devoted to some discussions on
perturbing the contact form to make it nondegenerate. Then in
Sections~\ref{disks} and ~\ref{cylinders}, we give restrictions on
the holomorphic disks and cylinders.  In particular, we prove
Theorem~\ref{theorem: nodisks}, which states that for any $N\gg 0$
there is a contact $1$-form $\alpha$ for a contact structure which
is supported by an open book with pseudo-Anosov monodromy and
fractional Dehn twist coefficient $c>{1\over n}$, so that none of
the closed orbits $\gamma$ of action $\leq N$ are positive
asymptotic limits of (holomorphic) finite energy planes $\tilde u$.
The actual calculation of contact homology with such contact
$1$-forms will involve direct limits, discussed in
Section~\ref{section: direct limits}. We then prove
Theorem~\ref{main}(1) in Section~\ref{subsection: verification of
exhaustive}. Finally, we discuss the growth rate of periodic points
of a pseudo-Anosov homeomorphism and use it to conclude the proofs
of Theorem~\ref{main}(2) and Corollary~\ref{cor: infinitely many
simple orbits} in Section~\ref{section: nonvanishing}.

\section{Contact homology}\label{section:contact homology}

In this section we briefly describe the full contact homology, its
linearizations, and Morse-Bott theory. Contact homology theory is
part of the symplectic field theory of
Eliashberg-Givental-Hofer~\cite{EGH}. For a readable account, see
Bourgeois' lecture notes~\cite{Bo3}.

\s\n {\bf Disclosure.} The full details of contact homology have not
yet appeared.  In particular, the gluing argument and, more
importantly, the treatment of multiply-covered orbits are not
written anywhere. However, various portions of the theory are
available. For the asymptotics, refer to Hofer-Wysocki-Zehnder
~\cite{HWZ1}. Compactness was explained in \cite{BEHWZ}. The
Fredholm theory and transversality (for non-multiply-covered curves)
were treated by Dragnev~\cite{Dr}. For the Morse-Bott approach,
refer to \cite{Bo1}. Examples of contact homology calculations were
done by Bourgeois-Colin~\cite{BC}, Ustilovsky~\cite{U} and
Yau~\cite{Y1}.

\subsection{Definitions} \label{subsection: definitions}

Let $(M,\xi)$ be a contact manifold, $\alpha$ a contact form for
$\xi$, and $R=R_\alpha$ the corresponding Reeb vector field, i.e.,
$i_R d\alpha=0$ and $i_R \alpha=1$. Consider the symplectization
$(\R\times M,d(e^t\alpha))$, where $t$ is the coordinate for $\R$.
We will restrict our attention to almost complex structures $J$ on
$\R\times M$ which are {\em adapted to the symplectization}: If we
write $T_{(t,x)} (\R\times M)=\R{\bdry\over \bdry t} \oplus \R
R\oplus \xi$, then $J$ maps $\xi$ to itself and sends ${\bdry\over
\bdry t}\mapsto R$ and $R\mapsto -{\bdry\over \bdry t}$.

Let $\gamma$ be a closed orbit of $R$ with period $T$. The closed
orbit $\gamma$ is {\em nondegenerate} if the derivative
$\xi_{\gamma(0)}\rightarrow \xi_{\gamma(T)}$ of the first return map
does not have $1$ as an eigenvalue.   A Reeb vector field $R$ is
said to be {\em nondegenerate} if all its closed orbits $\gamma$ are
nondegenerate. Suppose $\alpha$ is a contact 1-form for which $R$ is
nondegenerate.

A closed orbit is said to be {\em good} if it does not cover a
simple orbit $\gamma$ an even number of times, where the first
return map $\xi_{\gamma(0)}\rightarrow \xi_{\gamma(T)}$ has an odd
number of eigenvalues in the interval $(-1,0)$. Here $T$ is the
period of the orbit $\gamma$. Let $\mathcal{P}=\mathcal{P}_\alpha$
be the collection of {\em good} closed orbits of $R=R_\alpha$. We
emphasize that $\mathcal{P}$ includes multiple covers of simple
periodic orbits, as long as they are good.

In dimension three, a closed orbit $\gamma$ has {\em even parity}
(resp.\ {\em odd parity}) if the derivative of the first return map
is of hyperbolic type with positive eigenvalues (resp.\ is either of
hyperbolic type with negative eigenvalues or of elliptic type). The
{\em Conley-Zehnder index} is a lift of the parity from $\Z/2\Z$ to
$\Z$. If $\gamma$ is a contractible periodic orbit which bounds a
disk $D$, then we trivialize $\xi|_D$ and define the Conley-Zehnder
index $\mu(\gamma,D)$ to be the Conley-Zehnder index of the path of
symplectic maps $\{d\phi_t: \xi_{\gamma(0)}\rightarrow
\xi_{\gamma(t)}, t\in[0,T]\}$ with respect to this trivialization,
where $\phi_t$ is the time $t$ flow of the Reeb vector field $R$. In
our cases of interest, $\pi_2(M)=0$, so $\mu(\gamma)$ is independent
of the choice of $D$. We will also sometimes write
$|\gamma|=\mu(\gamma)-1$.  If $\gamma$,
$\gamma'_1,\dots,\gamma'_m\in \mathcal{P}$ and
$[\gamma]=[\gamma'_1]+\dots+[\gamma'_m]\in H_1(M;\Z)$, then let $Z$
be a surface whose boundary is $\gamma-\gamma'_1-\dots-\gamma'_m$.
Trivialize $\xi|_Z$ and define the {\em Conley-Zehnder index
$\mu_{[Z]}(\gamma,\gamma'_1,\dots,\gamma'_m)$ with respect to the
relative homology class $[Z]\in H_2(M,\gamma\cup(\cup_i\gamma_i'))$}
to be the Conley-Zehnder index of $\gamma$ minus the sum from $i=1$
to $m$ of the Conley-Zehnder indices of $\gamma_i'$, all calculated
with respect to the trivialization on $Z$.

Now, we fix a point $m_\gamma$, called a {\em marker}, on each
simple periodic orbit $\gamma$.  Also, an {\it asymptotic marker} at
$z\in S^2$ is a ray $r$ originating from $z$.

Define $Hol_{[Z]}(J,\gamma,\gamma_1', \dots ,\gamma_m')$ to be the
set of all holomorphic maps
$$\tilde u=(a,u):(\Sigma=S^2-\{x,y_1,\dots,y_m\},j)
\rightarrow (\R\times M,J),$$ together with asymptotic markers $r$
at $x$ and $r_i$ at $y_i$, $i=1,\dots, m$, subject to the following:
\begin{itemize}
\item $\lim_{\rho \rightarrow 0} u(\rho ,\theta )= \gamma (\theta )$
near $x$;
\item $\lim_{\rho \rightarrow 0} u(\rho ,\theta )= \gamma_i' (\theta )$ near $y_1', \dots,
y_m'$;
\item the limit of $u$ as $\rho \rightarrow 0$ along $r$ is
$m_{\gamma}$;
\item the limit of $u$ as $\rho \rightarrow 0$ along $r_i$ is
$m_{\gamma_i'}$;
\item $\lim_{\rho \rightarrow 0} a(\rho ,\theta )= + \infty$ near
$x$;
\item $\lim_{\rho \rightarrow 0} a(\rho ,\theta )= - \infty$
near $y_1', \dots, y_m'$.
\end{itemize}
Here, $x,y_1,\dots,y_m\in S^2$, $j$ is a complex structure on
$\Sigma$, $u$ is in the class $[Z]$, we are using polar coordinates
$(\rho,\theta)$ near each puncture, and $\gamma (\theta )$ and
$\gamma_i' (\theta )$, $i=1,\dots, m$, refer to some parametrization
of the trajectories $\gamma$ and $\gamma_i'$.  The convergence for
$u(\rho ,\theta )$ and $a(\rho,\theta )$ is in the $C^0$-topology.
In the current situation, the punctures $x, y_1 ,\dots ,y_m$, the
complex structure $j$, and the asymptotic markers $r,r_1 ,\dots
,r_m$ are allowed to vary, while the markers $m_\gamma$ stay fixed.

Next, two curves $\tilde u :(\Sigma=S^2-\{x,y_1,\dots,y_m\},j)
\rightarrow (\R\times M,J)$ and $\tilde u'
:(\Sigma=S^2-\{x',y_1',\dots,y_m'\},j') \rightarrow (\R\times M,J)$
in $Hol_{[Z]}(J,\gamma,\gamma_1', \dots ,\gamma_m')$ are equivalent
if $\tilde u'=\tilde u\circ H$, where $H$ is a biholomorphism of
$\Sigma$ which takes the asymptotically marked punctures
$$((x',r'),(y_1' ,r_1' ),\dots (y_m' ,r_m' )) \rightarrow
 ((x,r),(y_1 ,r_1 ),\dots ,(y_m ,r_m )) .$$

We define the moduli space $\mathcal{M}_{[Z]}(J,\gamma,\gamma_1',
\dots, \gamma_m')$ to be the quotient of
$Hol_{[Z]}(J,\gamma,\gamma_1', \dots ,\gamma_m')$ under the
equivalence relation.  The space
$\mathcal{M}_{[Z]}(J,\gamma,\gamma_1', \dots, \gamma_m')$ supports
an $\R$-action (in the target), obtained by translating a curve
along the $\R$-direction of $\R \times M$. Assuming sufficient
transversality,
$\mathcal{M}_{[Z]}(J,\gamma,\gamma_1',\dots,\gamma_m')/\R$ is
endowed with the structure of a weighted branched manifold with
rational weights.
For that one can use the Kuranishi perturbation theory of
Fukaya-Ono~\cite{FO} or the multi-valued perturbation of Liu and
Tian~\cite{LT}; see also McDuff~\cite{McD}. In that case,
$\mathcal{M}_{[Z]}(J,\gamma,\gamma_1', \dots ,\gamma_m')/\R$ is a
union of manifolds with corners along a codimension one branching
locus, each piece having the expected dimension
$\mu_{[Z]}(\gamma,\gamma_1',\dots, \gamma_m')-(1-m)-1$. When this
dimension is $0$, we find a finite collection of points, according
to the ``Gromov compactness theorem'' due to \cite{BEHWZ}.

We now define the {\em full contact homology} groups
$FHC(M,\alpha,J)$. The contact homology groups are necessarily
defined over $\Q$, since we must treat multiply covered orbits. The
chain group is the supercommutative $\Q$-algebra $\A=\A(\alpha,J)$
with unit, which is freely generated by the elements of
$\mathcal{P}$. Here supercommutative means that $\gamma_1$ and
$\gamma_2$ commute if one of them has odd parity (and hence even
degree $|\cdot|$) and anticommute otherwise. Now the boundary map
$\bdry: \A\rightarrow \A$ is given on elements
$\gamma\in\mathcal{P}$ by:
$$\bdry\gamma= \sum
{n_{\gamma,\gamma_1', \dots, \gamma_m'}\over (i_1)! \dots
(i_l)!\kappa(\gamma_1')\dots \kappa (\gamma_m')} ~~\gamma_1' \dots
\gamma_m',$$ where the sum is over all unordered tuples
$\overline{\gamma'}=(\gamma_1',\dots,\gamma_m')$ of orbits of
$\mathcal{P}$ and homology classes $[Z]\in
H_2(M,\gamma\cup(\cup_i\gamma_i'))$ so that, for any given ordering
$\gamma_1' ,\dots , \gamma_m'$ of $\overline{\gamma'}$, the expected
dimension of the moduli space $\mathcal{M}_{[Z]}(J,\gamma,\gamma_1',
\dots ,\gamma_m')/\R$ is zero. Here $\kappa(\gamma)$ is the
multiplicity of $\gamma$. The integers $i_1,\dots ,i_l$ denote the
number of occurrences of each orbit $\gamma_i'$  in the list
$\gamma_1',\dots, \gamma_m'$. Also, we denote by
$n_{\gamma,\gamma_1', \dots,\gamma_m'}$ the signed weighted count of
points in $\mathcal{M}_{[Z]}(\gamma,\gamma_1', \dots,
\gamma_m')/\R$, for the corresponding ordering  of
$\overline{\gamma'}$, following a coherent orientation scheme as
given in \cite{EGH}. This definition does not depend on the ordering
of $\overline{\gamma'}$, since if we permute $\gamma_i'$ and
$\gamma_{i+1}'$, the coefficient $n_{\gamma ,\gamma_1' ,\dots
,\gamma_m'}$ is multiplied by $(-1)^{| \gamma_i' ||\gamma_{i+1}'
|}$, which is annihilated by the sign coming from the
supercommutativity of $\A$. If $\gamma$, $\gamma_1', \dots,
\gamma_m'$ are multiply covered, then each non-multiply-covered
holomorphic curve $\tilde u\in \mathcal{M}_{[Z]}(\gamma,\gamma_1',
\dots, \gamma_m')/\R$ contributes $\pm \kappa(\gamma)
\kappa(\gamma_1')\dots \kappa (\gamma_m')$ to $n_{\gamma,\gamma_1',
\dots, \gamma_m'}$. This is due to the fact that, for the puncture
$x$ (resp.\ $y_i$), there are $\kappa(\gamma)$ (resp.\ $\kappa
(\gamma_i' )$) possible positions for the asymptotic marker $r$
(resp.\ $r_i$). If $\tilde u$ is a cover of a somewhere injective
holomorphic curve, then it is counted as $\pm \frac{1}{k}
(\kappa(\gamma) \kappa(\gamma_1') \dots \kappa (\gamma_m'))$, where
$k$ is the number of automorphisms of the cover, since this group of
automorphisms acts freely  on the set of asymptotic markers and thus
allows to identify different positions.
The coefficient $i_1 !\dots i_l !$ takes into account the following
overcounting: if, for example, $y_1 ,\dots ,y_{i_1}$ go to
$\gamma_1'$, then, for any permutation of these indices, the
corresponding permutation of the punctures will give rise to
different maps in $\mathcal{M}_{[Z]}(\gamma,\gamma_1', \dots,
\gamma_m')/\R$. The definition of $\bdry$ is then extended to all of
$\A$ using the graded Leibniz rule.

\begin{theorem}(Eliashberg-Givental-Hofer)
\begin{enumerate}
\item$\bdry^2 =0$,
so that $(\A(\alpha,J), \bdry)$ is a differential graded algebra.
\item $FHC(M,\alpha,J)=H_* (\A(\alpha,J),\bdry )$ does
not depend on the choice of the contact form $\alpha$ for $\xi$, the
complex structure $J$ and the multi-valued perturbation.
\end{enumerate}
\end{theorem}

\s The action $A_\alpha(\gamma)=\int_\gamma\alpha$ of a closed orbit
$\gamma$ with respect to the $1$-form $\alpha$ gives rise to a
filtration, which we call the {\em action filtration}. Define the
action of $\gamma_1 \dots \gamma_m$ to be $A_\alpha (\gamma_1 \dots
\gamma_m )=\sum_{i=1}^m A_\alpha (\gamma_i )$. The boundary map is
action-decreasing, since every nontrivial holomorphic curve has
positive $d\alpha$-energy. There is a second filtration which comes
from an open book decomposition, which we call the {\em open book
filtration}, and is given by the number of times an orbit intersects
a given page.  This will be described in more detail in
Section~\ref{cylinders}.

\subsection{Linearized contact homology}
In this subsection we discuss augmentations, as well linearizations
of contact homology induced by the augmentations. We have learned
what is written here from Tobias Ekholm~\cite{Ek}. Details of the
assertions are to appear in \cite{BEE}. The notion of an
augmentation first appeared in \cite{Chk}, in the context of
Legendrian contact homology.

Let $(\A=\A(\alpha,J),\bdry)$ be the chain group for the full
contact homology as defined above. An {\em augmentation} for $\A$ is
a $\Q$-algebra homomorphism $\varepsilon: \A\rightarrow \Q$ which is
also a chain map. (Here we are assuming that the boundary map
$\bdry'$ for $\Q$ satisfies $\bdry' a=0$ for all $a\in \Q$. This
means that $\varepsilon\bdry=0$.)  In this paper, we will assume
that $\varepsilon(a)=0$ if $a$ not contractible or if $a$ is
contractible but $|a|\not=0$.  (Recall that $\pi_2(M)=0$ in this
paper.)  Let $\mbox{Aug}(\A,\bdry)$ denote the set of augmentations
of $(\A,\bdry)$.

An augmentation $\varepsilon$ for $(\A,\bdry)$ induces a ``change of
coordinates'' $a\mapsto \overline{a}= a-\varepsilon(a)$ of $\A$,
where $a\in \mathcal{P}$. Then $\bdry \overline{a}$ has the property
that it does not have any constant terms when expressed in terms of
sums of words in $\overline{a}_i$. (Proof by example: Suppose $\bdry
a= 1+a_1+a_2a_3$. Then $$\bdry \overline{a}=1+(\overline
a_1+\varepsilon(a_1))+(\overline a_2+\varepsilon(a_2))(\overline
a_3+\varepsilon(a_3))=\varepsilon(1+a_1+a_2a_3)+\mbox{h.o.}=\mbox{h.o.}$$
Here `h.o.' means higher order terms in the word length filtration.)
In other words, with respect to the new generators $\overline{a}$,
$\bdry$ is nondecreasing with respect to the word length filtration,
i.e., $\bdry = \bdry_1+\bdry_2+\dots$, where $\bdry_j$ is the part
of the boundary map which counts words of length $j$ in the
$\overline{a}_i$'s. Therefore it is possible to define the {\em
linearized contact homology} group $HC^\varepsilon(M,\alpha,J)$ with
respect to $\varepsilon$ to be the homology of $(\A_1,\bdry_1)$,
where $\A_1$ is the $\Q$-vector space generated by the $\overline
a_i$ and $a_i\in \mathcal{P}$.

\s\n {\bf Example 1: cylindrical contact homology.} When $\bdry a$
does not have a constant term for all $a\in \mathcal{P}$, then it
admits the {\em trivial augmentation} $\varepsilon$ which satisfies
$\varepsilon(1)=1$ and $\varepsilon(a)=0$ for all $a\in
\mathcal{P}$. The linearized contact homology with respect to the
trivial augmentation $\varepsilon$ is usually called {\em
cylindrical contact homology}, and will be denoted $HC(M,\alpha,J)$.
If we restrict to the class of nondegenerate Reeb vector fields
$R_\alpha$ with trivial augmentations, then $HC(M,\alpha,J)$ does
not depend on $\alpha$ (or on $J$) and will be written as
$HC(M,\xi=\ker \alpha)$.

\s
We make two remarks about cylindrical contact homology. First, the
trivial augmentation does not always exist.  Second, it is possible
to have finite energy planes which asymptotically limit to $a$ at
the positive end and still have $\bdry a$ without a constant term,
as long as the total signed count is zero.

\s\n {\bf Example 2: augmentations from cobordisms.} Suppose
$(X^4,\omega)$ is an exact symplectic cobordism with $(M,\alpha)$ at
the positive end and $(M',\alpha')$ at the negative end, and $J$ be
a compatible almost complex structure on $(X^4,\omega)$. If
$\A(M',\alpha',J|_{\ker \alpha'})$ admits an augmentation
$$\varepsilon':\A(M',\alpha',J|_{\ker \alpha'})\rightarrow \Q,$$ then
we can compose it with the chain map $$\Phi_{(X,J)}:
\A(M,\alpha,J|_{\ker \alpha})\rightarrow \A(M',\alpha',J|_{\ker
\alpha'})$$ to obtain the pullback augmentation
$\varepsilon=\Phi_{(X,J)}^*\varepsilon'=\varepsilon'\circ
\Phi_{(X,J)}$.  Moreover, we have an induced map
$$HC^\varepsilon(M,\alpha)\rightarrow
HC^{\varepsilon'}(M',\alpha')$$ between the linearized contact
homology groups.

\s Two augmentations
$\varepsilon_0,\varepsilon_1:(\A,\bdry)\rightarrow \Q$ are said to
be {\em homotopic} if there is a derivation $K:
(\A,\bdry)\rightarrow (\A,\bdry)$ of degree $1$ satisfying
$$\varepsilon_1=\varepsilon_0\circ e^{\bdry \circ K +K\circ \bdry}.$$

\begin{thm}[Bourgeois-Ekholm-Eliashberg] \label{thm: aug} $\mbox{}$
\be
\item If $\varepsilon_0,\varepsilon_1$ are homotopic augmentations of
$(\A(M,\alpha,J),\bdry)$, then
$$HC^{\varepsilon_0}(M,\alpha,J)\simeq
HC^{\varepsilon_1}(M,\alpha,J).$$
\item Given a $1$-parameter family of compatible almost complex
structures $J_t$, $t\in[0,1]$, on the exact symplectic cobordism
$(X^4,\omega)$ from $(M,\alpha)$ to $(M',\alpha')$ which agrees with
$J$ on $M$ and $J'$ on $M'$, and an augmentation $\varepsilon'$ on
$\A(M',\alpha',J')$, the pullback augmentations
$\varepsilon_0=\varepsilon'\circ \Phi_{(X,J_0)}$ and
$\varepsilon_1=\varepsilon'\circ \Phi_{(X,J_1)}$ are homotopic and
induce the same map $$HC^{\varepsilon_i}(M,\alpha,J)\rightarrow
HC^{\varepsilon'}(M',\alpha',J').$$
\item
The set
$$\{HC^\varepsilon(M,\alpha,J)~|~ \varepsilon\in
\mbox{\em Aug}(\A(M,\alpha,J),\bdry)\}$$ of linearized contact
homologies up to isomorphism is an invariant of
$(M,\xi=\ker\alpha)$. \ee
\end{thm}

\subsection{Morse-Bott theory} \label{subsection: Morse-Bott}

We briefly describe how to compute the cylindrical contact homology
of a degenerate contact form of Morse-Bott type. For more details,
we refer the reader to Bourgeois' thesis~\cite{Bo2}. Again, let
$\phi_t$ be the time $t$ flow of the Reeb vector field $R_\alpha$ of
$\alpha$.

A contact form $\alpha$ is of {\it Morse-Bott type} if:
\begin{enumerate}
\item the {\it action spectrum} $\sigma (\alpha )=\{ A_\alpha
(\gamma ) ~|~\gamma {\rm \; periodic \; orbit} \}$ is discrete;
\item the union $N_T$ of fixed points of $\phi_T$
is a closed submanifold of $M$;
\item the rank of $d\alpha \vert_{N_T}$
is locally constant and $T_p N_T =\ker (d\phi_T (p)-I)$.
\end{enumerate}
The submanifold $N_T$ is foliated by orbits of $R_\alpha$. In the
case where $\dim N_T =3$, the manifold $N_T$ is a Seifert fibered
space. The quotient space $S_T$ is thus an orbifold, whose
singularities with singularity groups $\Z/m\Z$ are the projections
of orbits of actions ${T\over m}$.

Choose a complex structure $J$ on $\xi$ which is invariant under the
$S^1$-action on $N_T$ induced by the flow $\phi_t$. Now, for each
$T$, pick a Morse function $f_T : S_T \rightarrow \R$ so that the
downward gradient trajectory of $f_T$ with respect to the metric
induced from $d\alpha (\cdot,J\cdot)$ (by quotienting out the
$S^1$-direction) is of Morse-Smale type. We also assume that, if
$S_T \subset S_{kT}$, then $f_{kT}$ extends the function $f_T$ so
that $f_{kT}$ has positive definite Hessian in the normal directions
to $S_T$.

Let $\gamma\in S_T$. We choose a trivialization of $\xi
\vert_{\gamma}$. As before, define the Conley-Zehnder index
$\mu(\gamma)$ to be the Conley-Zehnder index of the path
$\{d\phi_t(\gamma(0)): \xi_{\gamma(0)}\rightarrow \xi_{\gamma(T)},
t\in[0,T]\}$ with respect to the trivialization, using the
Robbin-Salamon definition~\cite{RS}. (Note that the value $1$
belongs to the spectrum of the map $d\phi_T(\gamma(0))$, with
eigenspace isomorphic to the tangent space of $S_T$.)

If $\gamma\in S_T$ is a critical point of $f_T$, then define the
grading $|\gamma|$ of $\gamma$ as:
$$|\gamma| = \mu (\gamma)-{1\over 2} \dim S_T +\mbox{index}_\gamma
(f_T)-1.$$ The parity of $|\gamma|$ does not depend on the choice of
framing of $\xi$ along $\gamma$. Also, if $\gamma\in S_T$, then a
choice of framing for $\xi$ along $\gamma$ induces a framing along
any $\gamma' \in S_T$, by isotoping through fibers. The index $\mu
(\gamma)$ then does not depend on $\gamma\in S_T$ for this
particular family of framings. If $\gamma\in S_T$, then let
$m\gamma$ denote the $m$-fold cover of $\gamma$ in $S_{mT}$. A
critical point $\gamma$ of $f_T$ is {\em bad} if it is an even
multiple $2k\gamma'$ of a point $\gamma'$ whose parity differs from
the one of $\gamma$, and is {\em good} otherwise.

Let $MBC(\alpha,J)$ be the free $\Q$-vector space generated by the
good critical points of $f_T$, for all $T\in \sigma (\alpha)$. We
now briefly describe the differential $\partial$ on $MBC$. Let
$\gamma^+$ and $\gamma^-$ be good critical points of $\sqcup_T f_T$.
The coefficient $\langle \bdry \gamma^+,\gamma^-\rangle$ of
$\gamma^-$ in the differential $\partial \gamma^+$ is a signed count
of points of $0$-dimensional moduli spaces of generalized
holomorphic cylinders from $\gamma^+$ to $\gamma^-$.  A {\em
generalized holomorphic cylinder} from $\gamma^+$ to $\gamma^-$ is a
finite collection of $J$-holomorphic cylinders $\{ C_1 ,\dots ,C_k
\}$, together with downward gradient trajectories
$\{a_0,a_1,\dots,a_{k+1}\}$ of $f_T$ on $S_T$, satisfying the
following:
\begin{itemize}
\item The holomorphic cylinder $C_i$, $i=1,\dots,k$, is
asymptotic to $\gamma^+_i$ at $+\infty$ and $\gamma^-_i$ at
$-\infty$;
\item $\gamma^+=\gamma^-_0$ and $\gamma^-=\gamma^+_{k+1}$;
\item The orbits $\gamma^-_i$ and $\gamma^+_{i+1}$, $i=0,\dots, k$, lie on the same component of
$S_T$, and $a_i$ connects $\gamma_i^-$ to $\gamma_{i+1}^+$;
\end{itemize}
The map $\partial$ is extended linearly to all of $MBC(\alpha,J)$.
The main theorem of Bourgeois' thesis is the following:

\begin{theorem}
If no orbit of $R_\alpha$ is the asymptotic limit of a finite energy
plane, then $(MBC(\alpha,J),\partial )$ is a chain complex and its
homology is isomorphic to $HC(M,\xi)$.
\end{theorem}

\s\n {\bf Example.} If $M$ is fibered by Reeb periodic orbits of
action $T$, then $S_T$ is a smooth surface. Since all the orbits in
$N_T =M$ have the same action $T$, the only generalized cylinders
between orbits in $S_T$ are the gradient flow lines of $f_T$. Thus,
if $\gamma,\gamma' \in S_T$, then $\langle\partial
\gamma,\gamma'\rangle$ will be the same as that given by the Morse
differential.

\s\n {\bf Example.} If $M$ is a Seifert fibered space with singular
fibers of orders $s_1,s_2,\dots,s_n$ so that all its fibers are Reeb
orbits, then all regular orbits have the same action $T$, and the
singular orbits have actions ${T\over s_1} ,\dots {T\over s_n}$. If
$\gamma_i$ denotes the singular fiber of action ${T\over s_i}$, then
$S_{T/s_i}$ is $\gamma_i$.  Moreover, $\mu(\gamma_i )$ is odd since
the regular fibers rotate about the singular fiber. Hence
$|\gamma_i|=\mu (\gamma_i )-1$ is even.

\s The above examples will be explored in more detail in
Section~\ref{section:periodic}.

\section{The periodic case}\label{section:periodic}

Suppose the contact 3-manifold $(M,\xi)$ admits an open book
decomposition $(S,h)$ with periodic monodromy.  Let $c_i$ be the
fractional Dehn twist coefficient of the $i$th boundary component
and $\psi$ be the periodic representative of $h$.

\begin{thm}\label{thm:invariant}
If all the $c_i$ are positive, then $(M,\xi)$ is an $S^1$-invariant
contact structure which is transverse to the $S^1$-fibers.
\end{thm}

A transverse contact structure $\xi$ (= transverse to the fibers) on
a Seifert fibered space $M$ with base $B$ and projection map
$\pi:M\rightarrow B$ is said to be {\em $S^1$-invariant} if there is
a Reeb vector field $R$ of $\xi$ so that (i) each fiber
$\pi^{-1}(p)$ is an orbit of $R$ and (ii) a neighborhood of a
singular fiber is a $\Z/m\Z$-quotient of $S^1\times D^2$ with the
standard contact form $dt+\beta$, where $t$ is the coordinate for
$S^1$, $\beta$ is rotationally invariant and independent of $t$,
$d\beta$ is an area form on $D^2$, and the Reeb vector field is
${\bdry\over \bdry t}$.

\begin{proof}
Suppose $(S,h)$ is periodic. Let $\beta$ be a $1$-form on $S$
satisfying $d\beta>0$. We additionally require that, along each
component of $\bdry S$, $\beta={C\over 2\pi} d\phi$, where $\phi$ is
the angular coordinate of the boundary component equipped with the
boundary orientation, and $C>0$ is a constant. If the periodic
representative $\psi$ of $h$ has order $n$ (here $c_i={k_i\over n}$,
where $k_i$ and $n$ are relatively prime), we average $\beta$ by
taking
$$\overline\beta={1\over n}\sum_{i=0}^{n-1} (\psi^i)^*\beta.$$

Consider the $[0,1]$-invariant contact $1$-form
$\alpha=dt+\overline\beta$ on $S\times[0,1]$.  Here $t$ is the
$[0,1]$-coordinate. By construction, $\alpha$ descends to a contact
form on $N=(S\times[0,1])/(x,1)\sim (\psi(x),0)$ and the
corresponding Reeb vector field $R$ on $S\times[0,1]$ is
${\bdry\over \bdry t}$. The manifold $N$ is a Seifert fibered space
whose fibers are closed orbits of $R$. Observe that a nonsingular
fiber intersects $S\times\{0\}$ at $n$ points. Although $R$ is
probably the most natural Reeb vector field, it is highly
degenerate, i.e., for each $p\in S$ there is a corresponding closed
orbit $\{p\}\times S^1$.  Hence we are in the Morse-Bott situation.

We then extend the contact 1-form to the neighborhood $N(K)\simeq
S^1\times D^2=\R/\Z\times D^2$ of each binding component $K$. Let us
use cylindrical coordinates $(z,(r,\theta))$ on $N(K)$, so that the
pages restrict to $\theta=const$. From the construction of $\alpha$
on $N$, $\bdry N(K)$ is (i) linearly foliated by the Reeb vector
field $R$ of slope $c_i$; and (ii) linearly foliated by the
characteristic foliation of $\xi$ of slope $-{1\over C}$, $C>0$.
Here we are using coordinates $({\theta\over 2\pi},z)$ to identify
$\bdry N(K)\simeq \R^2/\Z^2$.

Start with $[0,1]\times D^2$ with coordinates $(z,(r,\theta))$ and
contact form $\alpha=dz+ {1\over 2} r^2d\theta$.  Here ${1\over
2}r^2d\theta$ is the primitive of an area form for $D^2$ and is
invariant under rotation by $\theta=\theta_0$.  Moreover,
$R={\bdry\over \bdry z}$. Now glue $\{1\}\times D^2$ to $\{0\}\times
D^2$ via a diffeomorphism $\phi$ which sends $(r,\theta)\mapsto
(r,\theta+\theta_0)$ for some constant $\theta_0$. The Reeb vector
field $R$ will then have slope ${2\pi\over \theta_0}$; pick
$\theta_0$ so that $c_i={2\pi \over \theta_0}$.  Furthermore, if we
adjust the size of the disk $D^2$ to have a suitable radius, then
the characteristic foliation on $\bdry (S^1\times D^2)$ would have
slope $-{1\over C}$.  (For another, more or less equivalent,
construction, see Section~\ref{subsub: extension to binding}.)

By taking the $n$-fold cover of $S^1\times D^2$ we obtain a
transverse contact structure on $S^1\times D^2$ which is fibered by
Reeb vector fields and which does not have any singular fibers. This
completes the proof of Theorem~\ref{thm:invariant}.
\end{proof}

\begin{thm}\label{thm:right-veering is exact}
If $(S,h)$ has periodic monodromy, then $(M,\xi)$ is tight if and
only if $h$ is right-veering.  Moreover, the tight contact
structures are Stein fillable.
\end{thm}

Let $M$ be a Seifert fibered space over an oriented closed surface
of genus $g$ and with $r$ singular fibers, whose Seifert invariants
are ${\beta_1\over \alpha_1},\dots, {\beta_r\over \alpha_r}$.  Then
the {\em Euler number} $e(M)=\sum_{i=1}^r {\beta_i\over \alpha_i}$.

Some of the contact structures will be (universally) tight contact
structures on lens spaces, which we know are Stein fillable.

\begin{proof}
By Theorem~\ref{ot}, $h$ is right-veering if and only if all the
$c_i$ are nonnegative. Moreover, if any coefficient $c_i$ is
negative, then $(M,\xi)$ is overtwisted.  If some $c_i=0$, then $h$
must be the identity, since $\psi$ is periodic. In this case
$(M,\xi)$ is the standard Stein fillable contact structure on
$\#(S^1\times S^2)$.

Hence it remains to consider the case where all $c_i>0$.  According
to Theorem~\ref{thm:invariant}, $(M,\xi)$ is $S^1$-invariant.
According to a result of Lisca and Mati\'c~\cite{LM}, a Seifert
fibered space $M$ carries an $S^1$-invariant transverse contact
structure if and only if the Euler number $e(M)<0$. It is not hard
to see that this $S^1$-invariant contact structure is symplectically
fillable and universally tight.

Neumann and Raymond~\cite[Corollary~5.3]{NR} have shown that, if
$e(M)<0$, then $M$ is the link of an isolated surface singularity
with a holomorphic $\C^*$-action.  Hence $M$ is the oriented,
strictly pseudoconvex boundary of a compact complex surface (with a
singularity).  Let $\xi'$ be the complex tangencies $TM\cap J(TM)$.
The holomorphic $\C^*$-action on the complex surface becomes an
$S^1$-action on $M$.  The vector field $X$ on $M$ generated by the
$S^1$-action is transverse to $\xi'$, since $JX$ is transverse to
$M$.  Hence $X$ is a Reeb vector field for $\xi'$, and $\xi'$ is an
$S^1$-invariant transverse contact structure. Now, by Bogomolov
\cite{Bog} (also Bogomolov-de Oliveira~\cite[Theorem ($2'$)]{Bd}),
$(M,\xi')$ is also a strictly pseudoconvex boundary of a smooth
Stein surface.

It remains to identify the $S^1$-invariant transverse contact
structures $\xi$ and $\xi'$ on $M$.  By Lemma~\ref{lemma:unique},
there is a unique $S^1$-invariant horizontal contact structure on
$M$ up to isotopy, once the fibering is fixed.  By
Hatcher~\cite[Theorem~4.3]{Hat}, Seifert fiberings of closed
orientable Seifert fibered spaces over {\em orientable bases} are
unique up to isomorphism, with the exception of $S^3$, $S^1\times
S^2$, and lens spaces.  (The other items on Hatcher's list consist
of $M$ with boundary or identifications with Seifert fibered spaces
over nonorientable bases.)  All the tight contact structures on
$S^3$, $S^1\times S^2$, and lens spaces are Stein fillable.
\end{proof}

\begin{lemma}\label{lemma:unique}
For any Seifert fibered space $M$ with a fixed fibering, any two
$S^1$-invariant transverse contact structures are isotopic.
\end{lemma}

\begin{proof}
Let $\pi:M\rightarrow B$ be a fixed fibering and let $\xi$, $\xi'$
be $S^1$-invariant transverse contact structures on $M$. Given any
point $p$ in $B$ ($p$ may be a singular fiber), there exist small
neighborhoods $U, U'\subset B$ of $p$ so that the holonomy of the
characteristic foliation of $\xi$ on $\pi^{-1}(\bdry U)$ and $\xi'$
on $\pi^{-1}(\bdry U')$ agree.  By taking a diffeomorphism of $U$ to
$U'$, we may assume that $U=U'$.  Writing $\pi^{-1}(U)=S^1\times U$
with fibers $S^1\times\{pt\}$ and coordinates $(t,(x,y))$, we may
modify $t\mapsto t+f(x,y)$ in a neighborhood of $\bdry U$ so that
$\xi=\xi'$ along $S^1\times \bdry U$.  The case of a singular fiber
is similar.

The rest of the argument is similar to that which appears in
Giroux~\cite{Gi2}. We now have $S^1$-invariant transverse contact
structures $\xi$ and $\xi'$ on $S^1\times B'$, where $B'$ is a
surface with boundary and $\xi=\xi'$ on $S^1\times \bdry B'$.  We
may then write $\xi=\ker(dt+\beta)$ and $\xi'=\ker(dt+\beta')$,
where $\beta=\beta'$ on $S^1\times \bdry B'$.   Here $\beta$ and
$\beta'$ are 1-forms on $B'$ which are independent of $t$.  We
simply interpolate by taking $\alpha_s=dt+(1-s)\beta+s\beta'$. Since
$d\beta$ and $d\beta'$ are area forms on $B'$, $\alpha_s$ is a
contact form.
\end{proof}

\begin{thm}\label{thm:cylindrical}
If all the $c_i>0$, then the cylindrical contact homology is
well-defined.  If the universal cover of $M$ is $\R^3$, then the
cylindrical contact homology is nontrivial.
\end{thm}

\begin{proof}
If $M\rightarrow B$ is the Seifert fibration by orbits of the Reeb
vector field, then view the closed, oriented base $B$ as an
orbifold.  Since we are disallowing the case when the page $S=D^2$,
$B$ is always a {\em good orbifold} in the sense of
Scott~\cite[Theorem~2.3]{Sc}, and admits a finite covering which is
a closed surface with no orbifold singularities.  Now, if we view a
Seifert fibered space $M$ as an orbifold circle bundle, then the
pullback bundle of an orbifold cover $\pi: B'\rightarrow B$ is a
genuine covering space $M'$ of $M$ \cite[Lemma~3.1]{Sc}.  Taking a
closed surface $B'$ with no singularities, we see that $M'$ is a
circle bundle over $B'$. The Euler number $e(M)$ lifts to the Euler
number $e(M')$, which is $e(M)$ times the degree of the cover. Since
$e(M)<0$, it follows that $e(M')<0$.

Suppose first that $B'\simeq S^2$.  Then the universal cover $\wt M$
of $M$ must be $S^3$ and the Reeb fibration becomes the Hopf
fibration. In particular, there can be no contractible periodic
orbit $\gamma$ of $M$ with Conley-Zehnder index $\mu(\gamma)=2$,
since there is none in $\wt M$. On the other hand, if $g(B')\geq 1$,
then $\wt M\simeq \R^3$.  Since every fiber of $M$ lifts to $\R$,
there are no contractible periodic orbits $\gamma$ of $M$. In either
case, the cylindrical contact homology is well-defined.

Next we prove that if $\gamma'$ winds $m'$ times around a regular
fiber and $\gamma''$ winds $m''$ times around a regular fiber, then
there are no holomorphic cylinders in the symplectization from
$\gamma'$ to $\gamma''$, provided $m'\not=m''$.  If there is such a
holomorphic cylinder in $\R\times M$, then there would be a cylinder
from $\gamma'$ to $\gamma''$ in $M$. Since the cylinder has the
homotopy type of, say, $\gamma'$, it can be lifted to $M'$ since
regular fibers are not expanded under bundle pullbacks.  Now
$e(M')<0$, so the homology class of a regular fiber is a generator
of $\Z/|e(M')|\Z\subset H_1(M;\Z)$. If we take an $n$-fold cover
$B''$ of $B'$, then the pullback $M''$ satisfies $e(M'')=n\cdot
e(M')$, and we can distinguish $\gamma'$ from $\gamma''$
homologically, provided $n$ is sufficiently large. Analogous
statements can also be made for multiple covers of singular fibers,
by simply viewing a singular fiber as a suitable fraction of a
regular fiber $F$.

Now suppose that $\wt M=\R^3$.  First suppose that $B$ does not have
any orbifold singular points.  Then the orbits of smallest action
are simple orbits around the $S^1$-fibers, parametrized by the base
$B$. Therefore, the portion of $HC(M,\xi)$ with the least action is
$H_*(B;\Q)$, by Bourgeois' Morse-Bott theory sketched in
Section~\ref{subsection: Morse-Bott}. Next suppose that the orbifold
singularities of $S/\psi$ have orders $s_1,\dots,s_m$, arranged in
nonincreasing order (these are the ``interior'' singularities).  The
orbifold singularities coming from the binding all have order $n$,
where $c_i={k_i\over n}$ as before. Hence the simple Reeb orbits
corresponding to the singular fibers are ${1\over s_1},\dots,{1\over
s_m}, {1\over n}$ of a regular fiber $F$.  They are all elliptic
orbits and have even parity, so there are no holomorphic cylinders
amongst them.  Hence simple orbits around the singular fibers
correspond to nontrivial classes in $HC(M,\xi)$.
\end{proof}

\begin{rmk}
The techniques involved in proving Theorem~\ref{thm:cylindrical} are
sufficient to completely determine the cylindrical contact homology
groups of the relevant contact structures.
\end{rmk}

\section{Rademacher functions} \label{section: Rademacher}

We now define the Rademacher function and its generalizations.  The
usual Rademacher function is a beautiful function on the Farey
tessellation, which admits an interpretation as a bounded cohomology
class in $H_b^2(SL(2,\Z))$.  For more details, see
\cite{BG,GG1,GG2}. The (generalized) Rademacher functions are used
to measure certain types of ``lengths" of arcs in the universal
cover $\wt{S}$ of a compact hyperbolic surface $S$ with geodesic
boundary.

In this section we do not make any assumptions about the number of
boundary components of $S$.

\subsection{The usual Rademacher function}
Let $S$ be a compact hyperbolic surface with geodesic boundary. We
first triangulate $S$ with geodesic arcs which begin and end on
$\bdry S$. Here the boundary of each ``triangle'' consists of three
geodesic arcs (which may happen to coincide), together with subarcs
of $\bdry S$. (Henceforth we omit the quotes when referring to {\em
triangles}. In general, when we refer to an {\em $n$-gon}, we will
not be counting the subarcs of $\bdry S$.) Let $\tau$ the set of
geodesic arcs of the triangulation that are not subarcs of $\bdry
S$.  Also let $\tilde \tau=\pi^{-1}(\tau)$, where
$\pi:\wt{S}\rightarrow S$ is the universal covering map.

The {\em Rademacher function} $\Phi$ is a function
$\tilde\tau\rightarrow \Z$, defined as follows:  Pick a reference
arc $a\in \tilde\tau$, and set $\Phi(a)=0$.  Given $a'\in \tilde
\tau$, take an {\em oriented} geodesic arc $\delta$ in $\wt{S}$ from
$a$ to $a'$.  Then $\Phi(a')$ is the number of right turns taken
minus the number of left turns taken along the path $\delta$ from
$a$ to $a'$. In other words, if $a',a'',a'''\in \tilde\tau$ form a
triangle in $\wt{S}$, where the edges are in counterclockwise order
around the triangle, and we have inductively defined $\Phi(a')$ but
not $\Phi(a'')$ and $\Phi(a''')$, then we set $\Phi(a'')=\Phi(a')+1$
and $\Phi(a''')=\Phi(a')-1$. Here the induction is on the distance
of the triangle from the reference arc $a$.

Let us also define $\Phi(\gamma)$, where $\gamma$ is an oriented arc
with endpoints on $a',a''\in \tilde \tau$, to be the number of right
turns minus the number of left turns of a geodesic representative of
$\gamma$.  We will write $\gamma^{-1}$ for $\gamma$ with reversed
orientation, and $\gamma\gamma'$ for the concatenation of $\gamma$,
followed by $\gamma'$.

\begin{figure}[ht]
  {\epsfxsize=3in\centerline{\epsfbox{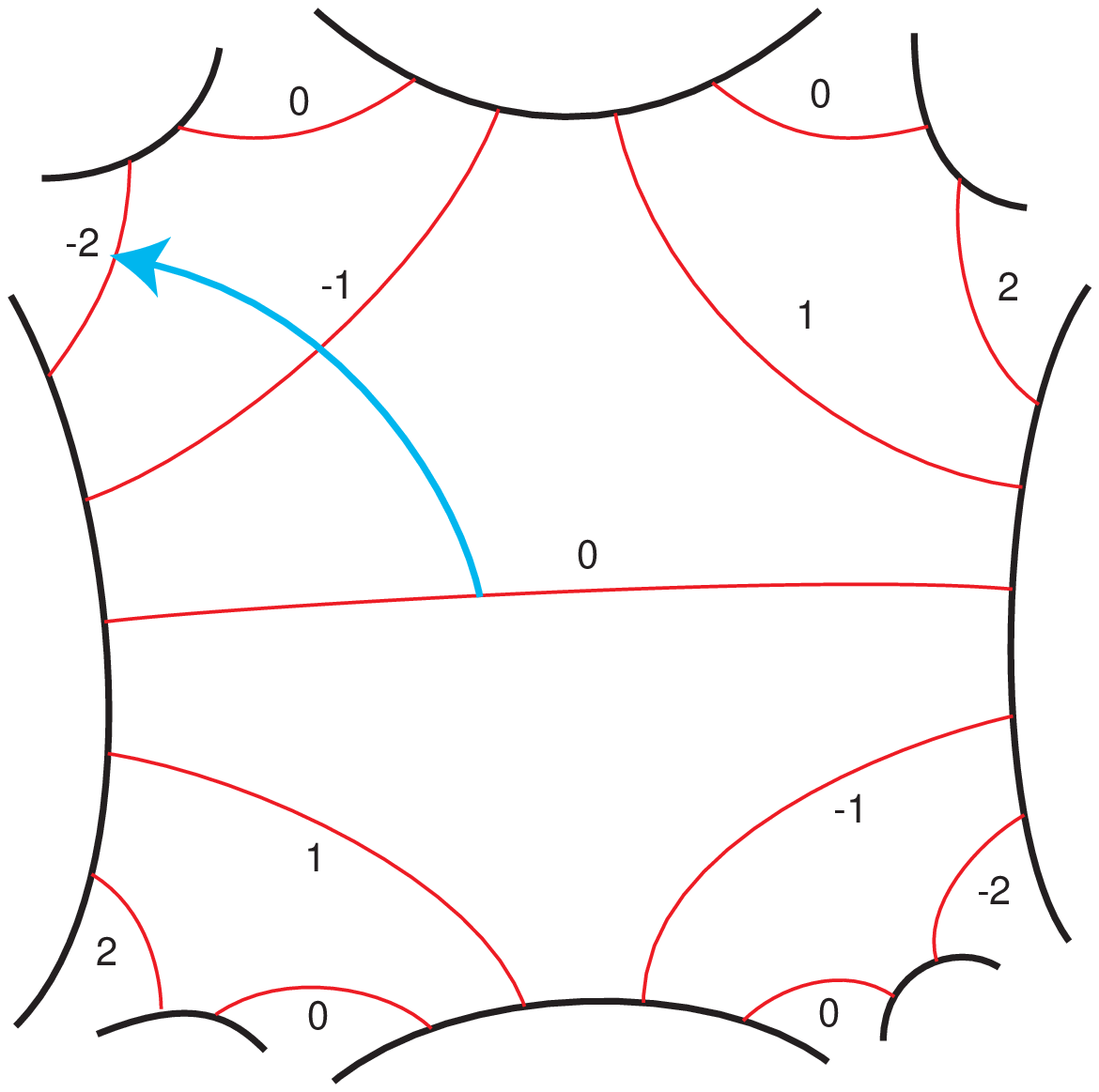}}}
\caption{The tessellation of the universal cover $\wt{S}$ of $S$ and
values of the Rademacher function on the tessellation (given right
next to each edge).} \label{rademacher}
\end{figure}

The Rademacher function has the following useful properties:

\begin{lemma}    \label{Rad}
Let $a',a'',a'''\in \tilde\tau$,  $\gamma$ be a geodesic arc from
$a'$ to $a''$ and $\gamma'$ be a geodesic arc from $a''$ to $a'''$.
Then: \be
\item $\Phi(\gamma^{-1}) = -\Phi(\gamma)$.
\item $\Phi(\gamma)+\Phi(\gamma')=\Phi(\gamma\gamma')+3\varepsilon$, where
$\varepsilon=-1,0,1$, depending on the angles made by $\gamma$ and
$\gamma'$.\ee
\end{lemma}

\begin{proof}
(1) is immediate --- a right turn becomes a left turn when traveling
in the other direction.

To prove (2), suppose first that $a',a'',a'''$ form a triangle in
$\wt{S}$.  If $a',a'',a'''$ are in counterclockwise order, then
$\Phi(\gamma)+\Phi(\gamma') = \Phi(\gamma\gamma') +3$;  if
$a',a'',a'''$ are in clockwise order, then
$\Phi(\gamma)+\Phi(\gamma')=\Phi(\gamma\gamma')-3$.  We can then
reduce to the above situation by applying (1) to subarcs of $\gamma$
and $\gamma'$ that cancel.  Observe that  $\varepsilon=0$ happens
when either (i) the concatenation of $\gamma$ and $\gamma'$ is
already efficient with respect to $\tilde\tau$, i.e.,
$\gamma\gamma'$ and its geodesic representative intersect
$\tilde\tau$ in the same number of times, or (ii) the sequence of
arcs of $\tilde\tau$ intersecting $\gamma'$ is exactly the reverse
of those intersecting $\gamma$ (or vice versa).
\end{proof}

\subsection{Rademacher functions for periodic diffeomorphisms}

We initially envisioned a more complicated proof of
Theorem~\ref{thm:cylindrical} which involved Rademacher functions.
Although no longer logically necessary, in this subsection we
describe Rademacher functions which are adapted to periodic
diffeomorphisms.

Let $\psi$ be a periodic diffeomorphism on $S$ and let $S'$ be the
orbifold obtained by quotienting $S$ by the action of $\psi$. (For
more details on 2-dimensional orbifolds, see \cite{Sc}.) The
orbifold $S'$ will have the same number of boundary components as
$S$, and $m$ orbifold singularities.  Assume $S'$ is not a disk with
$m=1$. Then cut up $S'$ using (not always geodesic) arcs from $\bdry
S'$ to itself which do not pass through any orbifold singularities,
so that the complementary regions are either (1) triangles which do
not contain any singularities, or (2) monogons containing exactly
one singularity.   Denote the union of such arcs on $S'$ by $\tau'$,
their preimage on $S$ by $\tau$, and the preimage in the universal
cover $\wt{S}$ by $\tilde\tau$. The connected components of $\wt{S}-
\tilde\tau$ are $s$-gons, where $s>1$ (no monogons!). In particular,
if we have a connected component of $\wt{S}- \tilde\tau$ which
projects to a monogon containing a singularity of order $s$, then
the component is an $s$-gon.

We now define the (generalized) Rademacher function $\Phi$ on the
oriented geodesic arcs $\gamma$ of $\wt{S}$ which have endpoints on
$\tilde \tau$.  The function $\Phi$ will now take values in $\Q$
instead of $\Z$. We define $\Phi(\gamma)$ to be the sum, over the
set of $s$-gons $P$ intersecting $\gamma$ in their interior, of
$\Phi(\gamma|_{P})$, so we may assume $\gamma$ to be an arc in $P$.
Order the edges of $P$ in $\tau$ in counterclockwise order to be
$a_0,a_1,\dots,a_{s-1}$. If $\gamma$ goes from $a_0$ to $a_i$, then
define
$$\Phi(\gamma|_{P})= 3\left({s-2 -2(i-1)\over s}\right)=3-{6i\over s}.$$
Observe this formula agrees with the previous definition of the
Rademacher function when $\tilde\tau$ consists only of triangles.
Also, it is possible that $s=2$, in which case
$\Phi(\gamma|_{P})=0$. [It is instructive to compute
$\Phi(\gamma|_{P})$ if $\gamma$ connects $a_0$ to $a_i$ and $s=7$.
In that case, the values are, in counterclockwise order, ${15\over
7}, {9\over 7}, {3\over 7}, -{3\over 7}, -{9\over 7}, -{15\over
7}$.]

It is not difficult to see that the generalized $\Phi$ also
satisfies Lemma~\ref{Rad}.  Also observe that $\Phi$ is {\em
invariant} under $\psi$.

\subsection{Rademacher functions for pseudo-Anosov
homeomorphisms}

Let $S$ be a compact oriented surface endowed with a hyperbolic
metric so that $\bdry S$ is geodesic, and let $\psi$ be a
pseudo-Anosov homeomorphism of $S$.  The reader is referred to
\cite{FLP} for the stable/unstable foliation perspective and to
\cite{Bn,CB} for the lamination perspective.

We will first explain the Rademacher function $\Phi$ from the
lamination perspective, and later rephrase the definition in the
language of singular foliations.  The well-definition of the
Rademacher function and its properties are easier to see in the
lamination context, whereas the characteristic foliations that we
construct in Section~\ref{section: construction} will closely hew to
the stable foliation.

In the geodesic lamination setting, $\Phi$ of an oriented arc
$\gamma:[0,1]\rightarrow S$ is defined as follows:  Let
$\Lambda=\Lambda^s$ be the stable lamination and $\S$ be the union
of all the prongs of $S$. Then isotop $\gamma$ relative to its
endpoints so that $\gamma$ is geodesic, or at least intersects
$\S\cup\Lambda$ efficiently (assuming $\gamma$ is not contained in
$\S\cup\Lambda$). Also let $\wt\Lambda$ and $\wt{\S}$ be the
preimages of $\Lambda$ and $\S$ in the universal cover $\pi:\wt
S\rightarrow S$, and $\wt\gamma$ be any lift of $\gamma$. Now
consider the (open) intervals of
$\operatorname{Im}(\wt\gamma)-(\wt{\S}\cup \wt\Lambda)$. Then
$\Phi(\gamma)$ is a signed count of intervals, both of whose
endpoints lie on $\wt{\S}$. (We throw away all other intervals!) The
sign is positive if the interval is oriented in the same direction
as $\bdry \wt{S}$, and negative otherwise. Although there are
infinitely many intervals of
$\operatorname{Im}(\wt\gamma)-(\wt{\S}\cup \wt\Lambda)$, the sum
$\Phi(\gamma)$ is finite. In fact, if $Q$ is a connected component
of $S-(\S\cup\Lambda)$ which nontrivially intersects $\bdry S$, and
$\wt{Q}$ is its lift to the universal cover, then the distance
between two lifts $\wt{P}_j$, $\wt{P}_{j+1}$ of prongs on $\bdry
\wt{Q}$ is bounded below.  See Figure~\ref{rade} for a sample
calculation of $\Phi(\gamma)$. We will usually blur the distinction
between arcs and isotopy classes of arcs.

\begin{figure}[ht]
\begin{overpic}[height=1.5in]{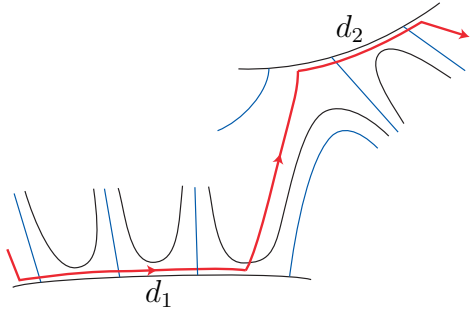}
\put(30,-3){$d_1$} \put(71,54) {$d_2$}\end{overpic} \vskip.15in
\caption{The Rademacher function $\Phi$ on the given arc is $1$,
with a contribution of $2$ from the component $d_1$ of $\bdry\wt{S}$
and a contribution of $-1$ from the component $d_2$ of
$\bdry\wt{S}$.  Here the blue arcs are the lifts of the prongs.}
\label{rade}
\end{figure}

\begin{proposition}\label{prop: rademacher}
The Rademacher function $\Phi$ satisfies the following: \be
\item $\Phi$ is invariant under $\psi$;
\item $\Phi(\gamma^{-1})=-\Phi(\gamma)$;
\item $\Phi (\gamma \gamma' )=\Phi (\gamma )+\Phi (\gamma'
)+\varepsilon$, where $\varepsilon =-1,0$ or $1$;
\item Let $\gamma$ be an arc which parametrizes
a component $(\bdry S)_i$ of $\bdry S$, i.e., $\gamma(0)=\gamma(1)$
and $\gamma$ wraps once around $(\bdry S)_i$, in the direction of
the boundary orientation of $S$. If $\gamma(0) \in \S$, then $\Phi
(\gamma )=n_i$, where $n_i$ is the number of prongs along $(\bdry
S)_i$, and if $\gamma(0) \notin \S$, then $\Phi (\gamma )= n_i-1$.
\ee
\end{proposition}

\begin{proof}
(2) and (4) are straightforward.

(1) Recall that $\psi(\Lambda)=\Lambda$.  This implies that $\psi$
maps complementary regions of $\Lambda$, i.e., connected components
of $S-\Lambda$, to complementary regions of $\Lambda$.  In
particular, an interior $n$-gon is either mapped to itself (its
edges might be cyclically permuted) or mapped to another interior
$n$-gon (with the same $n$).  A semi-open annulus region $A$, i.e.,
$n$-gon with a disk removed, is also mapped to itself, and its edges
cyclically permuted.  We may also assume that the prongs along the
boundary component $(\bdry S)_i$ are cyclically permuted.

Therefore, if $\gamma$ is a geodesic arc, then $\psi(\gamma)$ is not
necessarily geodesic, but at least intersects $\S\cup \Lambda$
efficiently. Moreover, $\psi$ is type-preserving: intervals of
$\operatorname{Im}(\wt\gamma)-(\wt{\S}\cup \wt\Lambda)$ with
endpoints on $\wt\S$ get mapped to intervals with endpoints on
$\wt\S$, and intervals without both endpoints on $\wt\S$ get mapped
to intervals without both endpoints on $\wt\S$. This proves (1).

(3)  First isotop $\gamma$ and $\gamma'$ relative to their endpoints
so that they are geodesic. Then lift $\gamma$, $\gamma'$ and
$\gamma\gamma'$ to $\wt{S}$.  We abuse notation and omit tildes,
with the understanding that the terminal point of $\gamma$ is the
initial point of $\gamma'$, even in the universal cover.

Suppose $\gamma$ and $\gamma'$ can be factored into
$\gamma_0\gamma_1$ and $\gamma_0'\gamma_1'$, respectively, where the
initial point of $\gamma_1$ and the terminal point of $\gamma_0'$
lie on the same leaf $\wt L$ of $\wt \Lambda$.  In that case we may
contract $\gamma_1\gamma_0'$ to a point on $\wt L$, using (2) in the
process.  By successively shortening $\gamma$ and $\gamma'$ if
possible, we are reduced to the cases (i), (ii), (iii), or (iv)
below.

Let $\wt Q$ be a connected component of $\wt S-(\wt{\S}\cup
\wt\Lambda)$ which nontrivially intersects $\bdry \wt{S}$, and let
$\wt Q'$ be a lift of an interior $m$-gon.

\s (i). Suppose $\gamma$ and $\gamma'$ are arcs in $\wt Q'$. There
are no contributions from interior $m$-gons, so $\Phi(\gamma)$,
$\Phi(\gamma')$, and $\Phi(\gamma\gamma')$ are all zero.

\s (ii). Suppose $\gamma, \gamma'$ are arcs in $\wt Q$, and all the
endpoints of $\gamma$, $\gamma'$ lie on $\bdry \wt Q\cap
(\wt{\S}\cup \wt\Lambda)$. If $\gamma(0)$ and $\gamma'(1)$ lie on
the same leaf of $\wt\S\cup \wt\Lambda$, then
$\Phi(\gamma\gamma')=\Phi(\gamma)+\Phi(\gamma')$.  Hence we may
assume that $\gamma(0)$ and $\gamma'(1)$ lie on distinct leaves.
Depending on whether $\gamma(0)$, $\gamma(1)=\gamma'(0)$,
$\gamma'(1)$ are in counterclockwise order or not, we have
$\Phi(\gamma\gamma')=\Phi(\gamma)+\Phi(\gamma')\pm 1$.

\s (iii). Suppose $\gamma,\gamma'$ are arcs in $\wt Q$, and
$\gamma(0), \gamma'(1)$ lie on $\bdry \wt Q\cap (\wt\S\cup
\wt\Lambda)$, but $\gamma(1)=\gamma'(0)$ does not. Then
$\Phi(\gamma)=\Phi(\gamma')=0$, and $\Phi(\gamma\gamma')$ is $0$ or
$1$.  (A similar consideration holds if two of $\gamma(0)$,
$\gamma(1)=\gamma'(0)$, $\gamma'(1)$ lie on $\bdry \wt Q\cap
(\wt{\S}\cup\wt\Lambda)$.)

\s (iv). Suppose $\gamma, \gamma'$ are arcs in $\wt Q$, and
$\gamma(0), \gamma'(1)$ do not lie on $\bdry \wt Q\cap (\wt{\S}\cup
\wt\Lambda)$.  Then $\Phi(\gamma)$, $\Phi(\gamma')$, and
$\Phi(\gamma\gamma')$ are all zero.
\end{proof}

Next we translate the definition of $\Phi$ into the singular
foliation language. Let $\F^s$ (resp.\ $\F^u$) be the invariant
stable (resp.\ unstable) foliation of $S$ with respect to $\psi$. We
will take $\F=\F^s$. The boundary of $S$ is tangent to $\F$, and
$\F$ has $n_i$ singular points of saddle type along the $i$th
component of $(\bdry S)_i$ of $\partial S$. Here $n_i$ is also the
number of prongs that end on the $(\bdry S)_i$ in the lamination
picture. Let $\S$ be the union of separatrices of the saddle points
on $\bdry S$ that are not tangent to $\bdry S$.  (This set
corresponds to the union of the prongs in the lamination picture.)
Then $\wt{\S}=\pi^{-1}(\S)$ can be decomposed into a disjoint union
of sets $\wt{\S}_d$, where $d$ is a component of $\partial \wt{S}$
and $\wt{\S}_d$ is the union of components of $\wt {\S}$ which
intersect $d$.

Given an oriented arc $\gamma:[0,1]\rightarrow S$, we isotop it,
relative to its endpoints, to an oriented arc $\gamma'$ so that (i)
$\gamma'$ has efficient intersection with $\S$, (ii) $\gamma'$ is
piecewise smooth, and (iii) each smooth piece is either transversal
to $\F$ away from the interior singularities, or is contained in
$\partial S$.  Such an arc $\gamma'$ is called a {\em
quasi-transversal} arc, in the terminology of \cite[Expos\'e 5
(I.7)]{FLP}. The proof of the existence of the isotopy is given in
\cite[Expos\'e 12 (Lemma 6)]{FLP}.  To pass from the geodesic
lamination $\Lambda$ to the foliation $\F$, we collapse the
interstitial regions of $\Lambda$. If we take a geodesic
representative $\gamma''$ of $\gamma$, then the desired
quasi-transversal arc $\gamma'$ is the image of $\gamma''$ under the
collapsing operation.

We now rephrase $\Phi(\gamma)$ with respect to $\mathcal{F}$. Given
the arc $\gamma$, choose a quasi-transversal representative
$\gamma'$ with the same endpoints, and let $\wt\gamma'$ be any lift
of $\gamma'$ to $\wt{S}$. Then $\Phi (\gamma)$ is the sum, over all
components $d$ of $\partial \wt{S}$, of the signed number of
intervals of $\operatorname{Im}(\wt\gamma')- \wt{\S}_d$ that do not
contain an endpoint of $\wt\gamma'$.  The signs of the intervals are
assigned as follows: positive if $\wt\gamma'$ is oriented in the
same direction as $\bdry \wt S$ along the interval, and negative
otherwise. Alternatively, $\Phi(\gamma)$ is the sum of the signed
number of intersections of $\wt{\S}_d$ with $\wt\gamma'$ minus one,
if we have at least one intersection.

The are also slight variants of $\Phi$ described above.  The
simplest modification is to use the unstable lamination instead of
the stable one.  Also, we can take the universal cover of $S-\cup_i
D_i$, where $D_i$ are small disks removed from interior $n$-gons;
this version then also counts contributions along interior $n$-gons.
However, our $\Phi$ and its variants are ``fake'' Rademacher
functions, which only register boundary rotations and discard all
other intersections with $n$-gons. We close this section with a
question:

\begin{q}
Is there a ``genuine'' Rademacher function $\Phi(\gamma)$ which is
adapted to a stable geodesic lamination $\Lambda^s$ in the sense
that it actually somehow sums the ``left turn'' and ``right turn''
contributions of $\gamma$, where the sum is over all the intervals
$\gamma-\Lambda^s$.
\end{q}

\section{Construction of the Reeb vector field}\label{section: construction}

\subsection{First return maps} \label{subsection: flux}

Let $S$ be a compact oriented surface with nonempty boundary,
$\omega$ be an area form on $S$, and $h$ be an area-preserving
diffeomorphism of $(S,\omega)$.  Suppose for the moment that
$h|_{\bdry S}$ is not necessarily $id$, but does not permute the
boundary components.

Consider the mapping torus $\Sigma (S,h)$ of $(S,h)$, which we
define as $(S\times [0,1])/(x,1) \sim (h(x),0)$. Here $(x,t)$ are
coordinates on $S\times[0,1]$. If there is a contact form $\alpha$
on $\Sigma(S,h)$ for which $d\alpha \vert_{S\times \{0\}} =\omega$
and the corresponding Reeb vector field $R_\alpha$ is directed by
$\bdry_t$, then we say $h$ is {\em the first return map of
$R_\alpha$}.

We are interested in the realizability of a given pseudo-Anosov
$\psi$ as the first return map of some $R_\alpha$, after possibly
perturbing $\psi$ near the singular points to make $\psi$ smooth. We
summarize the following results from \cite{CHL}:

\s\n {\bf Fact 1.} If $h^*-id: H^1(S;\R)\rightarrow H^1(S;\R)$ is
invertible, then $h$ can be realized as the first return map of some
$R_\alpha$.  Hence, a pseudo-Anosov homeomorphism $\psi$ (after a
small perturbation near its singular points) isotopic to such an $h$
can be realized as the first return map of some $R_\alpha$.

\s\n {\bf Fact 2.} On the other hand, there exist pseudo-Anosov
homeomorphisms $\psi$ which (even after a small perturbation near
its singular points) cannot be realized as the first return map of
any $R_\alpha$.

\s If $\psi$ is realizable, then we can use $R_\alpha$ and avoid the
technicalities of the rest of Section~\ref{section: construction}.
If $\psi$ is not realizable as a first return map of a Reeb vector
field $R_\alpha$, then there are two strategies: (1) Enlarge the
class of vector fields to the class of {\em stable Hamiltonian}
ones, as described in \cite{BEHWZ}.  The drawback is that one needs
to prove invariance of the generalized contact homology groups using
the bifurcation strategy, instead of the continuation method which
is usually used in contact homology. (2) Carefully construct a Reeb
vector field for which we have some control over the periodic
orbits.  The drawback of this approach is that the construction is
rather complicated. Since the details of the bifurcation strategy do
not exist in the literature at this moment, we opt for (2).  This
will occupy the rest of the section.

\subsection{Preliminary constructions}

\subsubsection{Construction of contact $1$-form on $S\times[0,1]$}

Let $S$ be a compact oriented surface with nonempty boundary.
Consider $S\times[0,1]$ with coordinates $(x,t)$. We will first
construct a contact $1$-form $\alpha$ and the corresponding Reeb
vector field $R$ on $S\times[0,1]$.

\begin{lemma} \label{interpolation}
Given $1$-forms $\beta_0, \beta_1$ on $S$ which agree near $\bdry S$
and which satisfy $d\beta_i>0$, $i=0,1$, there exist contact
$1$-forms $\alpha=\alpha_\varepsilon$ and Reeb vector fields
$R=R_\varepsilon$ on $S\times[0,1]$, depending on $\varepsilon>0$
sufficiently small, which satisfy the following properties:

\begin{enumerate}
\item $\alpha=dt +\varepsilon\beta_t$, where $\beta_t$,
$t\in[0,1]$, is a $1$-form on $S$ which varies smoothly with $t$.

\item $R$ is directed by ${\bdry\over \bdry t}+Y$,
where $Y=Y_\varepsilon$ is tangent to $\{t=const\}$.

\item $Y=0$ in a neighborhood of $(\bdry S)\times[0,1]$.

\item At points $x\in S$ where $\beta_0$ and $\beta_1$
have the same kernel, $Y$ is tangent to $\ker \beta_0=\ker \beta_1$.

\item The direction of the Reeb vector field $R_\varepsilon$ does not depend on
the choice of $\varepsilon>0$, as long as $\varepsilon$ is
sufficiently small to satisfy the contact condition.

\item By taking $\varepsilon>0$ sufficiently small, $R_\varepsilon$ can be made
arbitrarily close to ${\bdry \over \bdry t}+Y$.

\end{enumerate}
\end{lemma}

\begin{proof}
Let $\chi:[0,1]\rightarrow[0,1]$ be a smooth map for which
$\chi(0)=0$, $\chi(1)=1$, $\chi'(0)=\chi'(1)=0$, and $\chi'(t)>0$
for $t\in(0,1)$. Consider the form
$$\beta_t= (1-\chi(t))\beta_0 +\chi(t)\beta_1.$$
Let us write $\omega_t= (1-\chi(t))d\beta_0+\chi(t)d\beta_1$.
Observe that $\omega_t$ is an area form on $S$.

We then compute
\begin{eqnarray*} d\alpha &=& \varepsilon((1-\chi
(t))d\beta_0 +\chi(t)d\beta_1 + \chi' (t)dt\wedge (\beta_1 -\beta_0
))\\ &=&\varepsilon(\omega_t +\chi' (t)dt\wedge(\beta_1 -\beta_0
)),\\
\alpha \wedge d\alpha &=& \varepsilon dt\wedge \omega_t
-\varepsilon^2\chi' (t)dt\wedge \beta_0 \wedge \beta_1.
\end{eqnarray*}
If $\varepsilon$ is small enough, $\alpha$ satisfies the contact
condition $\alpha \wedge d\alpha
>0$.

The Reeb vector field $R$ for $\alpha$ is collinear to
$\frac{\partial}{\partial t} +Y$, where $Y$ is tangent to the levels
$\{t=const\}$ and satisfies
\begin{equation}
\label{Y} i_Y
\omega_t =\chi' (t)(\beta_0 -\beta_1 ).
\end{equation}
(Verification: $$ i_{{\bdry\over \bdry t}+Y} d\alpha=
\varepsilon\cdot i_{{\bdry\over \bdry t}+Y}
(\omega_t+\chi'(t)dt\wedge(\beta_1-\beta_0))=0$$ implies that
$$
i_{Y}\omega_t+\chi'(t)(\beta_1-\beta_0) - dt\cdot
\chi'(t)(\beta_1-\beta_0)(Y)=0.
$$
This separates into two equations
\begin{eqnarray*} i_{Y}\omega_t+\chi'(t)(\beta_1-\beta_0) &=&0,\\
\chi'(t)(\beta_1-\beta_0)(Y)&=&0. \end{eqnarray*} The first is
Equation~\ref{Y}, and the second follows from the first.) Observe
that (3), (4), and (5) are consequences of Equation~\ref{Y}. To
prove (6), observe that $\alpha({\bdry\over \bdry t}+Y)=
1+\varepsilon \beta_t(Y)$. Then $R_\varepsilon= { {\bdry\over \bdry
t}+Y\over |1+\varepsilon\beta_t(Y)|}$, which approaches ${\bdry\over
\bdry t}+Y$ as $\varepsilon\rightarrow 0$.
\end{proof}

\subsubsection{Construction of contact $1$-form on $\Sigma(S,g)$.}
\label{subsub: N}

For notational simplicity, assume that $\bdry S$ is connected. Let
$\beta$ be a $1$-form on $S$ satisfying $d\beta>0$.  We say that
$\beta$ {\em exits $\bdry S$ uniformly with respect to} a
diffeomorphism $g: S\stackrel\sim\rightarrow S$ if there exists a
small annular neighborhood $A=S^1\times[0,1]$ of $\bdry S$ with
coordinates $(\theta,y)$ so that \be \item $\bdry S=S^1\times\{0\}$
and $\beta= (C-y)d\theta$, where $C$ is a constant $\gg 0$.
\item $g$ restricts to a rotation $(\theta,y)\mapsto (\theta+C',y)$ on
$S^1\times[0,1]$, where $C'$ is some constant. \ee

Suppose $\beta$ exits $\bdry S$ uniformly with respect to $g$. The
easiest construction of a contact $1$-form on $\Sigma(S,g)$ would be
to set $\beta_0=g_*\beta=(g^{-1})^*\beta$ and $\beta_1=\beta$, and
glue up the contact $1$-form from Lemma~\ref{interpolation}.
However, in this paper we will use a slightly more complicated
$1$-form, given below.

\s\n {\bf Construction.} Let $\beta_0=
g_*(f_{\varepsilon'}\beta)=f_{\varepsilon'} (g_*\beta)$,
$\beta_{1/2}=\beta$, and $\beta_1=f_{\varepsilon'}\beta$, where
$\varepsilon'>0$ is a sufficiently small constant. Here,
$f_{\varepsilon'}:S\rightarrow \R$ is $\varepsilon'$ outside the
small annular neighborhood $A$ of $\partial S$, and, inside $A$, is
independent of $\theta$, equals $1$ for $y\in[0,\varepsilon'']$, and
satisfies ${\bdry f_{\varepsilon'}\over \bdry y}<0$ for
$y\in(\varepsilon'',1)$. We can easily verify that
$df_{\varepsilon'} \wedge \beta \geq 0$; hence
$f_{\varepsilon'}\beta$ is a primitive of an area form on $S$. Then
let $\beta_t$ be the interpolation between $\beta_0$ and
$\beta_{1/2}$ for $t\in[0,{1\over 2}]$, given by
$$\beta_t=(1-\chi_0(t))\beta_0+\chi_0(t)\beta_{1/2},$$
where $\chi_0:[0,{1\over 2}]\rightarrow [0,1]$ is a smooth map for
which $\chi_0(0)=0$, $\chi_0({1\over 2})=1$,
$\chi'_0(0)=\chi'_0({1\over 2})=0$ and $\chi'_0(t)>0$ for
$t\in(0,{1\over 2})$. Similarly define the interpolation $\beta_t$
between $\beta_{1/2}$ and $\beta_1$ for $t\in[{1\over 2},1]$ by
$$\beta_t=(1-\chi_1(t))\beta_{1/2}+\chi_1(t)\beta_1,$$
where $\chi_1:[{1\over 2},1]\rightarrow [0,1]$ is a smooth map for
which $\chi_1({1\over 2})=0$, $\chi_1(1)=1$, $\chi'_1({1\over
2})=\chi'_1(1)=0$ and $\chi'_1(t)>0$ for $t\in({1\over 2},1)$. Then
we set $\alpha_{\varepsilon,\varepsilon'} =dt+\varepsilon\beta_t$ as
in Lemma~\ref{interpolation}. It induces a contact form
$\alpha_{\varepsilon,\varepsilon'}$ on $\Sigma(S,g)$.

Let $\omega_t=d_2\beta_t$, where $d_2$ indicates the exterior
derivative in the $S$-direction. Then the Reeb vector field
$R=R_{\varepsilon,\varepsilon'}$ for
$\alpha=\alpha_{\varepsilon,\varepsilon'}$ is collinear to
$\frac{\partial}{\partial t} +Y$, where $Y=Y_{\varepsilon'}$ is
tangent to the levels $\{ t=const\}$ and satisfies
\begin{equation} \label{equation: Y varepsilon}
i_{Y} \omega_t =-\dot{\beta_t}.
\end{equation}
Here a dot means ${d\over dt}$. Observe that $Y$ does not depend on
$\varepsilon$ (by (5) of Lemma~\ref{interpolation}) and the
direction of $R_{\varepsilon,\varepsilon'}$ does not depend on
$\varepsilon$. By taking $\varepsilon$ sufficiently small as in (6)
of Lemma~\ref{interpolation}, we can make $R$ as close to
${\bdry\over \bdry t}+Y$ as we like.

\s\n {\bf Description of $R_{\varepsilon,\varepsilon'}$.}  Let $Z$
be a vector field which directs $\ker \beta$. Fix a small
neighborhood $U\subset S$ of the singular set of $\beta$. Also let
$A'\subset A$ be the set $\{0\leq y\leq \varepsilon''\}$.

\begin{lemma} \label{lemma:uniform}
The Reeb vector field $R_{\varepsilon,\varepsilon'}$ is directed by
and is arbitrarily close to ${\bdry\over \bdry t}+Y_{\varepsilon'}$,
provided $\varepsilon>0$ is sufficiently small.  The vector field
$Y_{\varepsilon'}$ satisfies the following:
\begin{enumerate}
\item $Y_{\varepsilon'}=0$ on $A'\times[0,1]/\sim$.
In particular, $R_{\varepsilon,\varepsilon'}$ is tangent to $\bdry
\Sigma(S,g)$.
\item $Y_{\varepsilon'}=0$ when $t=0$ and $t={1\over 2}$.
\item On $(S-A'-U)\times (0,{1\over 2})$, ${Y_{\varepsilon'}(x,t)\over
|Y_{\varepsilon'}(x,t)|}\rightarrow - {Z(x)\over |Z(x)|}$ uniformly,
as $\varepsilon'\rightarrow 0$.
\item On $(S-A')\times ({1\over 2},1)$, $Y_{\varepsilon'}$ is parallel to
and in the same direction as $Z$.
\end{enumerate}
\end{lemma}

One can think of the vertical projections $Y$ of $R$ as what happens
in a ``puffer machine'': Between $t=0$ and $t={1\over 2}$, $Y$ flows
away from $\bdry S$ and is sucked towards the singularities of
$\beta$ along $\ker \beta$ (with some error), and, between
$t={1\over 2}$ and $t=1$, $Y$ flows away from the singularities of
$\beta$ towards $\bdry S$ along $\ker\beta$ (with no error).

\begin{proof}
This follows from Equation~\ref{equation: Y varepsilon}. First
suppose $t\in[0, {1\over 2}]$. Then
\begin{equation}\label{eqn: beta dot}
\dot\beta_t=\chi'_0(t)(\beta_{1/2}-\beta_0)=
\chi'_0(t)(\beta-f_{\varepsilon'} (g_*\beta)).
\end{equation}
(1) follows from $\dot\beta_t=0$ by observing that
$\beta_0=\beta_{1/2}$ on $A'$. Similarly, (2) follows from
$\dot\beta_t=0$ by observing that $\chi'_0(t)=0$ when $t=0$ or
$t={1\over 2}$. We now prove (3). The vector field
$Y_{\varepsilon'}$ directs the kernel of $f_{\varepsilon'}
(g_*\beta)-\beta$, since $\chi'_0(t)>0$ when $t\in(0,{1\over 2})$.
In order for ${Y_{\varepsilon'}(x,t)\over |Y_{\varepsilon'}(x,t)|}$
to make sense, we need $Y_{\varepsilon'}(x,t)$ to be nonzero ---
this is achieved by making $\varepsilon'$ sufficiently small and
restricting to $(S-A'-U)\times (0,{1\over 2})$. The uniform
convergence of $f_{\varepsilon'} (g_*\beta)-\beta$ to $-\beta$ on
$(S-A-U)\times(0,{1\over 2})$ as $\varepsilon'\rightarrow 0$ (note
that we wrote $A$ instead of $A'$) implies the uniform convergence
of ${Y_{\varepsilon'}(x,t)\over |Y_{\varepsilon'}(x,t)|}$ to $-
{Z(x)\over |Z(x)|}$ on $(S-A-U)\times(0,{1\over 2})$. On the other
hand, on $(A-A')\times (0,{1\over 2})$, ${Y_{\varepsilon'}(x,t)\over
|Y_{\varepsilon'}(x,t)|}= - {Z(x)\over |Z(x)|}$ already.

The situation $t\in[{1\over 2},1]$ is similar and is left to the
reader.
\end{proof}

\subsubsection{Extension to the binding} \label{subsub: extension to
binding} Let $S_0$ be the surface obtained by gluing an annulus
$A_0=S^1\times[-1,0]$ to $S$ so that $S^1\times\{0\}$ is identified
with $\bdry S$.  Let $h:S_0\stackrel\sim\rightarrow S_0$ be a
diffeomorphism which restricts to the identity on $\bdry S_0$.
Suppose $h=h_0\cup g$, where $g$ is the diffeomorphism on $S$ as
above and $h_0:S^1\times[-1,0]\stackrel\sim\rightarrow
S^1\times[-1,0]$ maps $(\theta,y)\mapsto (\theta+C'(y+1),y)$, where
$C'$ is the positive constant which records the rotation about the
boundary and $\theta\in S^1= \R/\Z$.

Fix a $1$-form $\beta$ on $S$ so that $d\beta>0$ and $\beta$ exits
$\bdry S$ uniformly with respect to $g$. Let
$\alpha=\alpha_{\varepsilon,\varepsilon'}$ and
$R=R_{\varepsilon,\varepsilon'}$ be the contact $1$-form and Reeb
vector field constructed on $\Sigma(S,g)$ as in the previous
subsection. According Lemma~\ref{lemma:uniform},
$R=\frac{\partial}{\partial t}$ on $\partial \Sigma(S,g)$; hence
$\bdry \Sigma(S,g)$ is linearly foliated by $R$. Also the
characteristic foliation of $\alpha$ along $\bdry \Sigma(S,g)$ is
linearly foliated by leaves which are close to $\bdry S$.

We now extend $\alpha$ and $R$ to the closed $3$-manifold $M$ which
corresponds to the open book $(S_0,h)$, by gluing a neighborhood
$N(K)$ of the binding to $\partial \Sigma(S,g)$. Endow $N(K)\simeq
\R/\Z \times D^2$ with cylindrical coordinates $(z,(r,\theta ))$ so
that $D^2=\{r\leq 1\}$. The fibration of the open book is given on
$N(K)$ by $(z,r,\theta)\mapsto \theta$. If we use coordinates
$({\theta\over 2\pi},z)$ to identify $\bdry N(K)\simeq \R^2/\Z^2$,
then $R$ has slope $C'>0$ and $\xi|_{\bdry (S^1\times D^2)}$ has
slope $-{1\over C_\varepsilon}$ for $C_\varepsilon\gg 0$. We extend
the contact form $\alpha_{\varepsilon, \varepsilon'}$ to $N(K)$ by
an equation of the form $a_{\varepsilon} (r) dz +b_{\varepsilon}
(r)d\theta$, where $a_\varepsilon(r)>0$ and $b_\varepsilon(r)\geq
0$. The characteristic foliation on $\{r=r_0\}$ will then be
directed by $a_\varepsilon(r_0){\bdry\over \bdry \theta}
-b_\varepsilon(r_0) {\bdry\over \bdry z}$. The contact condition is
given by the inequality: $a_{\varepsilon}
b_{\varepsilon}'-a_{\varepsilon}' b_{\varepsilon}>0$. It expresses
the fact that the plane curve $(a_{\varepsilon}
(r),b_{\varepsilon}(r) )$ is transverse to the radial foliation of
the plane, and rotates in the counterclockwise direction about the
origin. The Reeb vector field is given by $R_{\varepsilon}
=\frac{1}{a_{\varepsilon} b_{\varepsilon}'
-a_{\varepsilon}'b_{\varepsilon}}
(b_{\varepsilon}'\frac{\partial}{\partial z} -a_{\varepsilon}'
\frac{\partial}{\partial \theta} )$. The boundary condition uniquely
determines the values $a_{\varepsilon} (1)$, $b_{\varepsilon} (1)$,
$a_{\varepsilon}' (1)$ and $b_{\varepsilon}' (1)$.  In particular,
these values depend smoothly on $\varepsilon$. For all
$\varepsilon$, $(a_{\varepsilon} (1), b_{\varepsilon} (1))$ is in
the interior of the first quadrant.  Also, we require that
$(a_\varepsilon(r),b_\varepsilon(r))$ lie on a line segment that
starts on the positive $\theta$-axis and ends at $(a_{\varepsilon}
(1), b_{\varepsilon} (1))$, and is directed by
$(a'_\varepsilon(1),b'_\varepsilon(1))$. We then can extend
$a_{\varepsilon}$ and $b_{\varepsilon}$ on $[0,1]$, so that
$a_{\varepsilon}(r)=C_{0,\varepsilon}-C_{1,\varepsilon} r^2$ and
$b_{\varepsilon}(r)=r^2$ near $r=0$ (where
$C_{0,\varepsilon},C_{1,\varepsilon}$ are appropriate positive
constants which depend on $\varepsilon$) and so that they depend
smoothly on $\varepsilon$. By construction $R_\varepsilon$ will
linearly foliate the level tori $\{r=const>0\}$ so that the slope
remains constant ($=C'$). In particular, $R$ will be transverse to
the pages $S\times\{t\}$, except along the binding $\gamma_0$, which
is a closed orbit of $R$.

\s In the remaining subsections of this section, we will construct a
suitable diffeomorphism $g=\psi'$ which is freely homotopic to a
pseudo-Anosov homeomorphism $\psi$, and a $1$-form $\beta$ which is
adapted to $\psi$.

\subsection{Main proposition}
Let $M$ be a closed, oriented 3-manifold and $\xi$ be a cooriented
contact structure. Suppose that $\xi$ is carried by an open book
with page $S$ and monodromy $h:S\stackrel\sim\rightarrow S$.  Recall
that $h|_{\bdry S}=id$. For notational simplicity, assume that
$\bdry S$ is connected.

Suppose $h$ is freely homotopic to a pseudo-Anosov homeomorphism
$\psi$ with fractional Dehn twist coefficient $c={k\over n}$. Let
$(\F,\mu)=(\F^s,\mu^s)$ be the stable foliation on $S$, and $\lambda
>1$ be the constant such that $\psi_* \mu =\lambda \mu$. The
foliation $\F$ has saddle type singularities on $\bdry S$, and the
singular points of $\F$ on $\partial S$ are denoted by
$x_1,\dots,x_n$. (Here the subscript $i$ increases in the direction
given by the orientation of $\bdry S$.) Denote the interior
singularities of $\F$ by $y_1,\dots,y_q$. The homeomorphism $\psi$
is a diffeomorphism away from these singular points. Also let $P_i$
be the prong emanating from $x_i$, and let
$Q_{j1},Q_{j2},\dots,Q_{jm_j}$ be the prongs emanating from $y_j$,
arranged in counterclockwise order about $y_j$.

\subsubsection{$N(\bdry S)$}\label{matisse}
Let $N(\bdry S)\subset S$ be a neighborhood of $\bdry S$ with a
particular shape: \be
\item $\bdry (N(\bdry S))-\bdry S$ is a concatenation of smooth arcs which
are alternately tangent to $\mathcal{F}^u$ (the {\em vertical} arcs
$a_1,\dots,a_n$, since they are transverse to $\mathcal{F}$) and
tangent to $\F$ (the {\em horizontal} arcs $b_1,\dots,b_n$).  Here
$b_i$ is between $a_i$ and $a_{i+1}$, where the indices are taken
modulo $n$, and there is a prong $P_i$ starting at $x_i$ which exits
$N(\bdry S)$ through $a_i$.
\item Each transversal arc $a_i$ is divided into two subarcs by the prong
$P_i$ starting at $x_i$. We pick $N(\partial S)$ so that all these
subarcs have the same transverse measure $\delta \ll 1$. (This
becomes important later on!)
\item No horizontal arc $b_i$ is contained in any prong $P_j$ or
$Q_{jl}$. (This can be achieved by observing that the intersection
between $a_i$ and any prong is countable, and by shrinking $\delta$
if necessary.) \ee

Let $P_i'$ be the first component of $P_i\cap (S-int(N(\bdry S)))$
that can be reached from the singular point $x_i$, traveling inside
$P_i$.  By (3), $P_i'$ is a compact arc with endpoints on $int(a_i)$
and some $int(a_{i'})$.  Similarly, let $Q_{jl}'$ be the component
of $Q_{jl}\cap (S-int(N(\bdry S)))$ that begins at $y_j$ and ends on
some $int(a_{j'})$.

Next, endow each $a_i$ with the boundary orientation of
$S-int(N(\bdry S))$. For each $a_i$, define a parametrization
$p_i:[-\delta,\delta]\rightarrow a_i$ so that $p_i(-\delta)$ is the
initial point of $a_i$, $p_i(\delta)$ is the terminal point, and the
$\mu$-measure from $p_i(-\delta)$ to $p_i(s)$ is $s+\delta$. Let
$\varepsilon>0$ be a sufficiently small constant so that all the
leaves of $\F$ which start from $p_i([-\delta,-\delta+\varepsilon])$
exit together along some $a_{i'}$ and also avoid the prong
$P'_{i'}$. Also, for each $i$, define the map
$q_i:[-\delta,\delta]\rightarrow \bdry S$ so that $q_i(s)$ is the
point on $\bdry S$ which is closest to $p_i(s)$ with respect to the
fixed hyperbolic metric.  In particular, $q_i(0)=x_i$ and the
geodesic through $p_i(0)$ and $q_i(0)$ agrees with the prong $P_i$,
since the prong $P_i$ is perpendicular to $\bdry S$. Also let
$p_i(s)q_i(s)$ be the shortest geodesic between $p_i(s)$ and
$q_i(s)$.

\subsubsection{Walls}
Let $W$ be a properly embedded, oriented arc of $S$ so that $W\cap
N(\bdry S)$ consists of exactly two components.  The component
containing the initial point is the {\em initial arc} of $W$, the
component containing the terminal point is the {\em terminal arc} of
$W$, and $W\cap (S-int(N(\bdry S)))$ is the {\em middle arc} of $W$.
We now define the {\em walls} $W_{i,L}$, $W_{i,R}$ for
$i=1,\dots,n$. The wall $W_{i,L}$ (resp.\ $W_{i,R}$) is a properly
embedded, oriented arc of $S$ which intersects $N(\bdry S)$ in two
components. The initial arc of $W_{i,L}$ (resp.\ $W_{i,R}$) is the
geodesic arc $q_i(-\delta+\varepsilon)p_i(-\delta+\varepsilon)$
(resp.\ $q_i(\delta-\varepsilon)p_i(\delta-\varepsilon)$), the
terminal arc of $W_{i,L}$ is $p_{i'}(s_{i,L})q_{i'}(s_{i,L})$
(resp.\ $p_{i'}(s_{i,R})q_{i'}(s_{i,R})$), and the middle arc is a
leaf of $\F|_{S-int(N(\bdry S))}$.  (We may need to take a
$C^0$-small modification of the geodesic arcs, so that the walls
become smooth. From now on, we assume that such smoothings have
taken place, with the tacit understanding that the arcs
$p_i(s)q_i(s)$ are only ``almost geodesic''.) It is conceivable
that, a priori, $W_{i,L}=W_{i',R}$ for $i\not=i'$, with opposite
orientations. In that case, perturb $\varepsilon$ so that the walls
are pairwise disjoint.

\subsubsection{$N(y_j)$}
Now, for each $j$, we define $N(y_j)$ to be a sufficiently small
neighborhood of $Q_{j1}'\cup \dots\cup Q_{jm_j}'$ in $S-int(N(\bdry
S))$ so that $\bdry N(y_j)$ is a union of smooth arcs which are
alternately vertical and horizontal, and so that $N(y_j)$ has the
following properties:
\begin{enumerate}
\item $N(y_j)$ is disjoint from $N(y_{j'})$ for $j'\not=j$;
\item $N(y_j)$ does not intersect any $P_i'$;
\item Each vertical arc of $\bdry N(y_j)$ is contained in some $int(a_i)$
and is disjoint from $p_i([-\delta,-\delta+\varepsilon])$ and
$p_i([\delta-\varepsilon,\delta])$.
\end{enumerate}
(3) is possible since the horizontal arc $b_i$ is disjoint from all
the prongs.

\subsubsection{$S'$ and $S''$}
We now define the subsets $S''\subset S'\subset S$:
\begin{eqnarray*}
S' &=&S -\cup_{1\leq j\leq q} ~int(N( y_j ))- int (N(\partial S )),\\
S''&=& S'- \cup_{1\leq i\leq n} ~int(N(P_i')). \end{eqnarray*} Here
we take $N(P_i')$ to be a plaque of $\mathcal{F}$, each of whose
vertical boundary components is sufficiently short to be contained
in the interior of some vertical arc in $\bdry S'$, and is disjoint
from $p_i([-\delta,-\delta+\varepsilon])$,
$p_i([\delta-\varepsilon,\delta])$.

\subsubsection{Modified diffeomorphism $\psi'$.}
Finally, we describe the diffeomorphism
$\psi':S\stackrel\sim\rightarrow S$, which is derived from and
freely homotopic to the pseudo-Anosov homeomorphism $\psi$, and
agrees with $\psi$ outside a small neighborhood of $N(\bdry S)\cup
\psi(N(\bdry S))$. First consider the restriction $\psi:
S-int(N(\bdry S)) \rightarrow S$ with image $S-\psi(int(N(\bdry
S)))$.  Let $g_1$ be a flow on $S$ which is parallel to the stable
foliation, pushes $\psi(N(\bdry S))$ into $N(\bdry S)$, and maps
$\psi(a_i)$ inside $a_{i+k}$. (Recall ${k\over n}$ is the fractional
Dehn twist coefficient.) Next, let $g_2$ be a flow on $S$ which is
parallel to the unstable foliation and maps $g_1\circ \psi(N(\bdry
S))$ to $N(\bdry S)$. Observe that $g_2\circ g_1$ can be taken to be
supported in a neighborhood of $N(\bdry S)\cup \psi(N(\bdry S))$.
Now, $\psi'=g_2\circ g_1\circ \psi$ is a diffeomorphism from
$S-int(N(\bdry S))$ to itself, and, by choosing $g_1$ and $g_2$
judiciously, we can ensure that the restrictions $\psi':
a_i\rightarrow a_{i+k}$ and $\psi': b_i\rightarrow b_{i+k}$ are
transverse measure-preserving diffeomorphisms.  Hence we can extend
$\psi'$ to $N(\bdry S)$ by a rigid rotation about $\bdry S$ which
takes $p_i(s)q_i(s)$ to $p_{i+k}(s)q_{i+k}(s)$.  In the case $k=0$,
$\psi'|_{N(\bdry S)}$ is the identity.

\subsubsection{Statement of proposition.}

Let $\gamma_1$ and $\gamma_2$ be two properly embedded oriented arcs
of $S$ with the same initial point $x\in \bdry S$.  Let
$\wt\gamma_1$ and $\wt\gamma_2$ be lifts to the universal cover $\wt
S$, starting at the same point $\wt x$. We say that $\gamma_1$ is
{\em setwise to the left of} $\gamma_2$ and write $\gamma_1\leq
\gamma_2$, if $\wt\gamma_1$ does not intersect the component of $\wt
S-\wt\gamma_2$ whose boundary orientation is opposite that of the
orientation of $\wt\gamma_2$.  We make this definition to
distinguish from the notion of $[\gamma_1]$ being to the left of
$[\gamma_2]$, where $[\gamma_i]$, $i=1,2$, is the isotopy class of
$\gamma_i$ rel endpoints. Having said that, {\em in the rest of the
paper, ``to the left'' will always mean ``setwise to the left''}.
(The same definition can be made when only one of the $\gamma_i$ is
a properly embedded arc of $S$, and the other is an arc which just
starts at $x$.)

We are now ready to state the main proposition of this section.

\begin{prop} \label{prop: construction}
There exists a $1$-form $\beta$ on $S$ with $d\beta>0$ such that the
walls $W_{i,L}$, $W_{i,R}$, $i=1,\dots,n$, are integral curves of
$\ker \beta$, and satisfy the following properties:
\begin{enumerate}
\item Each wall contains exactly one singularity of $\beta$.  It is
an elliptic singularity, and is in the interior of the initial arc
or in the interior of the terminal arc.
\item $W_{i,L}$ (resp.\ $W_{i,R}$) is to the left (resp.\ to the
right) of the prong $P_i$.
\item $W_{i,L}$ (resp.\ $W_{i,R}$) is to the left of
$\psi'(W_{i-k,L})$ (resp.\ to the right of $\psi'(W_{i-k,R})$).
\item $\beta$ exits $\bdry S$ uniformly with respect to $\psi'$.
\item The initial and terminal arcs of $(\psi')^{-1}(W_{i,L})$ and
$(\psi')^{-1}(W_{i,R})$ are integral arcs of $\ker\beta$.
\item Suppose $\psi'$ maps an initial (resp.\ terminal) arc of
$(\psi')^{-1}(W_{i,L})$ to an initial (resp.\ terminal) arc of
$W_{i,L}$.  If both arcs contain elliptic singularities, then
$\psi'$ matches the germs of these elliptic singularities. The same
holds for $(\psi')^{-1}(W_{i,R})$ and $W_{i,R}$.
\end{enumerate}
\end{prop}

When comparing $W_{i,L}$ with $P_i$, we concatenate $P_i$ with a
small arc from $q_i(-\delta+\varepsilon)$ to $x_i$, and assume
$W_{i,L}$ and $P_i$ have the same initial points (and similarly for
$W_{i,R}$ and $P_i$).

The proof of Proposition~\ref{prop: construction} occupies the next
two subsections. In Subsection~\ref{subsection1} we construct the
$1$-form $\beta$, and in Subsection~\ref{subsection2} we verify the
properties satisfied by the walls.

\subsection{Construction of $\beta$} \label{subsection1} $\mbox{}$

\s\n {\bf Step 1.} (Construction of $\beta$ on $S''$.) The surface
$S''$ has corners and carries the nonsingular line field
$\F|_{S''}$.

\begin{claim}
The restriction of $\F$ to $S''$ is orientable.
\end{claim}

\begin{proof}
All the corners of $S''$ are convex, except those which are also
corners of $N(\bdry S)$.  Each concave corner $p_i(\pm \delta)$ is
always adjacent to a convex one of type $p_i(s)$, obtained by
removing the neighborhood of a prong; moreover, such an assignment
defines an injective map from the set of concave corners to the set
of convex corners.  Let $C$ be a connected component of $S''$. Since
the foliation $\F$ is either tangent to or transversal to every
smooth subarc of $\partial S''$, if we smooth $\partial C$, then for
every boundary component $c_i$ of $\partial C$, the degree of $\F$
along $\bdry c_i$ is
$$\deg (\F,\partial c_i )=\frac{1}{4} (\sharp \{ {\rm concave \
corners}\} -\sharp \{ {\rm convex \ corners} \})\leq 0.$$ Therefore
we see that $\chi (C) \geq 0$, and $C$ is either a disk or an
annulus. Moreover, in the case of an annulus, the degree of $\F$
must be zero on the two boundary components. The claim follows.
\end{proof}

From now on, fix an orientation of $\F|_{S''}$. The transversal
measure for $\F|_{S''}$ on $S''$ is now given by a closed $1$-form
$\nu$ which vanishes on $\F|_{S''}$ and satisfies
$\psi_*\nu=\lambda\nu$. The measure of an arc transversal to
$\F|_{S''}$ is given by the absolute value of the integral of $\nu$
along this arc. We stipulate that $\nu(Y)>0$ if $(X,Y)$ is an
oriented basis for $TS$ at a point and $X$ is tangent to and directs
$\F|_{S''}$.

The surface $S''$ can be covered by interiors of finitely many
Markov charts of $\F$ in $S-int(N(\bdry S))$.  (Here we are using
the topology induced from $S-int(N(\bdry S))$.) It suffices to
consider charts of type $R=[-1,1]\times [-\delta_0, \delta_0]$ and
$R'=([-1,1]\times [-\delta_0,\delta_0]) -([-{1\over 3},{1\over
3}]\times [-\delta_0,0))$, where $\delta_0>0$ is small and we use
coordinates $(x,y)$ for both types of charts. Here we take $\nu
=dy$; in particular, this means $\F =\{ dy=0\}$. Moreover, we
require $\{ \pm 1\} \times [-\delta_0,\delta_0]$ to be subarcs of
vertical arcs of $\partial S'$ for both $R$ and $R'$; $\{\pm {1\over
3}\}\times [-\delta_0,0]$ to be subarcs of vertical arcs of $\bdry
(N(\bdry S))$ for $R'$; and $[-{1\over 3},{1\over 3}]\times \{0\}$
to be a horizontal arc of $\bdry (N(\bdry S))$ for $R'$.

On each chart $U=R$ or $R'$, let $f_U$ be a function $R\rightarrow
\R^{\geq 0}$  (resp.\ $R'\rightarrow \R^{\geq 0}$) which satisfies:
\be
\item  $f_U=0$ on $[-1,1] \times \{ -\delta_0,\delta_0\}$
(resp.\ $([-1,1]\times\{-\delta_0,\delta_0\})\cap R'$).
\item $\frac{\partial f_U}{\partial x} (x,y)>0$ for $y\in
(-\delta_0,\delta_0)$. \ee If we sum the forms $f_U dy$ over all the
charts $U$, we obtain a form $\nu'$ on $S''$ with $d\nu'
>0$. Now consider the form $\beta=\nu+\varepsilon_0\nu'=F\nu$,
where $\varepsilon_0>0$. We have $d\beta>0$, since $d\nu=0$ and
$d\nu'>0$. Observe that $\beta$ is defined in a slight enlargement
of $S''$, and the desired $\beta$ is $\beta|_{S''}$.

\s\n {\bf Step 2.} (Extension of $\beta$ to $S$ in the absence of
interior singularities.)  Suppose there are no interior
singularities.  We first state and prove a useful lemma.

\begin{lemma} \label{extension to rectangle}
Consider the rectangle $R=[-1,1]\times[-\delta_1,\delta_1]$ with
coordinates $(x,y)$. Let $\beta= Fdy$, $F>0$, be the germ of a
$1$-form on $\bdry R$ which satisfies ${\bdry F\over \bdry x}>0$.
Then $\beta$ admits an extension to $R$ with the properties that
$d\beta>0$ and $\beta({\bdry\over \bdry y})>0$ if and only if
$\int_{\bdry R} \beta>0$.
\end{lemma}

\begin{proof}
The condition $\int_{\bdry R}\beta>0$ is clearly necessary by
Stokes' theorem.  We check that it is sufficient.  Let
$\phi_{-1}:[-\delta_1,\delta_1]\rightarrow [-\delta_1,\delta_1]$ be
an orientation-preserving diffeomorphism for which
$\phi_{-1}^*\beta(-1,y)=c_{-1}dy$, where $c_{-1}>0$ is a constant.
Similarly define $\phi_{1}$ so that $\phi_{1}^*\beta(1,y)=c_1dy$
with $c_1>c_{-1}$.  Take an isotopy $\phi_x$ rel endpoints between
$\phi_{-1}$ and $\phi_{1}$, and define $\phi(x,y)=(x,\phi_x(y))$.
Now, there exists a $1$-form $Gdy$ near $\bdry R$ which agrees with
$\phi^*(\beta)$.  We may extend $G$ to all of $R$ with the property
that ${\bdry G\over \bdry x}>0$. Now let $\beta=\phi_*(Gdy)$.
\end{proof}

Next consider the walls $W_{i,L}$, $W_{i,R}$, $i=1,\dots,n$. The
walls are pairwise disjoint, and are disjoint from $P_{i'}'$ as well
as the portions of $P_{i'}\cap N(\bdry S)$ of type
$p_{i'}(0)q_{i'}(0)$, for all $i'$. Moreover, we may assume that
$\psi'$ leaves the union of the initial arcs of $W_{i,L}$ (resp.\
$W_{i,R}$), $i=1,\dots,n$, invariant.

\s\n {\bf Step 2A.} Consider the region $B_{i-1}$ of $N(\bdry S)$
which is bounded by the initial arcs of $W_{i-1,R}$ and $W_{i,L}$,
the arc $b_{i-1}$, and an arc of $\bdry S$. Assume without loss of
generality that, with respect to the orientation on $\F|_{S''}$,
$b_{i-1}$ is oriented from $a_{i-1}$ to $a_i$. Our strategy is as
follows: Start with $\F'|_{S''}=\F|_{S''}$, extend $\F'$ to a
singular Morse-Smale characteristic foliation on $B_{i-1}$, and
construct a $1$-form $\beta$ with $d\beta>0$ so that
$\ker\beta=\F'$.

Consider a characteristic foliation $\F'$ on a small neighborhood of
$B_{i-1}$ with the following properties (see Figure~\ref{walls}):
\begin{itemize}
\item The initial arc of $W_{i-1,R}$ is a nonsingular leaf which points
out of $\bdry S$.
\item The initial arc of $W_{i,L}$ is an integral curve with one
positive elliptic singularity $e_{i-1}$ on it.  We place the
$e_{i-1}$ so that Property (1) of Proposition~\ref{prop:
construction} holds.
\item $int(B_{i-1})$ contains a positive hyperbolic singularity $h_{i-1}$,
where the stable separatrices come from $e_{i-1}$ and from
$p_{i-1}(\delta-\varepsilon_2)$ on $a_{i-1}$, and the unstable
separatrices go to $\bdry S$ and $p_i(-\delta+\varepsilon_3)$ on
$a_i$.
\item By making $\varepsilon_2>0$ small enough, we have
\begin{equation}\label{crazyrectangle}
\displaystyle\int_{[\delta-\varepsilon_2,\delta]} p_{i-1}^*\beta <
\int_{[-\delta+\varepsilon_3,-\delta]} p_i^*\beta.
\end{equation}
\item The foliation points out of $S$ along $\bdry S$.
\end{itemize}

\begin{figure}[ht]
\vskip.1in
\begin{overpic}[height=1.7in]{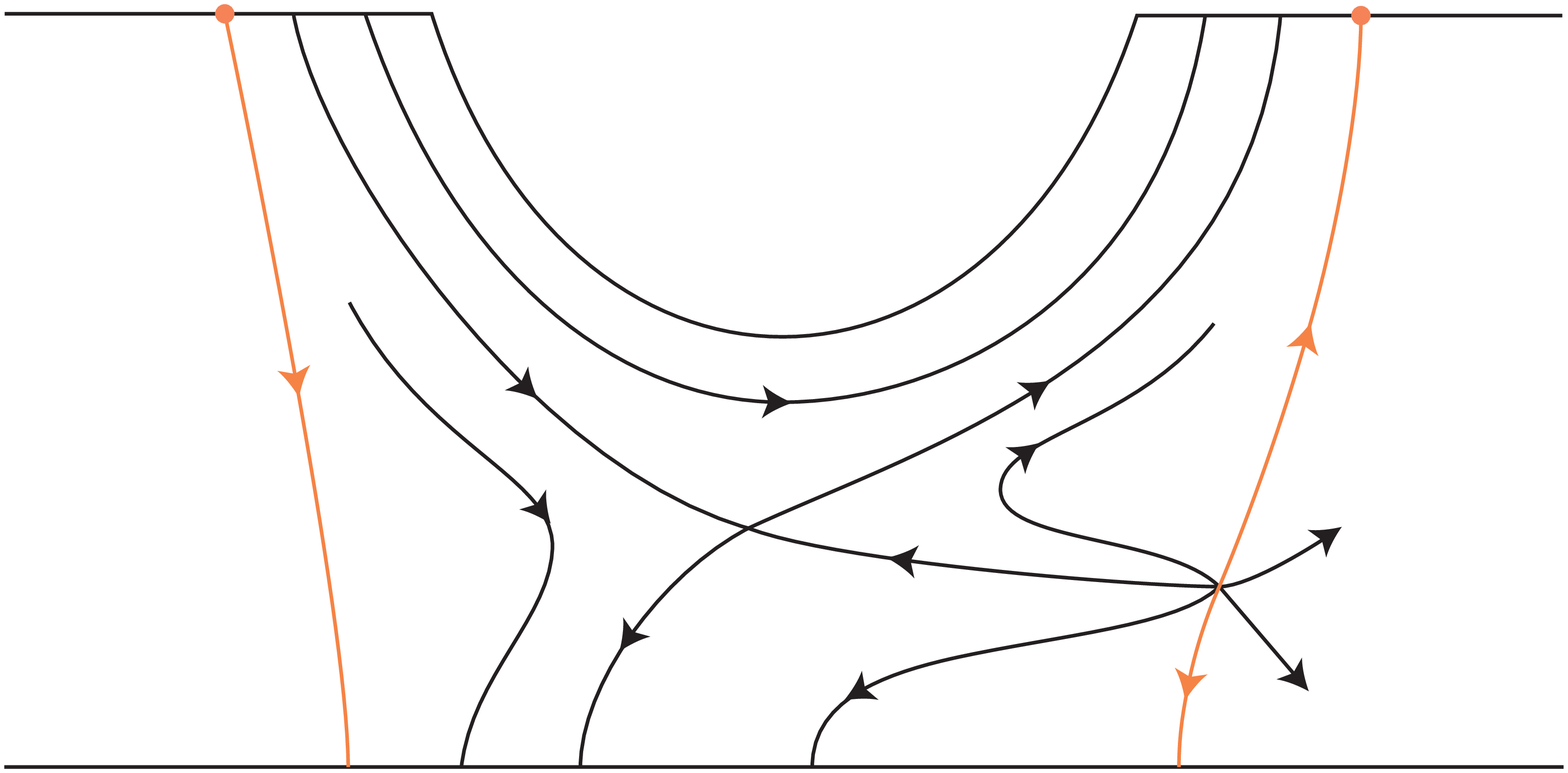}
\put(7,20){\tiny $W_{i-1,R}$} \put(85,25){\tiny $W_{i,L}$}
\put(0,50.5){\tiny $a_{i-1}$} \put(77,50.5){\tiny $a_i$}
\put(47.7,30.15){\tiny $b_{i-1}$} \put(12,50.5){\tiny
$p_{i-1}(\delta-\varepsilon)$} \put(85,50.5){\tiny
$p_i(-\delta+\varepsilon)$} \put(0,2){\tiny $\bdry S$}
\put(46.3,11.8){\tiny $h_{i-1}$} \put(81,10.4){\tiny $e_{i-1}$}
\end{overpic}
\caption{The kernel of the $1$-form $\beta$ in the region $B_{i-1}$
between the walls $W_{i-1,R}$ and $W_{i,L}$. The walls are orange
arcs.} \label{walls}
\end{figure}
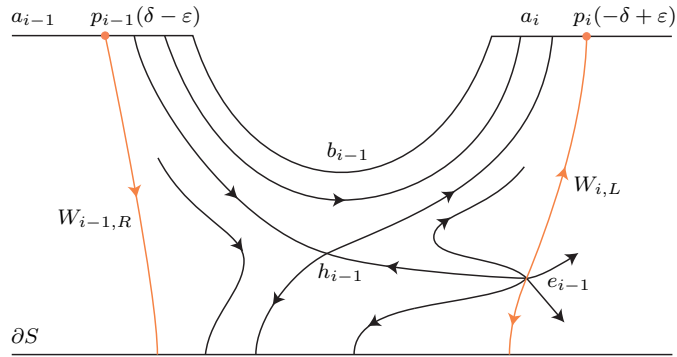

We now explain how to extend $\beta$ to $B_{i-1}$ so that $\ker
\beta$ is the above characteristic foliation and $d\beta >0$. This
procedure follows Giroux's construction in
\cite[Proposition~2.6]{Gi2}. The form $\beta$ can be defined in a
neighborhood of $h_{i-1}$ and $e_{i-1}$, and, provided $\beta$ is
sufficiently large on the boundary of the neighborhood of $h_{i-1}$,
the extension to small neighborhoods of the separatrices and initial
arcs of $W_{i-1,R}$ and $W_{i,L}$ is immediate. The complementary
regions are all foliated rectangles. All but one have one vertical
edge either on $\bdry S$ or on $\bdry N(e_{i-1})$, and easily
satisfy the conditions of Lemma~\ref{extension to rectangle}. The
condition that $\beta$ exit uniformly with respect to $\psi'$ is
also easily met, on the portion that is defined.  The remaining
component is a rectangle whose vertical edges are small retractions
of $p_{i-1}([\delta-\varepsilon_2,\delta])$ and
$p_i([-\delta+\varepsilon_3,-\delta])$. Observe that the conditions
of Lemma~\ref{extension to rectangle} also hold for the remaining
rectangle, thanks to Equation~\ref{crazyrectangle}. Now we can apply
Lemma~\ref{extension to rectangle}, and extend $\beta$ to all the
rectangles.

\s\n {\bf Step 2B.} Next, for each $i$, we extend the horizontal
arcs of $\bdry (N(P_i'))$ inside $N(\bdry S)$ by geodesic arcs to
$\bdry S$. More precisely, let $p_i(s_1)$, $p_i(s_2)$,
$p_{i'}(s_3)$, $p_{i'}(s_4)$ be the four corners of $N(P_i')$, where
$s_1<s_2$ and $s_3<s_4$.
Extend the horizontal arc $p_i(s_1)p_{i'}(s_4)$ by geodesic arcs
$p_i(s_1)q_i(s_1)$, $p_{i'}(s_4)q_{i'}(s_4)$ to obtain the arc
$d_{i,L}$ which is properly embedded in $S$. Similarly, extend
$p_i(s_2)p_{i'}(s_3)$ by geodesic arcs $p_i(s_2)q_i(s_2)$,
$p_{i'}(s_3)q_{i'}(s_3)$ to obtain $d_{i,R}$. Let $A_i\subset S$ be
the strip which lies between $d_{i,L}$ and $d_{i,R}$.  We extend
$\F'=\ker \beta$ to $A_i$. There are two cases:

\s\n (1) The orientations of $\F|_{S''}$ agree on $d_{i,L}\cap
\bdry(N(P_i'))$ and $d_{i,R}\cap \bdry (N(P_i'))$. This situation is
given in Figure~\ref{model3}.
\begin{figure}[ht]
\vskip.1in
\begin{overpic}[height=2.4in]{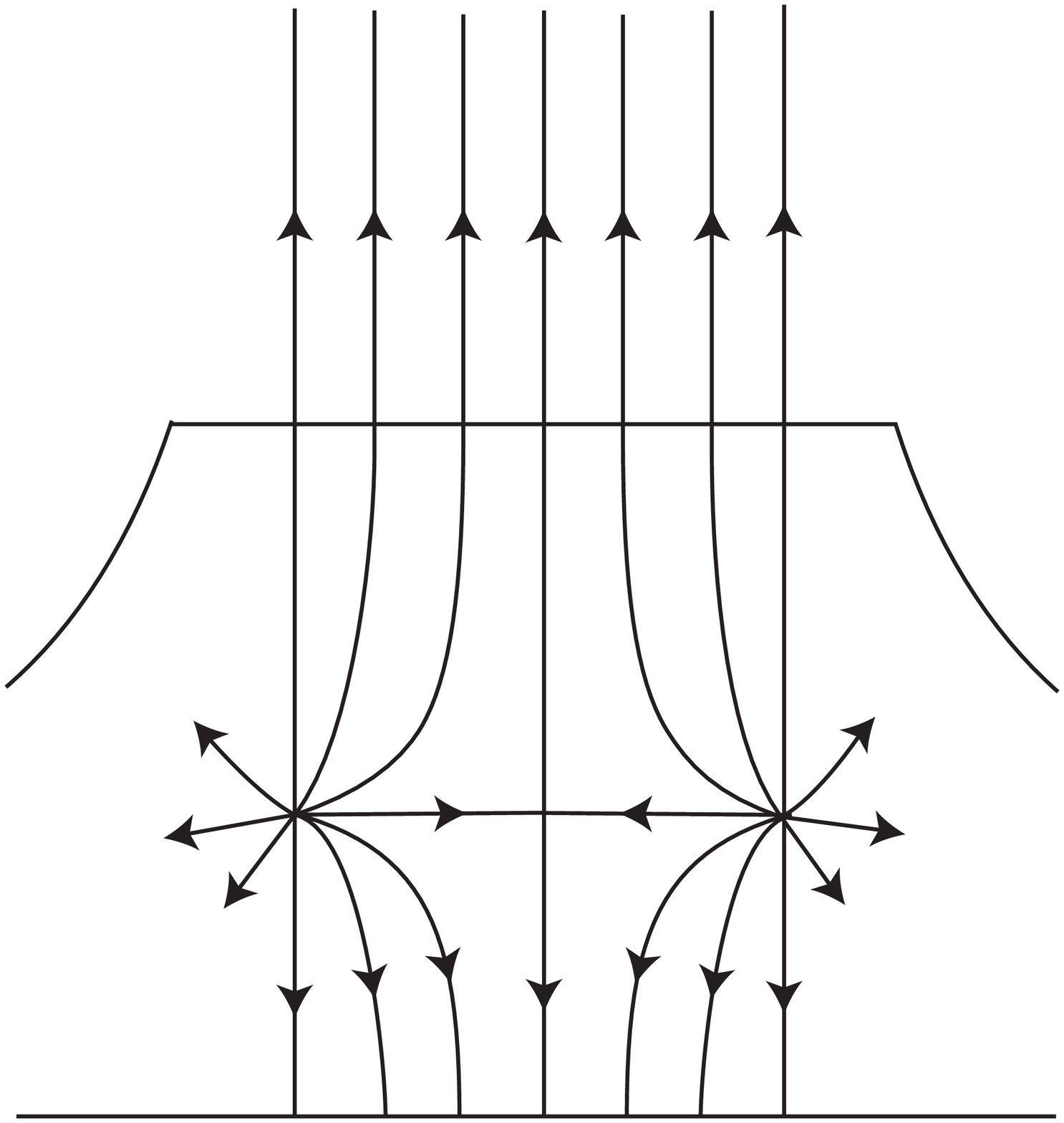}
\put(0,2){\tiny $\bdry S$} \put(15.5,80){\tiny $d_{i,L}$}
\put(72.5,80){\tiny $d_{i,R}$}
\end{overpic}
\caption{Construction of $\F'$ near $R_{\mbox{\tiny from}}$.}
\label{model3}
\end{figure}
In this case, we extend the characteristic foliation $\F'$ to
$N(P_i')$ so that it coincides with $\F$ on $N(P_i')$ and the
orientation agrees with that of $\F|_{S''}$ along $d_{i,L}\cap
\bdry(N(P_i'))$ and $d_{i,R}\cap \bdry (N(P_i'))$. There are two
remaining rectangles $R_{\mbox{\tiny to}}$ and $R_{\mbox{\tiny
from}}$ in $A_i$ to be foliated.  The rectangle $R_{\mbox{\tiny
to}}$ (resp.\ $R_{\mbox{\tiny from}}$) has a vertical edge in common
with $N(P'_i)$, along which $\F'$ exits (resp.\ enters) $N(P_i')$.
On $R_{\mbox{\tiny to}}$, the foliation $\F'$ consists of geodesic
arcs from $p_{i'}(s)$ to $q_{i'}(s)$, for $s\in[s_3,s_4]$, by
switching $i$, $i'$ if necessary. On $R_{\mbox{\tiny from}}$, we
place a positive elliptic singularity on each of $p_i(s_1)q_i(s_1)$
and $p_i(s_2)q_i(s_2)$ so that the two geodesics become integral
curves. Next we place a positive hyperbolic singularity in the
interior of $R_{\mbox{\tiny from}}$, so that both stable
separatrices come from the two elliptic points and the unstable
separatrices exit through $\bdry S$ and $\bdry N(P_i')$. Also, we
arrange so that $\F'$ exits from $S$ along $\bdry S\cap A_i$.  The
extension of $\beta$ to $A_i$ as a $1$-form with kernel $\F'$
subject to the condition $d\beta>0$ follows from
Lemma~\ref{extension to rectangle} and the considerations in Step
2A.

\s\n (2) The orientations of $\F|_{S''}$ on $d_{i,L}\cap
\bdry(N(P_i'))$ and $d_{i,R}\cap \bdry (N(P_i'))$ are opposite.
Without loss of generality assume that $\F'$ is oriented from
$p_i(s_1)$ to $p_{i'}(s_4)$.  Place an elliptic singularity between
$p_i(s_1)$ and $q_i(s_1)$, and between $p_{i'}(s_3)$ and
$q_{i'}(s_3)$, so that $d_{i,L}$ and $d_{i,R}$ are integral curves
of $\F'$.
\begin{figure}[ht] \vskip.1in
\begin{overpic}[height=2.4in]{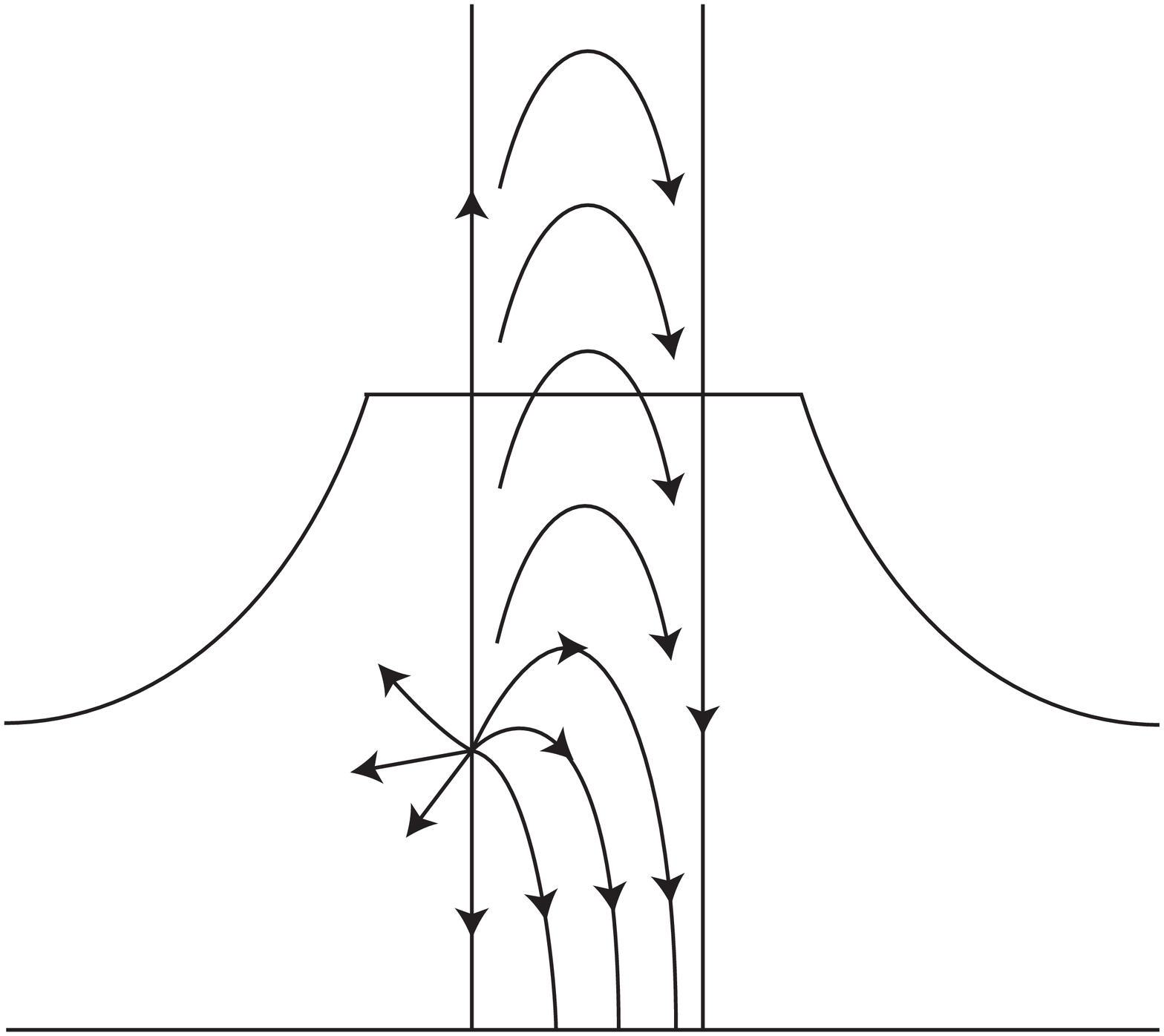}
\end{overpic}
\caption{} \label{model1}
\end{figure}
Next, place a hyperbolic singularity in the interior of one of the
components of $A_i-N(P_i')$. Its stable separatrices come from the
elliptic singularities on $d_{i,L}$ and $d_{i,R}$, and its unstable
separatrices exit $S$ along the two distinct components of $A_i\cap
\bdry S$. We can extend the foliation $\F'$ to all of $A_i$ without
adding any extra singularities and so that $\F'|_{N(P_i)}$ is a Reeb
component. See Figures~\ref{model1} and~\ref{model2} for the ends of
$A_i$. Finally, extend $\beta$ to $A_i$ as before.
\begin{figure}[ht]
\vskip.1in
\begin{overpic}[height=2.4in]{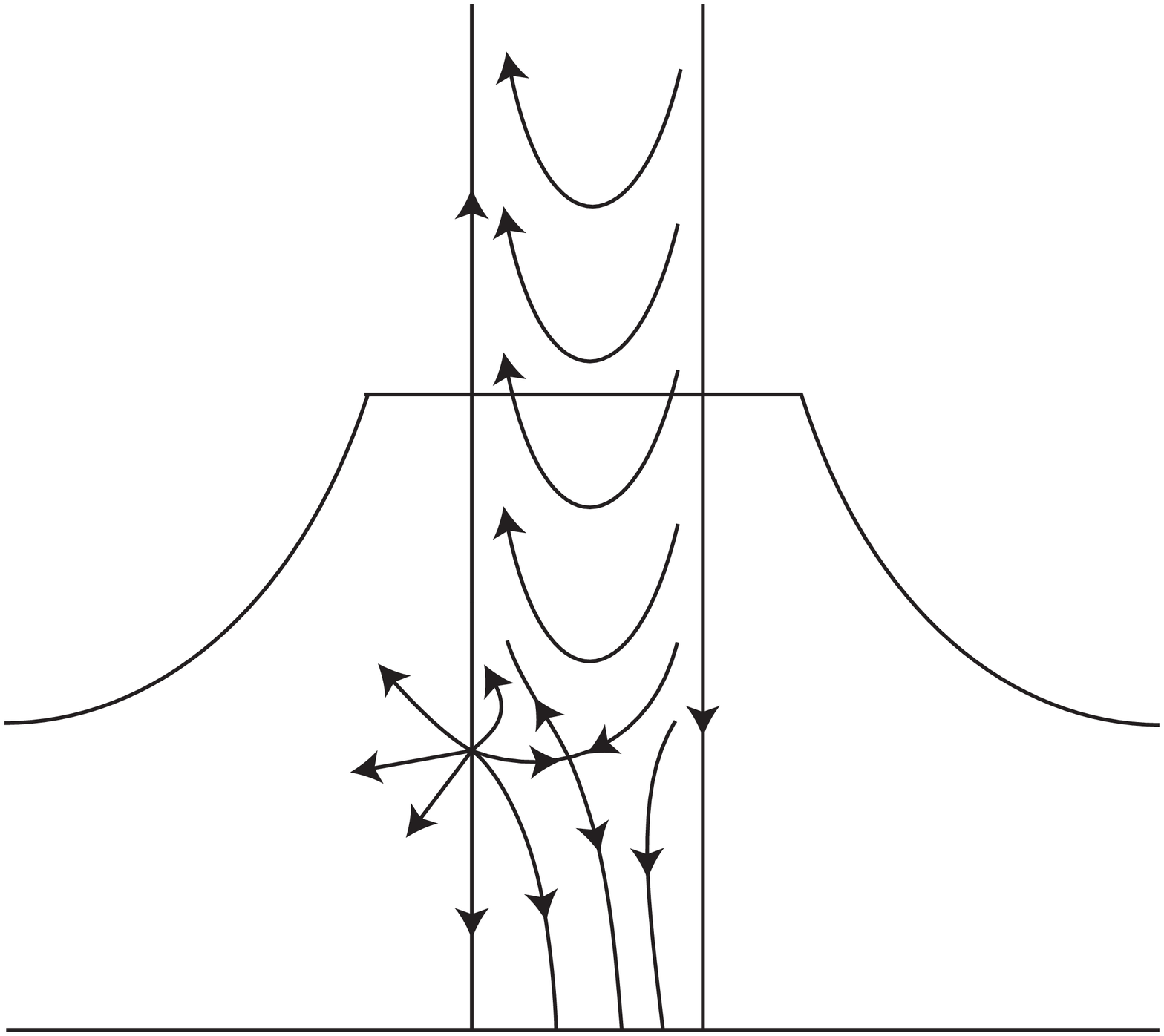}
\end{overpic}
\caption{} \label{model2}
\end{figure}

\s\n {\bf Step 2C.} After Steps 2A and 2B, we are left with
rectangles $\mathcal{R}$ in $N(\bdry S)$ whose vertical edges are on
$a_i$ and on $\bdry S$ and whose horizontal edges are of type
$p(s)q(s)$. We subdivide the rectangles by adding horizontal edges
so that the $W_{i,L}$, $W_{i,R}$ and $(\psi')^{-1}(W_{i,L})$,
$(\psi')^{-1}(W_{i,R})$ become integral arcs of $\ker \beta$, i.e.,
Property (5) of Proposition~\ref{prop: construction} is satisfied.
Observe that the initial arcs of $W_{i,L}$ and $W_{i,R}$ are already
integral arcs by Step 2A, and the initial arcs of
$(\psi')^{-1}(W_{i,L})$, $(\psi')^{-1}(W_{i,R})$ are the same as the
initial arcs of $W_{i-k,L}$ and $W_{i-k,R}$ by the definition of
$\psi'$. Hence we only consider the terminal arcs.

Let $p_{\phi_L(i)}(s_{i,L})$, $p_{\phi_R(i)}(s_{i,R})$,
$(\psi')^{-1}(p_{\phi_L(i)}(s_{i,L}))$, and
$(\psi')^{-1}(p_{\phi_R(i)}(s_{i,R}))$ be the initial points of the
terminal arcs of $W_{i,L}$, $W_{i,R}$, $(\psi')^{-1}(W_{i,L})$, and
$(\psi')^{-1}(W_{i,R})$. Here $\phi_L$, $\phi_R$ are some functions.
By construction, $p_{\phi_L(i)}(s_{i,L})$ is on the boundary of some
rectangle $\mathcal{R}$. If the orientation of $\F'$ at
$p_{\phi_L(i)}(s_{i,L})$ points into $\mathcal{R}$, then extend
$\F'$ and $\beta$ so that $\F'$ is tangent to and nonsingular along
the horizontal edge $p_{\phi_L(i)}(s_{i,L})q_{\phi_L(i)}(s_{i,L})$.
If the orientation points out of $\mathcal{R}$, then extend $\F'$
and $\beta$ so that $p_{\phi_L(i)}(s_{i,L})q_{\phi_L(i)}(s_{i,L})$
is an integral curve containing an elliptic singularity.  Next, if
$p_{\phi_L(i)-k}(s')=(\psi')^{-1}(p_{\phi_L(i)}(s_{i,L}))$ is on the
boundary of some rectangle $\mathcal{R}$, then $\F'$ and $\beta'$
can be extended similarly. If $p_{\phi_L(i)-k}(s')$ is inside some
$A_{i''}$, then let $W$ be a properly embedded arc in $S$ obtained
by concatenating $q_{\phi_L(i)-k}(s')p_{\phi_L(i)-k}(s')$, the leaf
of $\F|_{S-int(N(\bdry S))}$ through $p_{\phi_L(i)-k}(s')$, and a
terminal arc of type $p_{i''}(s'')q_{i''}(s'')$.  We then modify
$\F'$ by erasing $\F'|_{A_{i''}}$, extending $\F'$ to $W$ so that
$W$ is an integral curve which contains an elliptic singularity,
splitting $A_{i''}$ into two annuli along $W$, and applying the
procedure in Step~2B to each of the two annuli. The cases of
$p_{\phi_R(i)}(s_{i,R})$ and $\psi'(p_{\phi_R(i)}(s_{i,R}))$ are
treated similarly.

Finally, the extensions of $\F'$ and $\beta$ to the interiors of the
rectangles are identical to the extensions to $R_{\mbox{\tiny to}}$
and $R_{\mbox{\tiny from}}$ in Case (1) of Step 2B.

\s We remark that the extension of $\beta$ can be chosen so that
$\beta$ exits $\bdry S$ uniformly with respect to $\psi'$.

\s\n {\bf Step 3.} (Extension of $\beta$ to $S$ in the presence of
interior singularities.) We now explain how to extend $\beta$ to
$N(y_j)$. Let us denote the vertical boundary components of $N(y_j)$
by $c_1,\dots, c_{m_j}$ and the horizontal components by
$d_1,\dots,d_{m_j}$ (both ordered in a counterclockwise manner),
where the prong $Q_{jl}$ is between $d_l$ and $d_{l+1}$ and
intersects $c_l$. Here we orient $d_l$ by $\F'$. For a fixed $j$,
extend $d_l$ by geodesics to $\bdry S$ as before, and denote them by
$d'_l$.  Now let $C_j\supset N(y_j)$ be the subsurface of $S$
bounded by $d'_1,\dots,d'_{m_j}$. Refer to
Figure~\ref{singular-points}.
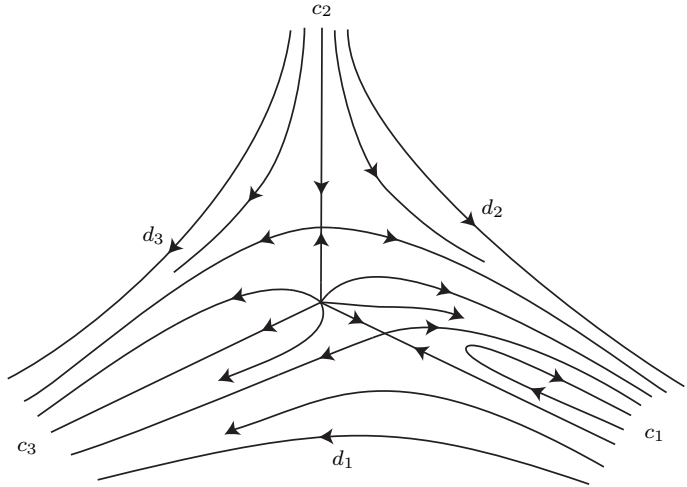
\begin{figure}[ht]
\vskip.15in
\begin{overpic}[height=2.4in]{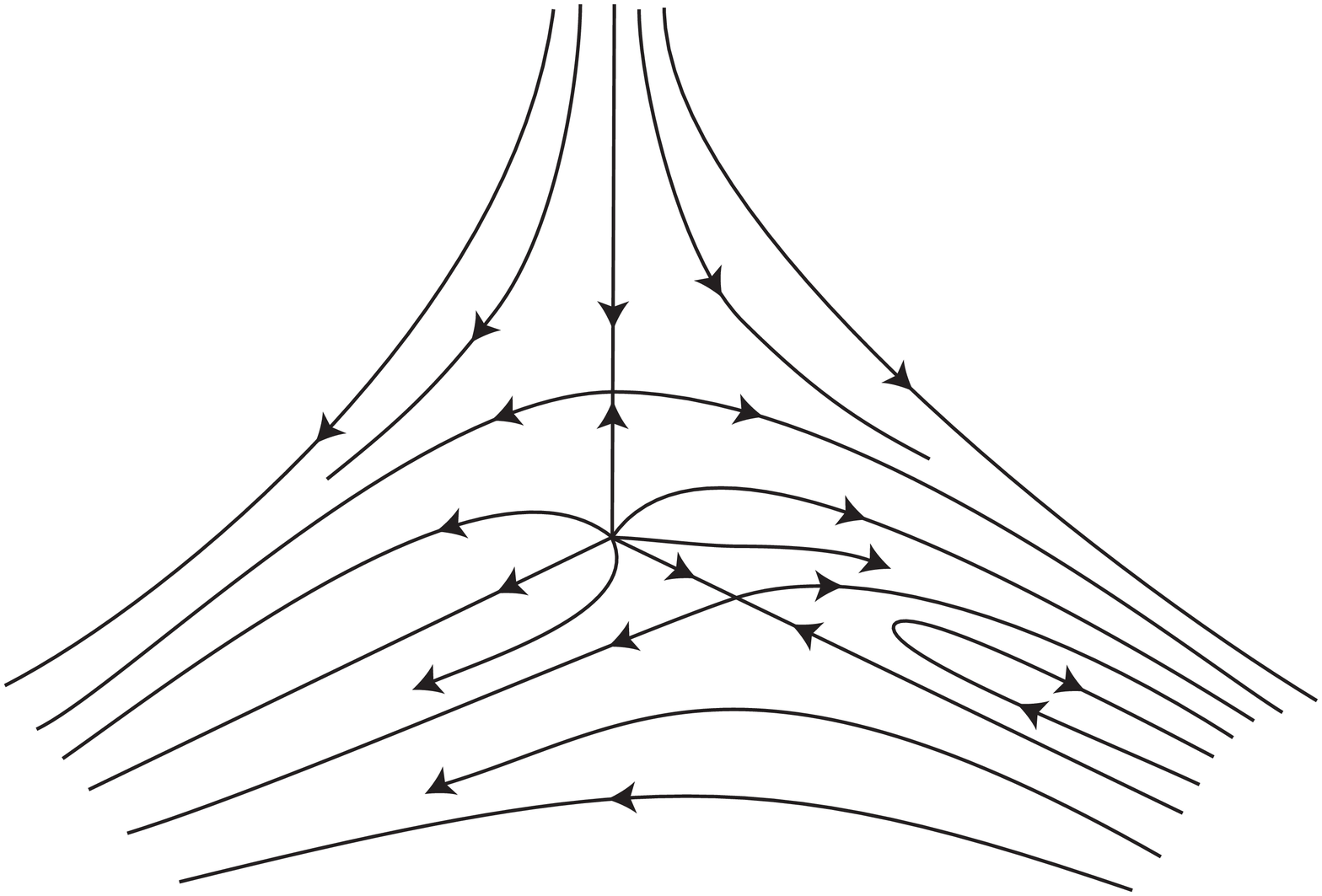}
\put(48,3){\tiny $d_1$} \put(70,40){\tiny $d_2$} \put(20,36){\tiny
$d_3$} \put(94,7) {\tiny $c_1$} \put(45,69.6){\tiny $c_2$}
\put(1.5,5){\tiny $c_3$}
\end{overpic}
\caption{Description of $\F'$ on $N(y_j)$.} \label{singular-points}
\end{figure}

We extend $\F'$ to $C_j$ as follows: Place an elliptic singularity
at $y_j$. Next, place a hyperbolic singularity on the prong
$Q_{jl}'$ emanating from $y_j$, if and only if at least one of
$d_l$, $d_{l+1}$ enters the region $S-int(N(\bdry S))$ along $c_l$.
In this case, the prong $Q_{jl}'$ is contained in the union of the
stable separatrices.  If both $d_l$ and $d_{l+1}$ enter
$S-int(N(\bdry S))$ along $c_l$, then the unstable separatrices exit
$N(y_j)$ along $c_{l-1}$ and $c_{l+1}$.  Otherwise, one unstable
separatrix exits along $c_l$ and the other exits along $c_{l-1}$
(resp.\ $c_{l+1}$) if $d_l$ (resp.\ $d_{l+1}$) enters $S-int(N(\bdry
S))$ along $c_l$. Finally, we complete $\F'$ and $\beta$ on $N(y_j)$
without adding extra singularities, and then extend to $C_j$ using
the models of $R_{\mbox{\tiny to}}$ and $R_{\mbox{\tiny from}}$
(Figure~\ref{model3}) from Case (1) of Step 2B, and
Figure~\ref{model1} from Case (2) of Step 2B.

As in Step~2C, if there is an initial point of a terminal arc of
$(\psi')^{-1}(W_{i,L})$ or $(\psi')^{-1}(W_{i,R})$ which lies in
$N(y_j)$, then we may need to insert extra arcs of type $W$ and redo
the construction of $\F'$ and $\beta$.

\s This completes the construction of $\beta$ on $S$.

\subsection{Verification of the properties.} \label{subsection2}
In this subsection we prove Properties~(1)--(6) of
Proposition~\ref{prop: construction}. Properties (1) and (4)--(6)
are clear from the construction.

\s\n (2) We compare $W_{i,L}$ and $P_i$. The wall $W_{i,L}$ is
initially to the left of $P_i$.  (More precisely,
$p_i(-\delta+\varepsilon)q_i(-\delta+\varepsilon)$ is to the left of
$p_i(0)q_i(0)$.) On $S-N(\bdry S)$, $W_{i,L}$ and $P_i'$ are leaves
of $\F$, and they do not cross. If there is some prong $P_j'$ or
$Q_{jk}'$ that intersects $p_i([-\delta+\varepsilon,0])$, then
$W_{i,L}$ and $P_i$ bifurcate in the universal cover $\wt{S}$ and
never reintersect. Otherwise, $W_{i,L}\cap (S-N(\bdry S))$ and
$P_i\cap (S-N(\bdry S))$ are parallel paths in $\wt{S}$. Let
$p_{i'}(s)$ be the ``other'' endpoint of $W_{i,L}\cap (S-N(\bdry
S))$, i.e., the one that is not $p_i(-\delta+\varepsilon)$.

If $s>0$, then we claim that the prong $P_{i'}'$ is between
$W_{i,L}$ and $P_i'$. Indeed, first observe that the transverse
distance between $W_{i,L}$ and $P_i'$ is $\delta-\varepsilon$.  Now,
in Section~\ref{matisse}, $\varepsilon>0$ was defined so that all
the leaves of $\F$ which start from
$p_i([-\delta,-\delta+\varepsilon])$ exit together along some
$a_{i'}$ and also avoid $P'_{i'}$. Hence $s\leq \delta-\varepsilon$
and $s\not=0$. This means that $P_i'$ intersects $a_{i'}$ at
$p_{i'}(s')$ with $s'<0$, and the prong $P_{i'}'$ is between
$W_{i,L}$ and $P_i'$.

Therefore $s<0$, and $P_i$ continues to the right while $W_{i,L}$
enters $N(\bdry S)$ and exits along $\bdry S$.

\s\n (3) Assume without loss of generality that $\psi(x_i)=x_i$. To
compare $W_{i,L}$ and $\psi'(W_{i,L})$, we first compare
$W=W_{i,L}\cap (S-N(\bdry S))$ and $\psi(W)$. Here, $\psi(W)$ and
$\psi'(W)=\psi'(W_{i,L})\cap (S-N(\bdry S))$ agree outside a
neighborhood of $N(\bdry S)$. The initial point of $W$ is
$p_i(-\delta+\varepsilon)$ and the initial point of $\psi(W)$,
projected to $a_i$ along $\F$, is $p_i({1\over
\lambda}(-\delta+\varepsilon))$. As in (2), if there is a prong
$P_j'$ or $Q_{jk}'$ that intersects
$p_i([-\delta+\varepsilon,{1\over \lambda}(-\delta+\varepsilon)])$,
then $W$ and $\psi(W)$ bifurcate in $\wt{S}$.  Hence the same can be
said about $W$ and $\psi'(W)$.

Otherwise, let $p_{i'}(s)$ be other endpoint of $W$ as in (2).  If
$s>0$, then let $p_{i'}(s'')$ be the first intersection of $\psi(W)$
with $a_{i'}$. Observe that $\psi(W)$ is longer than $W$, with
respect to the transverse measure $\mu^u$, so $s''$ cannot be in the
interval $[-{1\over \lambda}(\delta-\varepsilon), {1\over
\lambda}(\delta-\varepsilon)]$ if $W$ and $\psi(W)$ fellow-travel.
If $s''\geq 0$, then we necessarily have $s''\in [0,{1\over
\lambda}(\delta-\varepsilon)]$, since the distance $s-s''=
(1-{1\over \lambda})(\delta-\varepsilon)$. This is a contradiction.
Therefore, $s''<0$, which indicates a bifurcation. Now suppose
$s<0$. Let us parametrize $W$ (resp.\ $\psi(W)$) by the
$\mu^u$-distance from the point $p_i(-\delta+\varepsilon)$ (resp.\
$p_i({1\over \lambda}(-\delta+\varepsilon))$). Then either $\psi(W)$
does not intersect $a_{i'}$ at time $\mu^u(W)$, or intersects
$a_{i'}$ at $p_{i'}(s'')$ at time $\mu^u(W)$, where $s''\in
[-\delta,{1\over \lambda}(-\delta+\varepsilon))$. The possibility
$s''\in [{1\over \lambda}(-\delta+\varepsilon),0]$ has already been
ruled out.  The wall $W_{i,L}$ enters $a_{i'}$ and exits along
$\bdry S$, whereas $\psi(W)$ is ``to the right'' of $\psi(a_{i'})$
and hence $\psi'(W)$ is pushed ``to the right'' of $b_{i'-1}$.

\s\n This concludes the proofs of Properties (2) and (3), and also
the proof of Proposition~\ref{prop: construction}.

\subsection{Calculation of $\Phi$}

First observe the following:

\begin{lemma} \label{lemma: all in a row}
$$P_{i-1}\leq W_{i-1,R}\leq (\psi')^{-1}(W_{i-1+k,R})\leq
(\psi')^{-2}(W_{i-1+2k,R})\leq \dots \hspace{1in}$$
$$\hspace{1in}\dots\leq(\psi')^{-2}(W_{i+2k,L}) \leq (\psi')^{-1}(W_{i+k,L})\leq
W_{i,L} \leq P_{i}.$$
\end{lemma}

Recall that $a\leq b$ means $a$ is to the left of $b$. Also,
$c={k\over n}$ is the fractional Dehn twist coefficient, where $n$
is the number of prongs.

\begin{proof}
By Proposition~\ref{prop: construction}, $P_{i-1}\leq W_{i-1,R}$ and
$\psi'(W_{i-1,R})\leq W_{i-1+k,R}$.  Hence $W_{i-1,R}\leq
(\psi')^{-1}(W_{i-1+k,R})$, and the first row of inequalities holds.
Similarly, the second row of inequalities holds. Next,
$W_{i-1,R}\leq W_{i,L}$. (Reason: $W_{i-1,R}$ is initially to the
left of $W_{i,L}$. In order for them to reintersect, $W_{i-1,R}\cap
(S-int(N(\bdry S)))$ and $W_{i,L}\cap(S-int(N(\bdry S)))$ must both
have endpoints on the same $a_{i'}$. This implies the existence of a
monogon, which is a contradiction.) Repeated application of
$(\psi')^{-1}$ gives $(\psi')^{-j}(W_{i-1+jk,R})\leq (\psi')^{-j}
(W_{i+jk,L})$.
\end{proof}

Consider a trajectory $Q$ of the type
\begin{equation}\label{eqn: Q}
Q=\gamma_1 ((\psi')^{-1}(\gamma_2))\dots
((\psi')^{-m+2}(\gamma_{m-1})) ((\psi')^{-m+1}(\gamma_m)).
\end{equation}
The trajectory $Q$ is said to be an {\em ideal trajectory} if, for
each $i$, $\gamma_i:[0,1]\rightarrow S$ is tangent to $\F'$ and does
not pass through a singular point, and
$\gamma_{i}(1)=(\psi')^{-1}(\gamma_{i+1}(0))$.

\begin{prop} \label{ideal trajectory}
If $Q$ is an ideal trajectory, then $\Phi(Q)=0$.
\end{prop}

\begin{proof}
Assume without loss of generality that $\psi(x_i)=x_i$.  (This is
just for ease of indexing.)

We argue by contradiction. Suppose that $\Phi(Q)<0$. (The case
$\Phi(Q)>0$ is similar.) We lift to the universal cover $\pi:\wt
S\rightarrow S$. A tilde placed over a curve will indicate an
appropriate lift to $\wt S$.  Then a lift $\wt Q$ of $Q$ in $\wt{S}$
must intersect two consecutive prongs $\wt{P}_i$ and $\wt{P}_{i-1}$
which start from the same component $L$ of $\bdry\wt{S}$, in that
order. Also choose a lift $\wt \psi':\wt S\rightarrow \wt S$ of
$\psi'$ which fixes $L$ pointwise.

Assume $\wt\gamma_1(0)$ is strictly to the right of the lift
$\wt{W}_{i,L}$, whose initial point is between the initial points of
$\wt{P}_i$ and $\wt{P}_{i-1}$ on $L$. (The modifier ``strictly''
means $\wt\gamma_1(0)$ is not on $\wt{W}_{i,L}$.) Then, since the
trajectory is ideal, $\wt\gamma_1(1)$ is also strictly to the right
of $\wt{W}_{i,L}$. Next, $(\wt\psi')^{-1}(\wt{W}_{i,L})\leq
\wt{W}_{i,L}$, so $\wt\gamma_1(1)=(\wt\psi')^{-1}(\wt\gamma_2(0))$
is strictly to the right of $(\wt\psi')^{-1}(\wt{W}_{i,L})$. Again,
since the trajectory is ideal, $(\wt\psi')^{-1}(\wt\gamma_2(1))$ is
strictly to the right of $(\wt\psi')^{-1}(\wt W_{i,L})$. Eventually
we prove that $(\wt\psi')^{-m+1}(\wt\gamma_m(1))$ is strictly to the
right of $(\wt\psi')^{-m+1}(\wt W_{i,L})$. Since $\wt P_{i-1}\leq
(\wt\psi')^{-m+1}(\wt W_{i,L})$ by Lemma~\ref{lemma: all in a row},
it follows that $\wt Q$ cannot cross from $\wt P_i$ to $\wt
P_{i-1}$, a contradiction.

An equivalent proof (the one we use in the general case) is to
consider the sequence $\wt\gamma_1,\dots,\wt\gamma_m$, where
$\wt\psi'(\wt\gamma_i(1))=\wt\gamma_{i+1}(0)$. First, $\wt\gamma_1$
is strictly to the right of $\wt W_{i,L}$.  Hence
$\wt\psi'(\wt\gamma_1)$ is strictly to the right of $\wt\psi'(\wt
W_{i,L})$ and also strictly to the right of $\wt W_{i,L}$. This
implies that $\wt\gamma_2$ is strictly to the right of $\wt
W_{i,L}$.  Repeating the procedure, $\wt\gamma_m$ is strictly to the
right of $\wt W_{i,L}$.  Shifting back by $(\wt\psi')^{-m+1}$,
$(\wt\psi')^{-m+1}(\wt\gamma_m(1))$ is strictly to the right of
$(\wt\psi')^{-m+1}(\wt W_{i,L})$, a contradiction.
\end{proof}

Let $\beta$ be the $1$-form on $S$ constructed in
Section~\ref{subsection1}, and let
$\alpha_{\varepsilon,\varepsilon'}$ and
$R_{\varepsilon,\varepsilon'}$ be the contact $1$-form and Reeb
vector field constructed in Section~\ref{subsub: N} using $\beta$.
Define a {\em genuine trajectory} $Q$ to be a concatenation of the
type given by Equation~\ref{eqn: Q}, where each $\gamma_i$ is the
image of a trajectory of $R_{\varepsilon,\varepsilon'}$ from $t=0$
to $t=1$, under the projection $\pi: S\times[0,1]\rightarrow S$ onto
the first factor.

\begin{prop} \label{prop: genuine trajectory}
Given $N\gg 0$, for sufficiently small $\varepsilon,\varepsilon'>0$,
any genuine trajectory $Q$ of $R_{\varepsilon,\varepsilon'}$ with
$m\leq N$ satisfies $\Phi(Q)=0$.
\end{prop}

During the proof, a {\em leaf} of a singular foliation is understood
to be a maximal integral submanifold which {\em does not contain a
singular point}.

\begin{proof}
Let $Q$ be a genuine trajectory. Suppose each
$\gamma_i:[0,1]\rightarrow S$ is parametrized by $t$. We prove that
a lift $\wt Q$ of $Q$ cannot cross $\wt P_i$ and $\wt P_{i-1}$, as
in Proposition~\ref{ideal trajectory}.

\s\n {\bf Case 1.} Suppose that $\psi'$ matches the germ of an
elliptic point on $(\psi')^{-1}(W_{i,L})$ to the germ of an elliptic
point on $W_{i,L}$. Recall that, by Property (6) of
Proposition~\ref{prop: construction}, the germs of the elliptic
singularities are matched by $\psi'$ if the initial arcs (or
terminal arcs) of $(\psi')^{-1}(W_{i,L})$ and $W_{i,L}$ both contain
elliptic singularities.

We first recall Lemma~\ref{lemma:uniform}. Let $U$ be a small
neighborhood of the singular set of $\ker\beta$. On
$(S-U)\times[0,{1\over 2}]$, given $\delta_0>0$ small, there exists
$\varepsilon'>0$, so that, at points where $Y_{\varepsilon'}$ is
nonzero, $|{Y_{\varepsilon'}\over |Y_{\varepsilon'}|}+Z|\leq
\delta_0$, where $Z$ is a unit vector field which directs $\ker
\beta$. On $S\times[{1\over 2},1]$, $Y_{\varepsilon'}$ directs
$\ker\beta$, at points where $Y_{\varepsilon'}$ is nonzero.

Let $U_{i,L}$ be the connected component of $U$ which, after
possibly shrinking $U$, satisfies the following:
\begin{itemize}
\item $U_{i,L}$ is a small disk which is centered at the
elliptic point of $W_{i,L}$.
\item $\ker\beta=\ker(\psi')_*\beta$ on $U_{i,L}$.
\end{itemize}
The second condition holds since $\psi'$ matches the germs of the
elliptic singularities of $(\psi')^{-1}(W_{i,L})$ and $W_{i,L}$.
This implies that an arc $\gamma:[0,1]\rightarrow U_{i,L}$ which is
tangent to $Y_{\varepsilon'}$ does not jump from one leaf of
$\ker\beta|_{U_{i,L}}$ to another leaf.

Next let $N_{i,L}=([0,1]\cup[2,3])\times[-\tau,\tau]\subset S-U$ be
a foliated neighborhood of $W_{i,L}\cap (S-U)$ with coordinates
$(x,y)$, so that $y=const$ are leaves of $\ker \beta$ and $y=0$ is
$W_{i,L}\cap (S-U)$.  See Figure~\ref{nbhd-of-wall}. Since the
lengths of leaves of $\ker\beta|_{N_{i,L}}$ are bounded, there
exists a constant $K$ (independent of $\varepsilon$, $\varepsilon'$)
so that any arc $\gamma:[0,1]\rightarrow S-U$ which is tangent to
$Y_{\varepsilon'}$ and passes through $\{y=y_0\}$ must be contained
in $\{\max(y_0-K\delta_0,-\tau)\leq y\leq \min(
y_0+K\delta_0,\tau)\}\cup (S-N_{i,L}-U)$.  We then take $\delta_0$
sufficiently small so that $NK\delta_0 <\tau$. In other words,
$N_{i,L}$ is the protective layer of $W_{i,L}$ which makes it hard
to cross $W_{i,L}$ when $\varepsilon'$ is small.

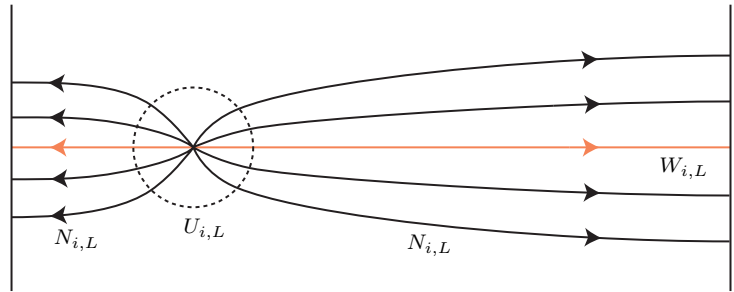
\begin{figure}[ht]
\vskip.15in
\begin{overpic}[height=1.5in]{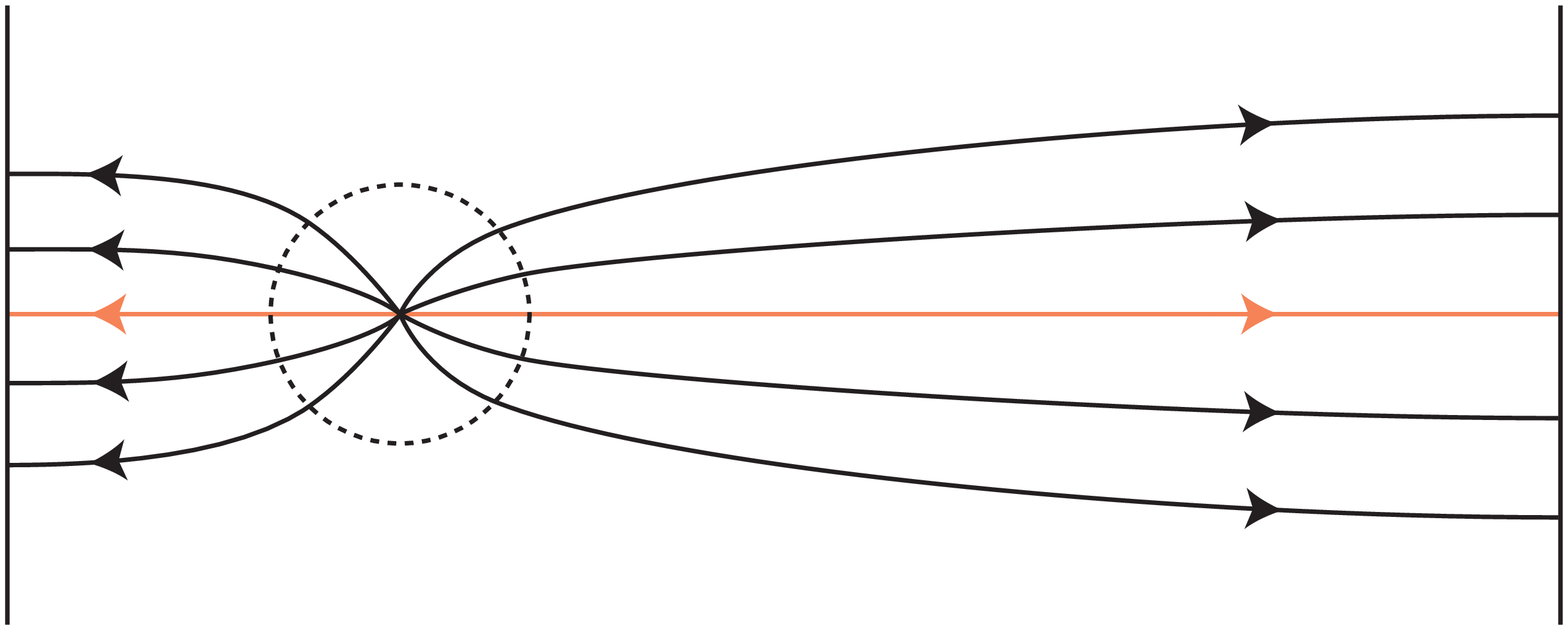}
\put(90,16.5){\tiny $W_{i,L}$} \put(24,8){\tiny $U_{i,L}$}
\put(6,6.5){\tiny $N_{i,L}$} \put(55,5.8){\tiny $N_{i,L}$}
\end{overpic}
\caption{The neighborhood of the wall $W_{i,L}$.}
\label{nbhd-of-wall}
\end{figure}

The initial point $\wt\gamma_1(0)$ must be strictly to the right of
$\wt P_i$ and also disjoint from $\wt N_{i,L}$ corresponding to $\wt
W_{i,L}$.  By the considerations of the previous two paragraphs,
$\wt\gamma_1([0,1])\cap \wt N_{i,L} \subset \{ \tau-K\delta_0\leq
y\}$, where we are taking $y>0$ to be to the right of $\wt W_{i,L}$.
Observe that, since $\wt\gamma_1|_{[1/2,1]}$ is tangent to
$\ker\beta$, it does not jump leaves.

Next, $\wt\gamma_2(0)=\wt\psi'(\wt\gamma_1(1))$. Since $\wt\psi'(\wt
W_{i-k,L})$ is to the right of $\wt W_{i,L}$, $\wt\gamma_2(0)$ is
further to the right of $\wt\gamma_1(1)$. In particular, if
$\wt\gamma_2(0)$ is in $\wt N_{i,L}$, then its $y$-coordinate is
greater than or equal to that of $\wt\gamma_1(1)$. Now apply the
same considerations to $\wt\gamma_2$ to obtain that
$\wt\gamma_2([0,1])\cap \wt N_{i,L}\subset \{ \tau-2K\delta_0\leq
y\}$. Continuing in this manner, we find that
$\wt\gamma_m([0,1])\cap \wt N_{i,L}\subset \{ \tau-mK\delta_0\leq  y
\}$. If $m\leq N$, then $\tau-mK\delta_0>0$. Hence the entire
trajectory $\wt Q$ must be to the right of $(\wt\psi')^{-m+1}(\wt
W_{i+(m-1)k,L})$, which, in turn, is to the right of $\wt P_{i-1}$.


\s\n {\bf Case 2.} Suppose that the initial arc of $W_{i,L}$
contains an elliptic point, whereas the initial arc of
$(\psi')^{-1}(W_{i,L})$ does not.

The difference with the previous case is that $\gamma_i$ can now
switch leaves inside $U_{i,L}$. Consider coordinates $(x,y)$ on
$U_{i,L}$ so that $W_{i,L}\cap U_{i,L}=\{y=0\}$, $W_{i,L}$ is
directed from $x>0$ to $x<0$, $\beta=xdy-ydx$, and
$f_{\varepsilon'}\psi'_*\beta= \varepsilon'(u(x)dy)$, with ${\bdry
u\over \bdry x}>0$ and $u>0$. We also suppose that being locally to
the right of $W_{i,L}$ means $y>0$. We compute that
$$f_{\varepsilon'}\psi'_*\beta-\beta= (-x+\varepsilon'u(x))dy+
ydx,$$ which is directed by $Y'=(-x+\varepsilon'u(x)){\bdry\over
\bdry x}-y{\bdry\over \bdry y}$. If $\varepsilon'$ is sufficiently
small, then the elliptic point of
$f_{\varepsilon'}\psi'_*\beta-\beta$ is contained inside $U_{i,L}$.
Suppose $\gamma_i|_{[0,1/2]}$ enters $U_{i,L}$ at $x<0$, $y>0$.
Comparing with the vector field $x{\bdry \over\bdry x}+y{\bdry \over
\bdry y}$ which directs $\ker\beta$, we see that
$\gamma_i|_{[1/2,1]}$ exits $U_{i,L}$ along a leaf of $\ker\beta$
which is further to the right of $W_{i,L}$, since
$-x+\varepsilon'u(x)>-x$.  If $x>0$ and $y>0$, then
$\gamma_i|_{[1/2,1]}$ will exit $U_{i,L}$ along a leaf of
$\ker\beta$ which is closer to $W_{i,L}$. However, this does not
present any problem, since initial arc of $W_{i,L}$ is an integral
arc of both $\beta$ and $(\psi')_*\beta$, and
Lemma~\ref{interpolation}(4) implies that $\gamma_i$ cannot cross
the wall $W_{i,L}$ along the initial arc.

The same proof holds when the terminal arc of $W_{i,L}$ contains an
elliptic point and the terminal arc of $(\psi')^{-1}(W_{i,L})$ does
not.
\end{proof}

\section{Nondegeneracy of Reeb vector fields} \label{section:
genericity}

In this section we collect some results on perturbing the contact
$1$-form to make the corresponding Reeb vector field nondegenerate.

We start with the proof of the following well-known fact (see for
example~\cite[Proposition~6.1]{HWZ6}).

\begin{lemma} \label{lemma: genericity}
Let $\alpha$ be a contact form on a closed $3$-manifold $M$. The set
of smooth functions $g:M\rightarrow (0,+\infty)$ for which the form
$g\alpha$ is nondegenerate is a dense subset of $C^{\infty}
(M,(0,+\infty))$ in the $C^\infty$-topology.
\end{lemma}

\begin{proof}
Fix $N >0$. We consider the set $G_N$ of functions $g :M\rightarrow
(0,+\infty )$ for which all the orbits of the Reeb vector field
$R_{g\alpha}$ of $g\alpha$ of period $\leq N$ are nondegenerate.

We first claim that $G_N$ is open. First observe that the union
$\Gamma_{\leq N}$ of all closed orbits of $R_\alpha$ of action $\leq
N$ is closed, and hence compact. Next, any closed orbit which comes
sufficiently close to an orbit $\gamma$, all of whose multiple
covers with period $\leq N$ are nondegenerate, is a long orbit,
i.e., has action $> N$. (A sequence of closed orbits of period $\leq
N$ converging to $\gamma$ implies that the return map for some
multiple cover of $\gamma$ before time $N$ has $1$ as eigenvalue.)
It is therefore possible to cover $\Gamma_{\leq N}$ by finitely many
sufficiently small disjoint solid tori $U_1 ,\dots ,U_k$, together
with security neighborhoods $V_1 ,\dots ,V_k$, so that, if $D(r)$ is
the disk of radius $r$ centered at the origin and $S^1=\R/\Z$, then:
\begin{enumerate}
\item For all $1\leq i\leq k$, $V_i \simeq S^1 \times D(2)$ and
$V_i \supset U_i \simeq S^1 \times int(D(1))$;
\item $S^1 \times \{ (0,0)\}$ is a nondegenerate periodic orbit of
$R_\alpha$ of action $\leq N$, and is the only periodic orbit of
action $\leq N$ inside $V_i$;
\item $R_\alpha$ is transverse to the foliation by horizontal disks on
$V_i$; \item For any $\theta\in S^1=\R/\Z$, all the orbits which
start from $\{ \theta\} \times D(1)$ stay inside $V_i$ at least for
an amount of time greater than $N+1$ and thus at least until the
first return to $\{\theta\} \times D(2)$, which occurs before time
$N+\varepsilon$.
\end{enumerate}
Any sufficiently small perturbation $R_{g\alpha}$ will still have a
single nondegenerate orbit, amongst those that start from
$\{\theta\}\times D(1)$. This follows from the transversality of the
graphs of the return maps with the diagonal of $(\{\theta\} \times
D(1)) \times (\{\theta\} \times D(2))$.  (The same considerations
also hold for multiple covers of the nondegenerate orbit of period
$\leq N$.) If $Z=M-\cup_{i=1}^k U_i$, then there is a small constant
$\tau>0$ so that no orbit $\delta :[0,N]\rightarrow M$ which is
strictly contained in $Z$ has two points $t_1,t_2$ sufficiently far
apart, so that $\delta(t_1)$ is within $\tau$-distance of
$\delta(t_2)$.  This implies that a sufficiently small perturbation
will not create any new periodic orbits of action $\leq N$. This
proves the claim.

Next we prove that $G_N$ is dense in $C^\infty (M,(0,+\infty))$. It
suffices to show that there exists a sequence of functions $f_n$
going to zero, so that $1+f_n\in G_N$ for all $n\in \N$. Let
$\{U_i\}_{i=1}^k$ be an open cover of $\Gamma_{\leq N}$ and $V_i$ be
the security neighborhood of $U_i$, satisfying (1), (3), (4) in the
previous paragraph. As in the previous paragraph, if we make a
sufficiently small perturbation of $\alpha$, we do not create orbits
of action less than $N$ outside $\cup_{1\leq i\leq k} U_i$.

We first modify $\alpha$ on $V_1$. Let $C>0$ so that the first
return occurs at a time strictly greater than $C$. We take an
embedding $[0,C] \times D({3\over 2}) \rightarrow V_1$ so that $\{
0\} \times D({3\over 2})$ is mapped to the horizontal disk $\{ 0\}
\times D({3\over 2}) \subset S^1 \times D(2)$, and so that, if $t$
is the coordinate of the $[0,C]$ factor, $\frac{\partial}{\partial
t} =R$.  We may also assume that $\alpha=dt+\beta$, where $\beta$ is
a $1$-form on $D({3\over 2})$ which is independent of $t$,
$\beta(0,0)=0$, and $\beta$ is small on $D({3\over 2})$ (by taking
the $V_i$ to be sufficiently small to start with). If $f$ is a
function on $M$ with support on $[0,C]\times D({3\over 2})$, then
$$d((1+f)\alpha)=\left(d_2f-{\bdry f\over \bdry t}\beta\right)\wedge
dt+(1+f)d\beta,$$ where $d_2$ means the exterior derivative in the
direction of $D({3\over 2})$. Provided $d_2f\gg {\bdry f\over \bdry
t}\beta$ on $[0,C]\times D(1)$ (which is the case for our specific
choice of $f$ below), $R_{(1+f)\alpha}$ is parallel to ${\bdry\over
\bdry t}+X$, which approximately satisfies
$$i_X d\beta=-{d_2f\over 1+f}.$$

Now we can look at a family of deformations corresponding to
functions which are zero outside of $[0,C] \times D({3\over 2})$ and
given by
$$f_{a,b} (x,y,t)=\chi_1 (t)\chi_2 (\sqrt{x^2 +y^2}) (ax +by)$$
inside, where $(t,x,y)$ are coordinates on $[0,C]\times D({3\over
2})$, $\chi_1$ and $\chi_2$ are positive cut-off functions such that
the $2$-jet of $\chi_1$ is $0$ at $t=0$ and $t=C$, $\chi_2
([0,1])=1$, and $\chi_2=0$ near ${3\over 2}$.
 This gives a sequence of families
$(F^l_{a,b} )_{1\leq l\leq p}$, with $1\leq p\leq  \frac{N}{C} +1$,
of $l$-th return maps which are  transversal to the diagonal in
$(\{0\} \times D(1)) \times (\{ 0\} \times D(2))$ as families: the
transversality is obtained by looking at derivatives of  $F^l_{a,b}$
in $(a,b)$ variables at $(a,b)=(0,0)$.

The transversality of the families $F^l_{a,b}$ implies that the
fixed points of $F^l_{a,b}$ are nondegenerate for a generic
$(a,b)\in \R^2$. Thus we can find a function $f_{a,b}$ as small as
we want so that the periodic orbits of period less than $N$ of
$R_{g_1\alpha}$, $g_1=1+f_{a,b}$, which are contained in $U_1$, are
nondegenerate. Next, we deform $g_1\alpha$ on $V_2$, by multiplying
a function $g_2$ which is very close to $1$, so that the orbits in
$U_1$ of period $\leq N$ remain nondegenerate and all the periodic
orbits inside $U_2$ of period $\leq N$ become nondegenerate. The
density of $G_N$ follows by induction.

Now the proof of the lemma follows by looking at $\cap_{N\in \N}
G_N$, which is a dense $G_\delta$-set.
\end{proof}

\begin{lemma} \label{lemma: perturbation}
Given $N\gg 0$, there is a $C^\infty$-small perturbation of
$\alpha_{\varepsilon,\varepsilon'}$ so that the perturbed Reeb
vector field $R$ satisfies the following:
\begin{enumerate}
\item All the closed orbits of $R$ with action $\leq N$ are nondegenerate;
\item The binding $\gamma_0$ and its multiple covers are nondegenerate
periodic orbits of $R$;
\item All other orbits of $R$ are positively transverse to the pages
of the open book;
\item All the orbits near $\gamma_0$ lie on the boundary of a solid
torus whose core curve is $\gamma_0$ and have irrational slope on
the solid torus.
\end{enumerate}
\end{lemma}

\begin{proof}
Recall the functions
$a_\varepsilon(r)=C_{0,\varepsilon}-C_{1,\varepsilon}r^2$ and
$b_\varepsilon(r)=r^2$ from Section~\ref{subsub: extension to
binding}.  We can slightly modify $a_\varepsilon(r)$ near $r=0$ by
replacing $C_{1,\varepsilon}\rightarrow C_{1,\varepsilon}+\tau$,
where $\tau$ is a small number so that the ratio
$b_\varepsilon'(r):{a_\varepsilon'(r)\over 2\pi}$ becomes
irrational.  Hence (2) and (4) are satisfied.  Also, (1) and (3) are
satisfied for orbits near $\gamma_0$.  Finally, the procedure in
Lemma~\ref{lemma: genericity} can be applied to $M-N(\gamma_0)$ to
yield (1) and (3).
\end{proof}

The following lemma is used in the proof of Corollary~\ref{cor:
infinitely many simple orbits}.

\begin{lemma} \label{lemma: perturb finite number of orbits}
Let $\alpha$ be a contact $1$-form on a closed $3$-manifold $M$. If
$\alpha$ is degenerate and has a finite number of simple orbits,
then for any $N\gg 0$ there exists a smooth function $g_N$ which is
$C^\infty$-close to $1$ so that all the periodic orbits of
$R_{g_N\alpha}$ of action $\leq N$ are nondegenerate and lie in
small neighborhoods of the periodic orbits of $R_\alpha$ (and hence
are freely homotopic to multiple covers of the periodic orbits of
$R_\alpha$).
\end{lemma}

\begin{proof}
Let $U_1,\dots,U_k$ be the neighborhoods of the simple orbits
$S^1\times \{(0,0)\}$ and $V_1,\dots,V_k$ be the security
neighborhoods as in Lemma~\ref{lemma: genericity}, taken so they are
sufficiently small and mutually disjoint. As before, on
$Z=M-\cup_{i=1}^k U_i$, any sufficiently small perturbation will not
create any new orbits of action $\leq N$. The perturbations inside
the $V_i$ will make the Reeb vector field nondegenerate.
\end{proof}

\section{Holomorphic disks} \label{disks}

Let $h$ be a diffeomorphism which is freely homotopic to a
pseudo-Anosov representative $\psi$ and let $S$ be a page of the
open book decomposition. For each boundary component $(\bdry S)_i$
of $\bdry S$, there is an associated fractional Dehn twist
coefficient $c_i={k_i\over n_i}$, where $n_i$ is the number of
prongs.  Let $\alpha_{\varepsilon,\varepsilon'}$ be the contact
$1$-forms and $R_{\varepsilon,\varepsilon'}$ be the corresponding
Reeb vector fields constructed in Section~\ref{section:
construction}. Recall that the direction of
$R_{\varepsilon,\varepsilon'}$ does not depend on $\varepsilon$,
provided $\varepsilon$ is small enough that the contact condition is
satisfied. In what follows we assume that
$\alpha_{\varepsilon,\varepsilon'}$ is nondegenerate, by applying
the $C^\infty$-small perturbation given in Lemma~\ref{lemma:
perturbation}. Let $(\R\times M,
d(e^t\alpha_{\varepsilon,\varepsilon'}))$ be the symplectization of
$(M,\xi_{\varepsilon,\varepsilon'}=\ker\alpha_{\varepsilon,\varepsilon'})$
and $J_{\varepsilon,\varepsilon'}$ be an almost complex structure
which is adapted to the symplectization. We have the following
theorem:

\begin{theorem} \label{theorem: nodisks}
Suppose the fractional Dehn twist coefficient $c_i\geq
\frac{2}{n_i}$ for each boundary component $(\bdry S)_i$.  Given
$N\gg 0$, for sufficiently small $\varepsilon, \varepsilon'>0$, no
closed orbit $\gamma$ of $R_{\varepsilon,\varepsilon'}$ with action
$\int_\gamma \alpha_{\varepsilon,\varepsilon'}\leq N$ is the
positive asymptotic limit of a finite energy plane $\tilde{u}$ with
respect to $J_{\varepsilon,\varepsilon'}$.
\end{theorem}

We will usually say that $\gamma$ which is the positive asymptotic
limit of a finite energy plane $\tilde{u}$ {\em bounds} $\tilde{u}$.
Theorem~\ref{theorem: nodisks} implies that:

\begin{cor}\label{cor: cylindrical}
Suppose $c_i\geq {2\over n_i}$ for each boundary component $(\bdry
S)_i$.  Given $N\gg 0$, for sufficiently small
$\varepsilon,\varepsilon'>0$, the cylindrical contact homology group
$HC_{\leq N}(M,\alpha_{\varepsilon,\varepsilon'})$ is well-defined.
\end{cor}

\begin{proof}[Outline of proof of Theorem~\ref{theorem: nodisks}]
Without loss of generality assume that $\bdry S$ is connected. Fix
$N\gg 0$. By Proposition~\ref{prop: genuine trajectory}, for
sufficiently small $\varepsilon, \varepsilon'>0$, any genuine
trajectory $Q$ of $R_{\varepsilon,\varepsilon'}$ which intersects a
page at most $N$ times satisfies $\Phi(Q)=0$. Also, for
$\varepsilon$ small, the number of intersections of a closed orbit
$\gamma$ with a page of the open book is approximately the same as
the action $A_{\alpha_{\varepsilon,\varepsilon'}}(\gamma)$, provided
$\gamma$ lies in $\Sigma(S,\psi')$. Fix sufficiently small constants
$\varepsilon,\varepsilon'>0$ so that the above hold, and write
$\alpha=\alpha_{\varepsilon,\varepsilon'}$,
$R=R_{\varepsilon,\varepsilon'}$,
$\xi=\xi_{\varepsilon,\varepsilon'}$, and
$J=J_{\varepsilon,\varepsilon'}$. We will prove that no closed orbit
of $R$ in $\Sigma(S,\psi')$ which intersects a page at most $N$
times bounds a finite energy plane, and that no closed orbit in
$M-\Sigma(S,\psi')$ bounds a finite energy plane.

\begin{figure}[ht]
\begin{overpic}[height=3in]{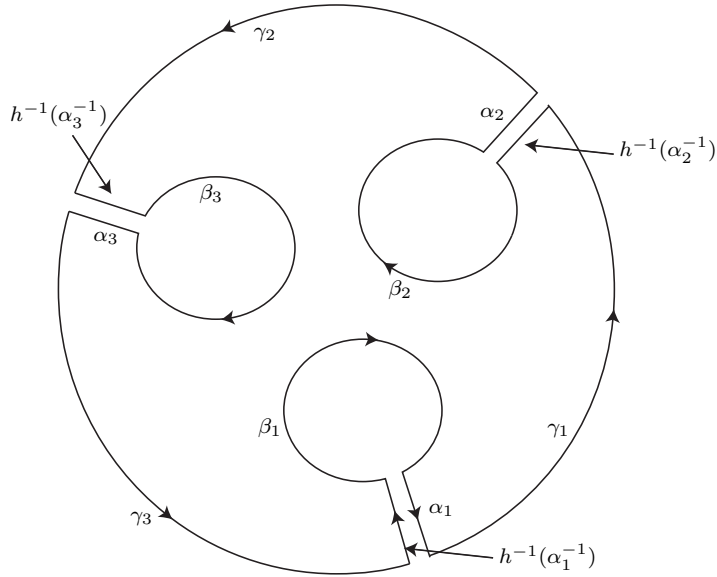}
\put(-8,79.5){\tiny{$h^{-1}(\alpha_3^{-1})$}}
\put(13,9.5){\tiny{$\gamma_3$}} \put(6.5,58.5){\tiny{$\alpha_3$}}
\put(25,65.8){\tiny{$\beta_3$}} \put(34.5,94){\tiny{$\gamma_2$}}
\put(35.2,25){\tiny{$\beta_1$}} \put(65,11){\tiny{$\alpha_1$}}
\put(58,49){\tiny{$\beta_2$}} \put(74,80.1){\tiny{$\alpha_2$}}
\put(77.3,2.25){\tiny{$h^{-1}(\alpha_1^{-1})$}}
\put(85.7,25){\tiny{$\gamma_1$}}
\put(98.3,73.25){\tiny{$h^{-1}(\alpha_2^{-1})$}}
\end{overpic}
\caption{The disk $\mathcal{D}$. The labelings indicate the image of
the given boundary arc under the map $\overline{u}$.} \label{disk}
\end{figure}

We argue by contradiction. Suppose there exists a nondegenerate
orbit $\gamma$ of $R$ which bounds a finite energy plane $\tilde
u=(a,u):\C\rightarrow \R\times M$. Assume in addition that $\gamma$
is not a cover of the binding $\gamma_0$, oriented as the boundary
of a page. By construction, $\gamma$ is transverse to the pages of
the open book. After perturbing $u$ if necessary, we may take $u$ to
be positively transverse to $\gamma_0$. The holomorphicity of
$\tilde u$ ensures that there are no negative intersections. Now let
$N(\gamma_0)$ be a sufficiently small tubular neighborhood of
$\gamma_0$, one which depends on $\gamma$. By restricting to
$M-N(\gamma_0)$ and reparametrizing, we view $u$ as a map
$\overline{u}$ from a planar surface $P$ to $M-N(\gamma_0)$. Here
$P$ is obtained from a unit disk $D$ by excising small disks which
consist of points whose images lie in $N(\gamma_0)$. Next identify
$M-N(\gamma_0)\simeq (S_0'\times[0,1])/(x,1)\sim (h(x),0)$, where
$S_0'$ is a small retraction of the page $S_0$. Cut $M-N(\gamma_0)$
open along $S_0'\times\{0\}$ and project to $S_0'$ via the
projection $\pi: S_0'\times[0,1]\rightarrow S_0'$ onto the first
factor. Then we obtain the map $\pi\circ
\overline{u}:\mathcal{D}\rightarrow S_0'$, where $\mathcal{D}$ is a
disk obtained from $P$ by making cuts along arcs as given in
Figure~\ref{disk}. The cutting-up/normalization process will be done
in detail in Section~\ref{subsection: cutting up}.  In
Section~\ref{subsection: noncontractible} we prove
Proposition~\ref{prop: not contractible}, which states that
$\pi(\overline u(\bdry\mathcal{D}))$ cannot be contractible if
$c\geq{2\over n}$. This is proved using the Rademacher function
$\Phi$ which is adapted to $\mathcal{F}$, and relies on the fact
from Proposition~\ref{prop: genuine trajectory} that $\Phi$ of a
genuine trajectory $Q$ is zero. This gives us the desired
contradiction. The case when $\gamma$ is a multiple cover of
$\gamma_0$ is similar.
\end{proof}

\subsection{Cutting up the finite energy plane} \label{subsection:
cutting up}

The type of argument we are using first appeared in
\cite[Annexe]{CH1}. The discussion in this subsection does not
depend the specific diffeomorphism type of $h$.

Suppose there exists a nonconstant finite energy plane $\tilde
u=(a,u): \C\rightarrow \R\times M$ which is bounded by $\gamma$. We
will slightly modify $u$ to get a map $\overline u: D^2=\{|z|\leq
1\}\rightarrow M$. Let $w$ be the coordinate for $\C$ and $z$ be the
coordinate for $D^2$. Also let $\Pi:TM=\R R\oplus \xi\rightarrow
\xi$ be the projection to $\xi$.

The following summarizes the results of
Hofer-Wysocki-Zehnder~\cite{HWZ1}, tailored to our needs:

\begin{prop}
There exists a smooth map $\overline u:D^2\rightarrow M$ which is
bounded by $\gamma$ and satisfies the following:
\begin{itemize}
\item $\overline u|_{int(D^2)}$ is immersed away from a finite number of points in $int(D^2)$.
\item $\overline u(z)=u(Rz)$ for large $R>0$ and $|z|\leq r_0<1$.
Hence $(a(Rz),\overline u(z))$ is holomorphic on the subdisk
$\{|z|\leq r_0\}$. Moreover, $\{|w|\leq Rr_0\}$ contains the set of
nonimmersed points of $u$.
\item At points where $\overline u$ is immersed, $\overline u$ is positively transverse to
$R$.
\item $\overline u(z)\not\in \operatorname{Im}(\gamma)$ for $r_0<|z|<1$.
\end{itemize}
\end{prop}

\begin{proof}
Let $\R/\Z\times D^2_\delta$ be a small neighborhood of
$\operatorname{Im}(\gamma)$, where $D^2_\delta$ is a disk of radius
$\delta>0$ and $\gamma$ maps to $\R/\Z\times\{0\}$. Let $k_0$ be the
multiplicity of $\gamma$, i.e., the number of times $\gamma$ covers
a simple orbit. When restricted to $|w|\gg 0$, $u(w)$ maps to
$\R/\Z\times D^2_\delta$ and has components $(u_0(w),u_1(w))$. Then,
according to \cite[Theorem~1.4]{HWZ1},
\begin{itemize}
\item $u_0(re^{2\pi it})$ asymptotically approaches $k_0t$, with error
term $O(e^{-Cr})$.  The same holds for all the derivatives of $u_0$.
\item $u_1(re^{2\pi it})=e^{\int_{r_1}^r \mu(\tau)d\tau}[E(e^{2\pi it})
+F(re^{2\pi it})]$, where $\mu:[r_1,\infty)\rightarrow \R$ is a
smooth function which limits to $\lambda<0$, $E(e^{2\pi i t})$ is a
nowhere vanishing function with values in $\R^2$, and $F(re^{2\pi
it})$ is the remainder term which approaches $0$ uniformly in $t$
for all derivatives, as $r\rightarrow \infty$.  (The function
$E(e^{2\pi it})$ is an eigenfunction of a suitable operator with
eigenvalue $\lambda$.)
\end{itemize}
We note that some care is required in choosing the coordinates on
$\R/\Z\times D^2_\delta$.

We now reparametrize $u:\C\rightarrow M$.  Consider the map
$\phi:int(D^2)\rightarrow \C$, $(r,\theta)\mapsto (f(r),\theta)$,
where $f(r)=Rr$ for $r\leq r_0$, $f'(r)>0$, and $f(r)={1\over 1-r}$
near $r=1$.  Then let $\overline u=u\circ \phi$ on $int(D^2)$ and
$\overline u(e^{2\pi it})=(k_0t,0)$.  The above asymptotics
guarantee the smoothness of $\overline u: D^2\rightarrow M$.

The first and last statements follow from \cite[Theorem~1.5]{HWZ1},
which states that (i) $\Pi\circ u_*$ is nonzero (and hence $u$ is an
immersion) away from a finite number of points and that (ii) $u$
intersects $\gamma$ at finitely many points.
\end{proof}

The map $\overline u$, restricted to $int(D^2)$, either intersects
$\gamma_0$ transversely and positively or intersects $\gamma_0$ at a
point where $\Pi\circ \overline u_*=0$. The following lemma allows
us to restrict to the former situation.

\begin{lemma} \label{lemma: perturbation overline u}
There exists a perturbation $\overline v$ of $\overline u$ with
small support inside $int(D^2)$ so that $\overline v$ is positively
transverse to $R$ away from isolated complex branch points and is
positively transverse to $\gamma_0$.
\end{lemma}

We emphasize that $\overline u$ and $\overline v$ are no longer
holomorphic everywhere.

\begin{proof}
Suppose without loss of generality that $\Pi\circ\overline u_*(0)=0$
and $\overline u(0)\in \gamma_0$. Let $[-1,1]\times D^2\subset M$ be
a small neighborhood of $\overline u(0)=(0,0)$, where the Reeb
orbits are $[-1,1]\times \{pt\}$ and $[-1,1]\times\{0\}$ is a subarc
of $\gamma_0$.  Now restrict the domain of $\overline u$ to a small
neighborhood $D^2_\varepsilon=\{|z|\leq \varepsilon\}$ of $0$ and
write $\overline u=(\overline u_0,\overline u_1)$, where $\overline
u_0$ is the component that maps to $[-1,1]$ and $\overline u_1$ is
the component that maps to $D^2$. Define a smooth function $f:
D^2_\varepsilon\rightarrow \R$ which satisfies the following:
\begin{itemize}
\item $f(z)=\delta$ on $|z|\leq \varepsilon''$.
\item $f(z)=0$ on $|z|\geq \varepsilon'$.
\item $|f'|$ is small on $D^2_\varepsilon$.  (This means that
$\delta$ must be a very small positive number.)
\end{itemize}
Here $0<\varepsilon''<\varepsilon'<\varepsilon$.  Then define
$\overline v(z)=(\overline u_0(z), \overline u_1(z)+f(z))$.  On
$|z|\leq \varepsilon''$, we are simply translating the holomorphic
disk; this does not affect the transversality with $R$.  Now,
provided $|f'|$ is sufficiently small, the transversality on
$\varepsilon''\leq |z|\leq \varepsilon'$ is unaffected. The point
near $z=0$ which intersects $\gamma_0$ is distinct from the point
$z=0$ at which $\Pi\circ \overline v_*=0$.
\end{proof}

The map $\overline v$ from Lemma~\ref{lemma: perturbation overline
u} will be renamed as $\overline u$.

\s Suppose that $\gamma$ does not cover $\gamma_0$. By this we mean
$\gamma\not=\gamma_0$ and $\gamma$ is not a multiple cover of
$\gamma_0$. Since $\gamma_0$ intersects $\overline u$ transversely,
there is a small neighborhood $N(\gamma_0)$ of $\gamma_0$ so that
$\overline u(D^2)\cap \bdry (M-N(\gamma_0))$ is a union of circles,
each of which intersects $\bdry (S_0'\times\{pt\})$ exactly once.
Let $P$ be the planar subsurface of $D^2$ obtained by excising all
$z\in D^2$ such that $\overline u(z) \in int (N(\gamma_0))$. We
write $\bdry P= \bdry_0 P + \bdry_1 P$ where $\bdry_0 P$ maps to
$\gamma$ and the components of $\bdry_1 P$ map to $\bdry
(M-N(\gamma_0))$.

Now take $S_0'= S_0'\times\{0\}$ and consider the intersection of
$S_0'$ and $\overline u(P)$. Observe that $\overline u|_P$ is
already transverse to $S_0'\times\{t\}$ in a neighborhood of $\bdry
P$, for all $t$. Next, by Sard's theorem, there exists
$S_0'\times\{\varepsilon\}$ which is transverse to $\overline u|_P$
with $\varepsilon$ arbitrarily small. By renaming the $t$-variable
(i.e., translating $t\mapsto t-\varepsilon$), we may assume that
$S_0'=S_0'\times\{0\}$ and $\overline u|_P$ intersect transversely.
We now have the following:

\begin{lemma}
The intersection $P\cap \overline u^{-1}(S_0')$ is a union of
properly embedded arcs and embedded closed curves in $P$ which
satisfy the following: \be
\item The embedded closed curves bound disks in $P$.
\item There is an arc $a_i$ which is the unique arc to connect the $i$th component
of $\bdry_1 P$ to $\bdry_0 P$.
\ee
\end{lemma}

Order the $a_i$ so that their endpoints in $\bdry_0 P$ are in
counterclockwise order, and order the components of $\bdry_1 P$
using the $a_i$. Also, $a_i$ is oriented {\em from} $\bdry_1 P$ {\em
to} $\bdry_0 P$.

\begin{proof}
Let $\delta$ be a closed curve of $P\cap \overline u^{-1}(S_0')$.
Then $\delta$ cuts off a planar subsurface $P_0$ whose boundary
consists of $\delta$, together with components of $\bdry_1 P$.  Now
consider the intersection pairing with $S_0'$.  Since $\langle
\overline u(\delta), S_0'\rangle=0$ but each component of $\overline
u(\bdry_1 P)$ intersects $S_0'$ negatively, it follows that $\bdry
P_0=\delta$.

Now if $\langle \gamma, S_0'\rangle=m>0$, then there must be $m$
endpoints of $P\cap \overline u^{-1}(S_0')$ on $\bdry_0 P$ and $1$
endpoint each on the $m$ components of $\bdry_1 P$.  If the arc
$a_i$ which begins at the $i$th component of $\bdry_1P$ ends on
another component of $\bdry_1P$, then there must be an arc from
$\bdry_0P$ to itself.  This would contradict $\langle \gamma,
S_0'\rangle=m$.  The lemma follows.
\end{proof}

Take an embedded closed curve of $P\cap \overline u^{-1}(S_0')$ in
$P$ which bounds an innermost disk $D_0$. Then $\overline u(D_0)$
can be pushed across $S_0'$ by either flowing forwards or backwards
along $R$ (depending on the situation). In this way we can eliminate
all the embedded closed curves of $P\cap \overline u^{-1}(S_0')$ in
$P$.

Now cut $P$ along the union of the arcs in $P\cap \overline
u^{-1}(S_0')$ to obtain a disk $\mathcal{D}$. We now have a map
$\overline{u}:\mathcal{D}\rightarrow S_0'\times [0,1]$. After
cutting open at one point, $\overline{u}(\bdry \mathcal{D})$ is
given by:
\begin{equation}
\overline{u}(\bdry
\mathcal{D})=h^{-1}(\alpha_1^{-1})\beta_1\alpha_1\gamma_1\cdots
h^{-1}(\alpha_m^{-1})\beta_m\alpha_m\gamma_m,
\end{equation}
where $\alpha_i$ refers to $\alpha_i\times\{0\}$, $\alpha_i^{-1}$ is
$\alpha_i$ with the opposite orientation, $h^{-1}(\alpha_i^{-1})$
refers to $h^{-1}(\alpha_i^{-1})\times\{1\}$, $\gamma_i$ are
components of the Reeb orbit $\gamma$ cut along $S_0'\times\{0\}$,
and $\beta_i$ are arcs of the type $\{pt\}\times [0,1]$ where $pt\in
\bdry S_0'$. See Figure~\ref{disk}.

Next we compose $\overline{u}$ with the projection
$\pi:S_0'\times[0,1]\rightarrow S_0'$ onto the first factor. Then
the curve $\pi(\overline{u}(\partial \mathcal{D}))\subset S_0'$ is
decomposed into consecutive arcs:
\begin{equation}\label{equation: original}
\pi(\overline{u}(\partial \mathcal{D}))= h^{-1} (\alpha_1^{-1} )
\alpha_1 \gamma_1\cdots  h^{-1} (\alpha_m^{-1} )\alpha_m \gamma_m ,
\end{equation}
where the $\gamma_i$ are actually $\pi(\gamma_i)$. Also note that
the $\beta_i$ project to points.

Rewrite $\pi(\overline{u}(\partial \mathcal{D}))$ as:
\begin{equation} \label{eqn: 813}
\pi(\overline{u}(\partial \mathcal{D}))= h^{-1} (\xi_1^{-1})\xi_1
h^{-1} (\xi_2^{-1}) \xi_2\cdots h^{-1} (\xi_m^{-1})\xi_m Q',
\end{equation}
where $Q'=h^{m-1}(\gamma_1) h^{m-2} (\gamma_2 )\cdots\gamma_m$ is
the projection $S_0'\times[0,m]\rightarrow S_0'$ onto the first
factor of a lift of $\gamma$ to $S_0'\times [0,m]$ and
\begin{eqnarray*}
\xi_1&=&\alpha_1, \\
\xi_2&=&\alpha_2h(\gamma_1^{-1})\\
\xi_3&=&\alpha_3 h(\gamma_2^{-1}) h^2(\gamma_1^{-1})\\
&\vdots &
\end{eqnarray*}

\subsection{Noncontractibility} \label{subsection: noncontractible}
Let $h:S_0\stackrel\sim\rightarrow S_0$ be a diffeomorphism with
$h|_{\bdry S_0}=id$, fractional Dehn twist coefficient $c={k\over
n}$, and pseudo-Anosov representative $\psi$. Writing $S_0=A_0\cup
S$, we may assume that $h|_{S}=\psi'$ and $h|_{A_0}$ is a
rotation/fractional Dehn twist by $c$. Also let
$h_0:S_0\stackrel\sim\rightarrow S_0$ be a homeomorphism which is
isotopic to $h$ relative to $\bdry S_0$, so that $h_0|_{S}=\psi$.

In this subsection we prove the following proposition:

\begin{prop}\label{prop: not contractible}
Suppose $\gamma$ does not cover $\gamma_0$. If $k\geq 2$, then
$\gamma$ does not bound a finite energy plane.
\end{prop}

Suppose $\gamma\subset \Sigma(S,\psi')$.  If $\gamma$ bounds a
finite energy plane $\tilde u$, then we can cut up the finite energy
plane to obtain $\mathcal{D}$ which satisfies Equation~\ref{eqn:
813}. If we apply $h^{-m+1}$ to Equation~\ref{eqn: 813}, then we
obtain:
\begin{equation}
\label{equation: concat with Q}
\Gamma:=h^{-m+1}(\pi(\overline{u}(\partial \mathcal{D})))= h^{-1}
(\zeta_1^{-1})\zeta_1 h^{-1} (\zeta_2^{-1}) \zeta_2\cdots h^{-1}
(\zeta_m^{-1})\zeta_m Q.
\end{equation}
Here $\zeta_i= h^{-m+1}(\xi_i)$ and
\begin{eqnarray*}
Q&=&\gamma_1 h^{-1} (\gamma_2 )\cdots h^{-m+1}(\gamma_m)\\
&=&\gamma_1 (\psi')^{-1}(\gamma_2)\dots (\psi')^{-m+1}(\gamma_m)
\end{eqnarray*}
is a concatenation of the type appearing in Equation~\ref{eqn: Q}.
The goal is to prove that $\Gamma$ is not contractible.

The key ingredient to proving the non-contractibility of $\Gamma$ is
is the Rademacher function $\Phi$ with respect to the stable
foliation $\F$. Let $(\theta,y)$ be coordinates on
$A_0=S^1\times[-1,0]$ so that $S^1\times\{0\}$ is identified with
$\bdry S$. Pick a retraction $\rho: S_0\rightarrow S$ which sends
$(\theta,y)\mapsto (\theta,0)$. If $\tau$ is an arc in $S_0$, then
we define $\Phi(\tau)=\Phi(\rho(\tau))$.

\begin{lemma}\label{lemma: estimate2}
Let $\eta$ be an arc on $S_0$. Then $\Phi (h_0^{-1} (\eta^{-1} )
\eta )=k-1$ or $k$.
\end{lemma}

When we compute $\Phi$ values, we often suppress $\rho$ and write
$\tau$ to mean $\rho(\tau)$.

\begin{proof}
First observe that the arc $h_0^{-1} (\eta^{-1} )\eta$ is isotopic,
relative to its endpoints, to the concatenation $\psi^{-1}
(\eta^{-1}) \delta \eta$, where $\delta$ is a subarc of $\partial
S$.

Next lift $\psi^{-1} (\eta^{-1}) \delta \eta$ to the universal cover
$\wt S$. We place a tilde to indicate a lift. Let $\wt\psi$ be an
appropriate lift of $\psi$ so that $\wt\psi^{-1} (\wt{\eta^{-1}})
\wt\delta \wt\eta$ is the chosen lift of $\psi^{-1} (\eta^{-1})
\delta \eta$. Let $d$ be the component of $\bdry\wt S$ which
contains $\wt\delta$. Recall that $\wt{\S}_d$ is the union of prongs
$\wt P_i$ that emanate from $d$. We orient each component $\wt P_i$
of $\wt{\S}_d$ so that it points into $\wt{S}$.

If necessary, we perturb the initial point of $\wt\eta$ (and hence
the terminal point of $\wt\psi^{-1}(\wt{\eta^{-1}})$) so that the
endpoints of $\wt\delta$ do not lie on $\wt{\S}_d$. In that case,
$\Phi (\delta)=k-1$, since $\Phi$ of an oriented arc on $d$ is the
oriented intersection number with $\wt{\S}_d$ minus $1$ when the
contribution is positive, and plus $1$ when it is negative. If the
terminal point of $\wt\eta$ lies on $\wt{\S}_d$, then the whole of
$\wt\psi^{-1} (\wt{\eta^{-1}}) \wt\delta \wt\eta$ can be isotoped
onto $d$ via an isotopy which constrains the endpoints to lie on
$\wt{\S_d}$.  In this case, $\Phi(h_0^{-1}(\eta^{-1})\eta)=k$.
Assume otherwise. Then we can isotop $\wt\eta$ while fixing one
endpoint and constraining the other to lie on $d$, so that $\wt\eta$
becomes disjoint from $\wt{\S}_d$. We may also assume that $\wt\eta$
is a quasi-transversal arc. By the $\wt\psi$-invariance of
$\wt{\S}_d$, it follows that $\wt\psi^{-1}(\wt{\eta^{-1}})$ is also
disjoint from $\wt{\S}_d$. Hence the contribution of $d$ towards
$\Phi(\psi^{-1} (\eta^{-1} )\delta \eta)$ is $k-1$.  Moreover, there
is no concatenation error if we use the $\wt\eta$ as normalized
above. Since $\Phi(\psi^{-1}(\eta^{-1}))=-\Phi(\eta)$ by
Proposition~\ref{prop: rademacher}, it follows that $\Phi (h_0^{-1}
(\eta^{-1}) \eta)=k-1.$
\end{proof}

Let $x$ be the initial point of $Q$. Then Equation~\ref{equation:
concat with Q} can be written as:
\begin{equation} \label{equation: concat with Q final version}
\Gamma = R' h_0^{-1} (\eta_1^{-1})\eta_1 h_0^{-1} (\eta_2^{-1})
\eta_2\dots h_0^{-1} (\eta_m^{-1})\eta_m Q,
\end{equation}
where $R'$ is the path $g_1 h_0^{-1}(g_2)\cdots h_0^{-m+1}(g_m)$
which joins $h^{-m}(x)$ to $h_0^{-m}(x)$, and
\begin{eqnarray*}
\eta_m&=&\zeta_m,\\
\eta_{m-1}&=&\zeta_{m-1}g_m,\\
\eta_{m-2}&=&\zeta_{m-2}g_{m-1}h_0^{-1}(g_m)\\
&\vdots &
\end{eqnarray*}
Here $g_i$ is a path from $h^{-1}(\zeta_i^{-1})(x)$ to
$h_0^{-1}(\zeta_i^{-1})(x)$ so that $h^{-1}(\zeta_i^{-1})=g_i
h_0^{-1}(\zeta_i^{-1})$.

In what follows, we pass to the universal cover $\wt{S_0}$ of $S_0$.
Choose a lift of $\pi(\overline{u}(\partial \mathcal{D}))$.  Let
$\wt{h_0^{-m}}$ be the lift of $h_0^{-m}$ which sends the terminal
point of the lift of $h_0^{-1} (\eta_1^{-1} )\eta_1 h_0^{-1}
(\eta_2^{-1}) \eta_2\dots h_0^{-1} (\eta_m^{-1} )\eta_m$ to its
initial point.  We may decompose it as
$$\wt{h_0^{-m}}=\wt{h_{0,m}^{-1}} \circ \wt{h_{0,m-1}^{-1}}
\circ \dots \circ \wt{h_{0,1}^{-1}},$$ where the $\wt{h_{0,i}^{-1}}$
is the lift of $h_0^{-1}$ which sends the terminal point of
$\wt\eta_i$ to the terminal point of $\wt \eta_{i-1}$. Also let
$\wt{h^{-m}} =\wt{h_m^{-1}} \circ \wt{h_{m-1}^{-1}} \circ \dots
\circ \wt{h_1^{-1}}$ be the lift of $h^{-m}$ which coincides with
$\wt{h_0^{-m}}$ near $\partial \wt{S_0}$, and let $\wt{(\psi')^{-m}}
=\wt{(\psi')_m^{-1}} \circ \dots \circ \wt{(\psi'_1)^{-1}}$ and
$\wt{\psi^{-m}} =\wt{\psi_m^{-1}} \circ \dots \circ
\wt{\psi_1^{-1}}$ be the restrictions of $\wt{h^{-m}}$ and
$\wt{h_0^{-m}}$ to $\wt{S}$. Now the arc $\wt R'$ is an arc which
joins $\wt{h^{-m}} (\wt\eta_m (1))$ to $\wt{h_0^{-m}} (\wt\eta_m (1)
)$ in $\wt{S_0}$.

Assume that $\psi(x_i)=x_i$, again for ease of indexing.

\begin{lemma}\label{lemma: left}
Suppose $\wt{\psi^{-m}}$ and $\wt{(\psi')^{-m}}$ are isotopic lifts
of $\psi^{-m}$ and $(\psi')^{-m}$, with $m\geq 1$. If $p\in \wt{S}$
and $\wt{\psi^{-m}} (p)$ is to the left of a lift $\wt{P_i}$ of
$P_i$, then $\wt{(\psi')^{-m}} (p)$ is strictly to the left of the
lift of $(\psi')^{-m} (W_{i,R})$ which starts near $\wt{P_i}$ on the
same component of $\bdry \widetilde S$.
\end{lemma}

The same holds if we replace all occurrences of ``left'' by
``right''.

\begin{proof}
If $\wt{\psi^{-m}} (p)$ is to the left of a lift $\wt{P_i}$ of
$P_i$, then $p$ is to the left of $\wt{\psi^m} (\wt{P_i})$. Since
$\wt{\psi^m} (\wt{P_i} )$ is a prong, $p$ is strictly to the left of
the lift $\wt{W}_{i,R}$ of $W_{i,R}$ which starts near it (as
$\wt{\psi^m} (\wt{P_i} ) \leq \wt{W}_{i,R}$). Applying
$\wt{(\psi')^{-m}}$ to $p$ and $\wt{W}_{i,R}$, the lemma follows.
\end{proof}

\begin{cor} \label{corollary: phi of g is zero}
Suppose $\wt{\psi^{-m}}$ and $\wt{(\psi')^{-m}}$ are isotopic lifts
of $\psi^{-m}$ and $(\psi')^{-m}$, with $m\geq 1$. For any $p\in
\wt{S}$, the path $g_{-m,p}$ from $\wt{\psi^{-m}} (p)$ to
$\wt{(\psi')^{-m}} (p)$ satisfies $\Phi(g_{-m,p})=0$.
\end{cor}

\begin{proof}
Suppose $\Phi(g_p)\not=0$. Then $g_{-m,p}$ needs to cross two
consecutive prongs $\wt{P}_i$ and $\wt{P}_{i+1}$ which emanate from
the same component of $\bdry \wt S$.  Suppose, without loss of
generality, that $\wt{\psi^{-m}} (p)$ is to the left of $\wt{P}_i$
and $\wt{(\psi')^{-m}} (p)$ is to the right of $\wt{P}_{i+1}$.  This
contradicts Lemma~\ref{lemma: left}, since $\wt{(\psi')^{-m}} (\wt
W_{i,R})\leq \wt{P}_{i+1}$.
\end{proof}

We now prove Proposition~\ref{prop: not contractible}.

\begin{proof}[Proof of Proposition~\ref{prop: not contractible}]
Suppose that $\gamma\subset\Sigma (S,\psi' )$.  By Lemma~\ref{lemma:
estimate2}, $\Phi (h_0^{-1}(\eta_j^{-1})\eta_j)\geq k-1$ for all
$j$. By (3) of Proposition~\ref{prop: rademacher}, we deduce that
$$\Phi (h_0^{-1} (\eta_1^{-1} )\eta_1 h_0^{-1} (\eta_2^{-1})\eta_2
\dots h_0^{-1} (\eta_m^{-1} )\eta_m )\geq m(k-1)-(m-1).$$ Since
$k\geq 2$, the right-hand side of the inequality is $\geq 1$. Hence
we know that there exist consecutive lifts $\wt P_i$ and $\wt
P_{i+1}$ starting on the same component $d$ of $\bdry \wt S$, so
that the initial point and the terminal point of a lift of the arc
$h_0^{-1} (\eta_1^{-1} )\eta_1 \dots h_0^{-1} (\eta_m^{-1}) \eta_m$
are respectively strictly to the left of $\wt P_i$ and strictly to
the right of $\wt P_{i+1}$.

As we saw in the proof of Proposition~\ref{prop: genuine
trajectory}, the endpoint of $\wt{Q}$ ($=$ the endpoint of the lift
$\widetilde\Gamma$ of $\Gamma$) is then strictly to the right of
$\wt{(\psi')^{-m+1}}(\wt W_{i+1,L})$, which starts on $d$ between
$\wt P_i$ and $\wt P_{i+1}$, provided $0<m\leq N$.  Now, by
Lemma~\ref{lemma: left}, the initial point of $\wt R'$ ($=$ the
initial point of $\widetilde\Gamma$) is strictly to the left of
$\wt{(\psi')^{-m}} (\wt W_{i,R} )$, which starts on $d$ between $\wt
P_i$ and $\wt P_{i+1}$. Since $\wt{(\psi')^{-m}} (\wt W_{i,R} ) \leq
\wt{(\psi')^{-m+1}} (\wt W_{i+1,L} )$, it follows that $\Gamma$ is
not contractible, which is a contradiction.

Next suppose that $\gamma$ lies in $M-\Sigma(S,\psi')$. In this
case, we retract $\Gamma$ using $\rho:S_0\rightarrow S$; this time
the endpoints of $\eta_i$ are moved to $\bdry S$.  The rest of the
argument is the same. This concludes the proof of
Proposition~\ref{prop: not contractible}.
\end{proof}

\s\n {\bf Case when $\gamma$ covers $\gamma_0$.} Finally consider
the case when $\gamma$ covers $\gamma_0$.  Let $N(\gamma_0)$ be a
small tubular neighborhood of $\gamma_0$ so that $(\Pi\circ
\overline u_*)(q) \not = 0$ for all $q$ with $\overline u(q)\in
N(\gamma_0)$. Also, by Lemma~\ref{lemma: perturbation}, we may take
$\bdry N(\gamma_0)$ to be foliated by Reeb orbits of irrational
slope $c$, where $\bdry N(\gamma_0)$ is identified with $\R^2/\Z^2$
so that the meridian has slope $0$ and $\bdry S_0'$ has slope
$\infty$. Consider $M-N(\gamma_0)=S_0'\times[0,1]/\sim$ as before.
As $\overline{u}|_{int(D^2)}$ is transverse to the Reeb vector field
$R$ away from finitely many branch points, it follows that the
component $\delta$ of $\overline{u}(D^2)\cap N(\gamma_0)$, parallel
to and oriented in the same direction as $\gamma$ in
$\overline{u}(D^2)$, would have slope ${1\over m_0}$ which satisfies
$m_0< {1\over c}$. On the other hand, all the other components of
$\overline{u}(D^2)\cap \bdry N(\gamma_0)$
--- those that bound
meridian disks in $N(\gamma_0)$ and are oriented as
$\bdry(\overline{u}(D^2)\cap (M-N(\gamma_0)))$ --- intersect $S_0'$
negatively. Since the oriented intersection number of
$\bdry(\overline{u}(D^2)\cap (M-N(\gamma_0)))$ with $S_0'$ is zero,
it follows that $m_0\geq 0$. Now, $m_0=0$ is impossible, since
$\delta$ could then be homotoped to $\partial S_0$ and
$\overline{u}(D^2)$ to a disk in $S_0$.  If $m_0>0$, then we can
apply the analysis of this section with a slightly smaller disk
whose boundary maps to $\delta$.  This time, $\Phi(Q)$ will
contribute positively, making $\Phi(\pi(\overline{u}(\partial
\mathcal{D}))$ more positive. This concludes the proof of
Theorem~\ref{theorem: nodisks}.

\section{Holomorphic cylinders} \label{cylinders}
In this section we give restrictions on holomorphic cylinders
between closed Reeb orbits.  We say that there is a holomorphic
cylinder {\em from $\gamma$ to $\gamma'$} if there is a holomorphic
cylinder in the symplectization $\R\times M$ whose asymptotic limit
at the positive end is $\gamma$ and whose asymptotic limit at the
negative end is $\gamma'$.


Let $\mathcal{P}_{\varepsilon,\varepsilon'}$ be the set of good
orbits of $R_{\varepsilon,\varepsilon'}$. A periodic orbit $\gamma$
which is an $m_0$-fold cover of the binding $\gamma_0$ will be
written as $m_0\gamma_0$. Let $\mathcal{P}^{>
0}_{\varepsilon,\varepsilon'}$ be the set of good orbits which are
not $m_0\gamma_0$ for any $m_0$. In other words, they nontrivially
intersect the pages of the open book.  We now define the {\em open
book filtration} $F:\mathcal{P}^{>
0}_{\varepsilon,\varepsilon'}\rightarrow \N$, which maps $\gamma$ to
the number of intersections with a given page. (This filtration was
pointed out to the authors by Denis Auroux.) Denote by $\gamma_b$
any periodic orbit in $\mathcal{P}^{> 0}_{\varepsilon,\varepsilon'}$
such that $F(\gamma_b)=b$.  The following lemma shows that the
boundary map is filtration nonincreasing.

\begin{lemma}\label{lemma: filtration restriction}
There are no holomorphic cylinders from $\gamma_{b}$ to
$\gamma_{b'}$ if $b<b'$.
\end{lemma}

\begin{proof}
The holomorphic cylinders intersect the binding positively.  (We may
need to perturb the holomorphic cylinder to also make it intersect
the binding transversely.) If there is a holomorphic cylinder
$\tilde u$ from $\gamma_b$ to $\gamma_{b'}$, then there is a map
$\overline{u}: P\rightarrow M-N(\gamma_0)$, obtained by a cutting-up
process given in Section~\ref{subsection: cutting up holom
cylinder}.  By examining the intersection number of $\bdry
\overline{u}(P)$ with $S_0'\times\{0\}$, we see that $b\geq b'$.
\end{proof}

The main theorem of this section is the following:

\begin{theorem}\label{theorem: cylinder}
Suppose $c_i\geq {3\over n_i}$ for each boundary component $(\bdry
S)_i$. Given $N\gg 0$, for sufficiently small $\varepsilon,
\varepsilon'>0$, there are no holomorphic cylinders from $\gamma$ to
$\gamma'$ if $\int_{\gamma}\alpha_{\varepsilon,\varepsilon'},
\int_{\gamma'}\alpha_{\varepsilon,\varepsilon'}\leq N$, and one of
the following holds: \be\item $\gamma=\gamma_{b}$,
$\gamma'=\gamma_{b'}$, and $b\not=b'$;
\item $\gamma=\gamma_b$ and $\gamma'=m_0\gamma_0$;
\item $\gamma=m_0\gamma_0$ and $\gamma'=\gamma_b$;
\item $\gamma=m_0\gamma_0$, $\gamma'=m_1\gamma_0$, and $m_0\not=m_1$.
\ee
\end{theorem}

The rest of this section is devoted to proving Theorem~\ref{theorem:
cylinder}.  The basic idea is exactly the same as the proof of
Theorem~\ref{theorem: nodisks}.

\subsection{Cutting up the holomorphic cylinder} \label{subsection:
cutting up holom cylinder}

Suppose that $\gamma =\gamma_b$ and $\gamma' =\gamma_{b'}$. By
Lemma~\ref{lemma: filtration restriction}, $b<b'$ is not possible,
so assume that $b>b'$.

Suppose there is a holomorphic cylinder $\tilde u=(a,u): S^1\times
\R \rightarrow \R\times M$ from $\gamma$ to $\gamma'$. Again, with
the aid of the asymptotics from \cite{HWZ1}, we view $u$ as a smooth
map $\overline u: S^1\times [0,1]\rightarrow M$ where:
\begin{itemize}
\item $\overline u(S^1\times\{1\})=\gamma$ and $\overline
u(S^1\times\{0\})=\gamma'$.  Moreover, the orientation on
$u(S^1\times\{1\})$ induced from $S^1\times[0,1]$ agrees
with that of $\gamma$ and the orientation on $u(S^1\times\{0\})$
induced from $S^1\times[0,1]$ is opposite that of $\gamma'$.
\item $\overline{u}|_{int(S^1\times[0,1])}$ is immersed away from
a finite number of points in $int(S^1\times[0,1])$.
\item At points where $\overline{u}$ is immersed, $\overline{u}$
is positively transverse to $R$.
\item $\overline{u}(z)\not\in \operatorname{Im}(\gamma)\cup
\operatorname{Im}(\gamma')$ for $z\in S^1\times([0,r_0]\cup
[1-r_0,1])$, where $r_0$ is a small positive number.
\end{itemize}
As before, perturb $\overline{u}$ so that $\overline{u}$ is still
positively transverse to $R$ away from $\bdry (S^1\times [0,1])$ and
a finite set $F$ in $int(S^1\times[0,1])$, points in $F$ are complex
branch points, and $(\Pi\circ \overline{u}_*)(z)\not=0$ for all $z$
with $\overline{u}(z)$ in a sufficiently small neighborhood
$N(\gamma_0)$ of $\gamma_0$.  Let $P$ be the planar subsurface of
$S^1\times[0,1]$ obtained by excising $z\in S^1\times[0,1]$ such
that $\overline{u}(z)\in int (N(\gamma_0))$. We write $\bdry P =
\bdry_0 P +\bdry_1 P + \bdry_2 P$, where $\bdry_0 P$ maps to
$\gamma$, $\bdry_1 P$ maps to $-\gamma'$, and the components of
$\bdry_2 P$ map to $\bdry N(\gamma_0)$.

We now consider the intersection of $S_0'=S_0'\times\{0\}$ and
$\overline{u}(P)$, which we may assume to be a transverse
intersection.  Then the set of points of $P$ which map to $S_0'$
under $\overline{u}$ are properly embedded arcs and embedded closed
curves.  By the positivity of intersection of $\bdry_0 P$, $-\bdry_1
P$, and each component of $-\bdry_2 P$ with $S_0'$, we find that the
embedded closed curves bound disks in $P$, hence can be isotoped
away as before.

Therefore, the holomorphic cylinder $\tilde u$ from $\gamma_b$ to
$\gamma_{b'}$ gives rise to a map $\overline{u}:P\rightarrow M-
N(\gamma_0 )$, where $P$ is a planar surface and $N(\gamma_0)$ a
small tubular neighborhood of the binding $\gamma_0$, such that:
\begin{itemize}
\item $\overline{u}$ is an immersion away from a finite number
of points;
\item $\overline{u}(\partial P )=\gamma_b \cup \gamma_{b'}^{-1}
\cup \overline{u}(\partial_2 P)$, where $\overline{u}(\partial_2 P)
\subset \partial N(\gamma_0)$ and consists of $b-b'$ closed curves
which are parallel to and oriented in the opposite direction from
the boundary of the meridian disk of $N(\gamma_0)$ which intersects
$\gamma_0$ positively;
\item $\overline{u}(P) \cap (S_0'\times \{0\})$ consists of $b$
properly embedded arcs. Here $b'$ arcs begin on $\gamma_{b'}$,
$b-b'$ begin on $\partial_2 P$, and all end on $\gamma_b$.
\end{itemize}
The arcs of $\overline{u}(P) \cap (S_0'\times \{ 0\})$ cut $P$ into
$b'$ disks. See Figure~\ref{bdrymap}.

\begin{figure}[ht]
\begin{overpic}[height=2.7in]{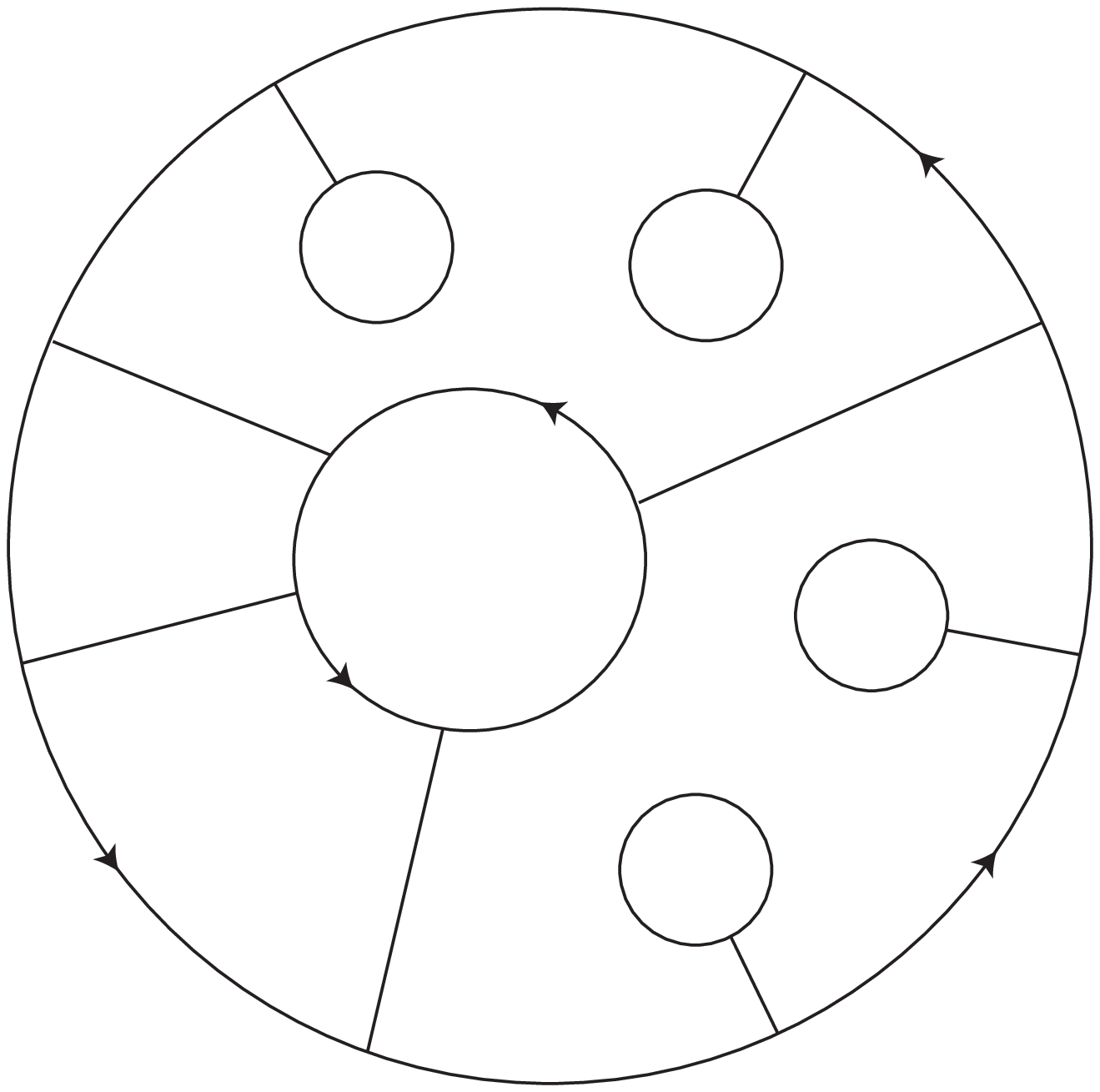}
\put(13.2,9.2){\tiny$\gamma$} 
\put(22.5,85){\tiny$\alpha_4$}
\put(38.7,19){\tiny$\tau$} \put(70,9.5){\tiny$\alpha_1$}
\put(57.2,38){\tiny$\gamma'$} \put(65,88){\tiny$\alpha_3$}
\put(90.15,43.1){\tiny$\alpha_2$} \put(84.5,19.7){\tiny$\gamma$}
\end{overpic}
\caption{} \label{bdrymap}
\end{figure}

We select one arc amongst the $b'$ arcs beginning on $\gamma_{b'}$,
and let $\tau$ be the image of the arc under $\overline{u}$.
Consider the disk $\mathcal{D}$ obtained by cutting $P$ along (the
arc that maps to) $\tau$ as well as along the $b-b'$ arcs from
$\overline{u}(\partial_1 P)$ to $\gamma_b$. Also denote the images
of the $b-b'$ arcs from $\bdry_2 P$ to $\gamma_b$ under
$\overline{u}$ by $\alpha_1,\dots,\alpha_{b-b'}$, in
counterclockwise order along $\gamma_b$.  Here, $\alpha_1$ is the
first arc reached from $\tau$, traveling counterclockwise along
$\gamma_b$. The disk $\overline{u}(\mathcal{D} )$ can be thought of
as living in $S_0'\times [0,b']$.  Consider the projection
$\pi:S_0'\times[0,b']\rightarrow S_0'$ onto the first factor.

We consider the curve $\pi (\overline{u} (\partial \mathcal{D}))$
and obtain a contradiction by showing that it is not contractible in
$S_0'$, in a manner analogous to Theorem~\ref{theorem: nodisks}. Let
$\Gamma$ (resp.\ $\Gamma'$) be the subarc of
$\pi(\overline{u}(\partial \mathcal{D}))$ which maps to
$\pi(\gamma_b) $ (resp.\ $\pi(\gamma_{b'})$) in $S_0'$. (By
$\pi(\gamma_b)$ we mean the projection of an appropriate lift of
$\gamma_b$ to $S_0'\times[0,b']$.)

\subsection{Completion of Proof of Theorem~\ref{theorem: cylinder}}
We now complete the proof of Theorem~\ref{theorem: cylinder}.
Suppose $k\geq 3$.

\s\n (1) Suppose $\gamma=\gamma_b$, $\gamma'=\gamma_{b'}$, and
$\gamma,\gamma'\subset \Sigma(S,\psi')$.  The case when at least one
of $\gamma,\gamma'\not\subset \Sigma(S,\psi')$ is similar.  As in
the proof of Theorem~\ref{theorem: nodisks}, we can write:
\begin{equation}
\pi(\overline{u} (\partial \mathcal{D} ))= h^{-1} (\xi_1^{-1} )\xi_1
h^{-1} (\xi_2^{-1} )\xi_2\cdots h^{-1} (\xi_{b-b'}^{-1})\xi_{b-b'}
\Gamma h^{-b'} (\tau^{-1}) (\Gamma')^{-1} \tau.
\end{equation}
Here we are writing $\tau$ for $\pi(\tau)$, $\Gamma=
h^{b-1}(\gamma_1)\cdots \gamma_b$ and
$\Gamma'=h^{b'-1}(\gamma_1')\cdots\gamma'_{b'}$.

Next, we apply $h^{-b+1}$ and rewrite our equation as:
\begin{equation}
h^{-b+1}(\pi(\overline{u}(\partial \mathcal{D}))) = R' \Gamma_1
(h^{-b+1}(\Gamma)) h^{-b'}(\kappa^{-1})(h^{-b+1}(\Gamma')^{-1})
\kappa,
\end{equation}
where
\begin{eqnarray*}
\Gamma_1 &= & h_0^{-1} (\eta_1^{-1})\eta_1 h_0^{-1} (\eta_2^{-1})
\eta_2\cdots h_0^{-1} (\eta_{b-b'}^{-1})\eta_{b-b'},
\end{eqnarray*}
and $R'$ is of the type which appears in Equation~\ref{equation:
concat with Q final version}.

We will apply the retraction $\rho:S_0\rightarrow S$ if necessary,
without further mention, and work on $S$. By taking a sufficiently
large cover of the holomorphic cylinder from $\gamma$ to $\gamma'$
(and replacing $\gamma$ and $\gamma'$ by $K\gamma$ and $K\gamma'$),
we may assume that $b-b'$ is sufficiently large. Hence,
$$\Phi(\Gamma_1)\geq
(b-b')(k-1)-(b-b')+1\gg 0.$$

Next we note that $\Phi(R')=0$ by Corollary~\ref{corollary: phi of g
is zero}.  Also, by Proposition~\ref{prop: genuine trajectory},
$\Phi((\psi')^{-b+1}(\Gamma))=0$, since
$(\psi')^{-b+1}(\Gamma)=\gamma_1\cdots h^{-b+1}(\gamma_{b})$.
Although $(\psi')^{-b+1}(\Gamma')=h^{b'-b}(\gamma'_1)\cdots
h^{-b+1}(\gamma'_{b'})$ is not quite a concatenation of type $Q$,
the same proof shows that $\Phi((\psi')^{-b+1}(\Gamma')))=0$.
Finally, $\Phi(\kappa)=-\Phi(\psi^{-b'}(\kappa^{-1}))$, and the
difference between $\psi^{-b'}(\kappa^{-1})$ and
$(\psi')^{-b'}(\kappa^{-1})$ is two arcs of the type $R'$.  Since
these two arcs of type $R'$ have $\Phi=0$, we have
$$\Phi(h^{-b+1}(\pi(\overline{u}(\partial \mathcal{D}))))
\approx \Phi(\Gamma_1)\gg 0.$$
This is a contradiction.

\s\n (2) Suppose $\gamma=\gamma_b$ and $\gamma'=m_0\gamma_0$. As in
Section~\ref{disks}, let $N(\gamma_0)$ be a small tubular
neighborhood of $\gamma_0$ so that $(\Pi\circ \overline
u_*)(q)\not=0$ for all $q$ with $\overline u(q)\in N(\gamma_0)$. We
retract the cylinder so that $\gamma$ is fixed but $\gamma'$ now is
on $\bdry N(\gamma_0)$. Since $\bdry N(\gamma_0)$ is foliated by
Reeb orbits of irrational slope $c$ and we require that the cylinder
be positively transverse to the Reeb vector field, it follows that
$\gamma'$ has slope ${m_0\over m_1}$ which satisfies ${m_1\over
m_0}>{1\over c}$. Therefore, $\Phi(\Gamma')$ contributes negatively,
and hence $\Phi((\Gamma')^{-1})$ contributes positively, which in
our favor.

\s\n (3) Suppose $\gamma=m_0\gamma_0$ and $\gamma'=\gamma_b$.
Consider $N(\gamma_0)$ as in (2). This time, we retract the cylinder
so that $\gamma$ is on $\bdry N(\gamma_0)$ and $\gamma'$ is fixed.
Then $\gamma$ has slope ${m_0\over m_1}$ which satisfies ${m_1\over
m_0}<{1\over c}$, and $\Phi(\Gamma)$ contributes more positively,
which is in our favor.

\s\n (4) This just combines the observations made in (2) and (3).

\s\n This completes the proof of Theorem~\ref{theorem: cylinder}.

\section{Direct limits} \label{section: direct limits}

\subsection{Direct limits in contact homology}
Let $\alpha$ and $\alpha'$ be contact $1$-forms for the same contact
structure $\xi$, with nondegenerate Reeb vector fields $R_\alpha,
R_{\alpha'}$. Denote by $\A_{\leq L}(\alpha)$ the supercommutative
$\Q$-algebra with unit generated by $\mathcal{P}^{\leq L}_\alpha$,
the set of good orbits of $R_\alpha$ with action $\int \alpha\leq
L$. The boundary map $\bdry$ is the restriction of
$\bdry:\A(\alpha)\rightarrow \A(\alpha)$ to $\A_{\leq
L}(\alpha)\subset \A(\alpha)$.

Write $\alpha'=f_0(x)\alpha$, where $f_0(x)$ is a positive function.
If $K$ is a constant satisfying $K> \sup_{x\in M} f_0(x)$, then a
sufficient condition for the existence of a chain map
$$\Phi_{\alpha\alpha'}:\A_{\leq L}(\alpha)\rightarrow \A_{\leq
L'}(\alpha')$$ is that $L'> KL$, as will be explained in the next
paragraph.

Consider $\R\times M$ with coordinates $(t,x)$.  We define a
function $f(t,x)$ for which (i) ${\bdry f\over \bdry t}>0$, (ii)
$f(t,x)\rightarrow K$ as $t\rightarrow +\infty$, (iii)
$f(t,x)\rightarrow f_0(x)$ as $t\rightarrow -\infty$, (iv) $f(t,x)$
does not depend on $x$ for $t$ large positive, and (v)
$f(t,x)=g(t)f_0(x)$ for $t$ large negative (this means that $g(t)$
is a function which approaches $1$ as $t\rightarrow -\infty$ and has
small positive derivative ${dg\over dt}$).  Then define the
symplectization $d(f(t,x)\alpha)$. Let $J$ be an almost complex
structure which is adapted to the symplectization at both ends. Take
the collections $\mathcal{P}_\alpha$, $\mathcal{P}_{\alpha'}$ of the
good orbits for $\alpha$ and $\alpha'$, respectively.  Let
$\mathcal{M}_{[Z]}(J,\gamma,\gamma'_1,\dots,\gamma'_m)$ be the
moduli space of $J$-holomorphic rational curves with (asymptotically
marked) punctures which limit to $\gamma\in \mathcal{P}_\alpha$ at
the positive end and to
$\gamma'_1,\dots,\gamma_m'\in\mathcal{P}_{\alpha'}$ at the negative
end. Then define
$$\Phi_{\alpha\alpha'}(\gamma)=\sum {n_{\gamma,\gamma'_1,\dots,\gamma_m'}
\over (i_1)!\dots (i_l)! \kappa(\gamma'_1)\dots
\kappa(\gamma'_m)}\gamma'_1\dots\gamma_m',$$ where the sum is over
all unordered tuples
$\overline{\gamma'}=(\gamma_1',\dots,\gamma_m')$ and homology
classes $[Z]\in H_2(M,\gamma\cup\gamma')$ so that the moduli space
$\mathcal{M}_{[Z]}(J,\gamma,\gamma'_1,\dots,\gamma'_m)$ is
$0$-dimensional. Here $n_{\gamma,\gamma'_1,\dots,\gamma'_m}$ is a
signed count of points in
$\mathcal{M}_{[Z]}(J,\gamma,\gamma'_1,\dots,\gamma'_m)$,
$\kappa(\gamma)$ is the multiplicity of $\gamma$, and
$i_1,\dots,i_l$ denote the number of occurrences of each orbit
$\gamma_i'$ in the list $\gamma_1',\dots,\gamma_m'$. By Stokes'
theorem applied to $d(f(t,x)\alpha)$, we see that if there is a
holomorphic curve from $\gamma$ to $\gamma'_1,\dots,\gamma'_m$, then
$K\int_\gamma\alpha \geq \sum_{i=1}^m\int_{\gamma'_i} f_0\alpha$.
Hence if $L'> KL$, then $\Phi_{\alpha\alpha'}$ is well-defined.
Moreover, $\Phi_{\alpha\alpha'}$ is a chain map, as can easily be
seen by analyzing the breaking of $1$-dimensional moduli spaces
$\mathcal{M}_{[Z]}(J,\gamma,\gamma'_1,\dots,\gamma'_m)$. There is an
induced map on the full contact homology:
$$\Phi_{\alpha\alpha'}:FHC_{\leq L}(M,\alpha)\rightarrow FHC_{\leq L'}(M,\alpha').$$
In this paper we will use the same notation $\Phi_{\alpha\alpha'}$
for the map on the chain level and map on the level of homology; it
should be clear from the context which we are referring to.

We now discuss direct limits. Fix a nondegenerate contact $1$-form
$\alpha$ for $(M,\xi)$. Let $\{\alpha_i=f_i\alpha\}$, $i=1,2,\dots$,
be a collection of contact $1$-forms and let $M_i=\sup \{f_i(x),
{1\over f_i(x)}~|~x\in M\}$. We say that the sequence $(\alpha_i,
L_i)$ is {\em exhaustive} if there is a sequence $L_i\rightarrow
\infty$ such that $$L_{i+1} > C M_i M_{i+1}L_i,$$ where $C>1$ is a
constant.

\begin{prop} \label{prop: invariance of direct limit}
Suppose the sequence $(\alpha_i,L_i)$ is exhaustive. Then the direct
limit $$\displaystyle\lim_{i\rightarrow \infty} FHC_{\leq
L_i}(\alpha_i)$$ exists.  Moreover,
$$\Phi:FHC(\alpha)\stackrel\sim\rightarrow \displaystyle\lim_{i\rightarrow
\infty} FHC_{\leq L_i}(\alpha_i).$$
\end{prop}

This implies that the direct limit calculates the full contact
homology of $(M,\xi)$, and is independent of the particular choice
of nondegenerate contact $1$-form.

\begin{proof}
Suppose $(\alpha_i,L_i)$ is exhaustive. Then the chain maps
$\Phi_{\alpha_i\alpha_{i+1}}:\A(\alpha_i)\rightarrow
\A(\alpha_{i+1})$ restrict to
$$\Phi_{\alpha_i\alpha_{i+1}}:\A_{\leq L_i}(\alpha_i)\rightarrow \A_{\leq
L_{i+1}}(\alpha_{i+1}),$$ since $L_{i+1} > (\sup_{x\in M}
{f_{i+1}(x)\over f_i(x)})\cdot L_i$ by the exhaustive condition.
Hence the direct limit exists.

Next we show that for any $N\gg 0$ there exists a pair
$(\alpha_i,L_i)$ so that $$\Phi_{\alpha\alpha_i}: \A_{\leq
N}(\alpha)\rightarrow \A_{\leq L_i}(\alpha_i),$$ obtained by
counting rigid marked rational holomorphic curves, is a chain map.
Since ${L_i\over M_i}> L_{i-1}$ and $L_{i-1}\rightarrow \infty$ as
$i\rightarrow \infty$, there is a symplectization from $\alpha$ to
$\alpha_i$ so that $\gamma$ with $A_\alpha(\gamma)\leq N$ is mapped
to $\overline{\gamma'}$ with $A_{\alpha_i}(\overline{\gamma'})\leq
L_i$. In fact, we simply need ${L_i\over M_i}>N$.


Now, $\Phi_{\alpha_i\alpha_{i+1}}\circ\Phi_{\alpha\alpha_i}$ and
$\Phi_{\alpha\alpha_{i+1}}$ are chain homotopic by the usual
argument, so the collection of maps $\Phi_{\alpha\alpha_i}$ induces
the map $$\Phi_{\leq N}: FHC_{\leq N}(\alpha) \rightarrow
\displaystyle\lim_{i\rightarrow \infty} FHC_{\leq L_i}(\alpha_i)$$
on the level of homology. Now, it is easy to see that
$FHC(\alpha)=\lim_{i\rightarrow \infty} FHC_{\leq N_i}(\alpha)$,
provided $N_i\rightarrow\infty$ (and is increasing).  By the usual
chain homotopy argument, $\Phi_{\leq N_i}$ is equal to the
composition of $FHC_{\leq N_i}(\alpha)\rightarrow FHC_{\leq
N_{i+1}}(\alpha)$ followed $\Phi_{\leq N_{i+1}}$. Hence, by the
universal property of direct limits, we obtain the map
$$\Phi: FHC(\alpha) \rightarrow \displaystyle\lim_{i\rightarrow
\infty} FHC_{\leq L_i}(\alpha_i).$$

Finally, to prove that $\Phi$ is an isomorphism, we use the usual
chain homotopy argument.  Given $(\alpha_i,L_i)$, there exist $N_i$
and $i'$ so that $L_i M_i< N_i <{L_{i'}\over M_{i'}}$.  Hence we
have maps $\A_{\leq L_i}(\alpha_i)\rightarrow \A_{\leq N_i}(\alpha)$
and $\A_{\leq N_i}(\alpha)\rightarrow \A_{\leq
L_{i'}}(\alpha_{i'})$, and their composition is chain homotopic to
$$\Phi_{\alpha_i \alpha_{i'}}: \A_{\leq L_i}(\alpha_i)\rightarrow
\A_{\leq L_{i'}}(\alpha_{i'}).$$ Therefore, the composition $$
\displaystyle\lim_{i\rightarrow \infty} FHC_{\leq
L_i}(\alpha_i)\stackrel{\Psi}\rightarrow FHC(\alpha)
\stackrel\Phi\rightarrow \displaystyle\lim_{i\rightarrow \infty}
FHC_{\leq L_i}(\alpha_i)$$ is equal to the identity map.  This gives
a right inverse of $\Phi$; the left inverse is argued similarly.
\end{proof}

\subsection{Verification of the exhaustive condition}
\label{subsection: verification of exhaustive}

The goal of this subsection is to show the existence of an
exhaustive sequence $(\alpha_i, L_i)$, where the $\alpha_i$ are all
adapted to the same open book $(S,h)$ with pseudo-Anosov monodromy
and fractional Dehn twist coefficient $c\geq {2\over n}$, so that
the direct limit process can be applied. Let $C_{\leq
L_i}(\alpha_i)$ be the $\Q$-vector space generated by
$\mathcal{P}^{\leq L_i}_{\alpha_i}$.

Let $\alpha_{\varepsilon,\varepsilon'}$ be the contact $1$-form
defined in Section~\ref{section: construction}.  In what follows, we
perturb $\alpha_{\varepsilon,\varepsilon'}$ with respect to a
suitable large constant $N\gg 0$, as in Lemma~\ref{lemma:
perturbation}.  For simplicity of notation, we will still call the
perturbed $1$-form $\alpha_{\varepsilon,\varepsilon'}$.

\begin{prop} \label{prop: direct limit}
Given a sequence $L_i$, $i=1,2,\dots$, going to $\infty$, there is a
sequence of contact $1$-forms
$\alpha_{\varepsilon_i,\varepsilon_i'}$, $i\in \N$, with
$\varepsilon_i$, $\varepsilon_i'\rightarrow 0$, so that:
\begin{enumerate}
\item The chain groups $C_{\leq L_i}(\alpha_{\varepsilon_i,
\varepsilon_i'} )$ are cylindrical.

\item There exists an isotopy $(\varphi^i_s )_{s\in [0,1]}$ of $M$
so that $(\varphi_1^i )^*\alpha_{\varepsilon_i ,\varepsilon_i'} =G_i
\alpha_{\varepsilon_0 ,\varepsilon_0'}$ and ${1\over 4^i}\leq
G_i\leq 4^i$.
\end{enumerate}
\end{prop}

We now apply Proposition~\ref{prop: direct limit} to obtain an
exhaustive sequence: In our situation $M_i=4^i$. Pick $L_i$ so that
the exhaustive condition is satisfied. By Proposition~\ref{prop:
direct limit}, there exist
$\alpha_i=(\varphi_1^i)^*\alpha_{\varepsilon,\varepsilon'}=
G_i\alpha_{\varepsilon_0,\varepsilon_0'}$ so that $(\alpha_i,L_i)$
is exhaustive.

We first prove some preparatory lemmas, which are proved for the
{\em unperturbed} $\alpha_{\varepsilon,\varepsilon'}$; however, the
same results are also true for the perturbed
$\alpha_{\varepsilon,\varepsilon'}$, since the perturbation is a
$C^\infty$-small one.

\begin{lemma}\label{lemma: estimation}
On $\Sigma (S,\psi' )$, the quantity $|\beta_t(Y_{\varepsilon'})|$
is bounded from above by a constant which is independent of
$0<\varepsilon' <1$ (and of course independent of $\varepsilon$).
\end{lemma}

\begin{proof}
For technical reasons, we specialize the function $f_{\varepsilon'}:
S\rightarrow \R$, defined in Section~\ref{subsub: N}, on the region
$A=S^1\times[0,1]$. In particular, we require that ${\bdry
f_{\varepsilon'}\over \bdry y}\leq -1$ when $y\in [y_0,y_1]$, where
$f_{\varepsilon'}(y_0)> {1\over 2}$ and
$f_{\varepsilon'}(y_1)=2\varepsilon'$.

We first restrict to the subset (away from $\bdry \Sigma(S,\psi')$)
where $f_{\varepsilon'}\leq 2\varepsilon'$.  Suppose $t\in[0,{1\over
2}]$. Recall that $i_{Y_{\varepsilon'}} \omega_t =-\dot{\beta}_t$
(Equation~\ref{equation: Y varepsilon}) and
$\dot\beta_t=\chi'_0(t)(\beta-f_{\varepsilon'}(g_*\beta))$
(Equation~\ref{eqn: beta dot}). The quantity $|\dot\beta_t|$ is
bounded above by a constant independent of $\varepsilon$, since
$\chi'_0(t)$, $\beta$, $g_*\beta$, and $f_{\varepsilon'}$ are all
bounded above. Next,
$$\omega_t=(1-\chi_0(t))d(f_{\varepsilon'}(g_*\beta))
+\chi_0(t)d\beta.$$  Clearly, $d\beta$ is bounded from below.  On
the other hand,
$d(f_{\varepsilon'}g_*\beta)=\varepsilon'd(g_*\beta)$ on $S-A$ and
$(f_{\varepsilon'}-{\bdry f_{\varepsilon'}\over \bdry
y}(C-y))d\theta dy$ on $S^1\times [y_1,1]$.  Hence
$d(f_{\varepsilon'}g_*\beta)$ is bounded below by $\varepsilon'$
times a positive constant. This means that $|Y_{\varepsilon'} (t)|$
is bounded above by $$\frac{C_0}{(1-\chi_0(t))\varepsilon' C_1
+\chi_0(t) C_2},$$ where $C_0, C_1, C_2>0$ are constants. Since
$\beta_t = (1-\chi_0(t))f_{\varepsilon'} (g_* \beta)
+\chi_0(t)\beta$, we obtain
$$|\beta_t (Y_{\varepsilon'})|\leq
\frac{(1-\chi_0(t))\varepsilon'C_3 +\chi_0(t)
C_4}{(1-\chi_0(t))\varepsilon' C_5 +\chi_0(t)C_6},$$ where
$C_3,\dots, C_6 >0$ are constants. The expression on the right-hand
side is bounded above by a constant independent of $\varepsilon'$.
The situation $t\in [\frac{1}{2} ,1]$ is treated similarly.

Now, on the subset $f_{\varepsilon'}\geq 2\varepsilon'$,
$d(f_{\varepsilon'}(g_*\beta))= (f_{\varepsilon'}-{\bdry
f_{\varepsilon'}\over \bdry y}(C-y))d\theta dy$, and ${\bdry
f_{\varepsilon'}\over \bdry y}<-1$ or $f_{\varepsilon'}>{1\over 2}$.
Hence $|d(f_{\varepsilon'}(g_*\beta))|$ is bounded below by a
positive constant which is independent of $\varepsilon'$.  Hence
$Y_{\varepsilon'}$ is bounded above by a constant independent of
$\varepsilon'$, and the conclusion follows easily.
\end{proof}

As a consequence of Lemma~\ref{lemma: estimation}, if $\varepsilon$
is sufficiently small, then
$\alpha_{\varepsilon,\varepsilon'}({\bdry\over \bdry t}
+Y_{\varepsilon'}) =1+\varepsilon \beta_t(Y_{\varepsilon'})$ is
bounded from below by a positive constant.

\s
Now we recall Moser's method: Let $(\alpha_s)_{s\in [0,1]}$ be a
path of contact $1$-forms. We are looking for an isotopy $(\phi_s
)_{s\in [0,1]}$ such that $\phi_s^* \alpha_s =H_s \alpha_0$. If
$X_s$ is a time-dependent vector field which generates $\phi_s$,
then it satisfies the equation:
$$\phi_s^*(\dot{\alpha}_{s} +\mathcal{L}_{X_s} \alpha_{s}) =\dot{H}_s \alpha_0,$$
where the dot means the derivative in the $s$-variable (at time
$s$). Using the relation $\phi_s^* \alpha_s =H_s \alpha_0$, this can
be rewritten as
\begin{equation} \label{eqn: moser}
\dot{\alpha}_{s} +\mathcal{L}_{X_s} \alpha_{s}= G_s\alpha_s,
\end{equation}
where $G_s=({d\over ds}\log H_s)\circ \phi_s^{-1}$. It will be
convenient to choose $X_s$ to be in $\ker \alpha_s$, in which case
$\mathcal{L}_{X_s} \alpha_s =i_{X_s} d\alpha_s$.

\begin{lemma}\label{lemma: estimation0}
For every $0<\varepsilon' <1$, one can find $\delta_1(\varepsilon' )
>0$ so that, for every $0<\varepsilon_1
<\varepsilon_0 <\delta_1(\varepsilon' )$, the $1$-forms
$\alpha_{\varepsilon_0, \varepsilon'}$ and $\alpha_{\varepsilon_1,
\varepsilon'}$ are contact and there exists an isotopy $(\phi_s
)_{s\in [0,1]}$ of $M$ starting from the identity such that
$\phi_1^* \alpha_{\varepsilon_1, \varepsilon'} =H
\alpha_{\varepsilon_0, \varepsilon'}$, with ${1\over 2} \leq H\leq
2$.
\end{lemma}

\begin{proof}
We first work on $\Sigma (S,\psi' )$. Apply Moser's method to the
path of contact $1$-forms given by $\alpha_{\varepsilon_s,
\varepsilon'}$, where $\varepsilon_s =(1-s)\varepsilon_0
+s\varepsilon_1$ and $s\in [0,1]$. The infinitesimal generator $X_s
\in \ker \alpha_{\varepsilon_s,\varepsilon'}$ of the isotopy
$\phi_s$ satisfies: $\dot{\alpha}_{\varepsilon_s,\varepsilon'} +
i_{X_s} d\alpha_{\varepsilon_s,\varepsilon'} =G_s
\alpha_{\varepsilon_s,\varepsilon'}.$ If we evaluate this equation
on $\frac{\partial}{\partial t} +Y_{\varepsilon'}$, we obtain
$$(\varepsilon_1 -\varepsilon_0) \beta_t (Y_{\varepsilon'} )= G_s
(1+\varepsilon_s \beta_t (Y_{\varepsilon'} )).$$ By
Lemma~\ref{lemma: estimation}, $|\beta_t (Y_{\varepsilon'})|$ is
bounded above and $|1+\varepsilon_s \beta_t (Y_{\varepsilon'} )|$ is
bounded below by a positive constant, provided we take
$\varepsilon_0$ and $\varepsilon_1$ small enough. Hence $|G_s|$ and
$|{d\over ds}\log H_s|$ are bounded above by a small constant. This
implies that ${1\over C}<H_s<C$ for $C>1$, say $C=2$.

In the neighborhood $N(K)=\R/\Z\times D^2$ of the binding,
$\alpha_{\varepsilon,\varepsilon'}$ is of the form
$a_\varepsilon(r)dz+b_\varepsilon(r)d\theta$, according to
Section~\ref{subsub: extension to binding}.  For sufficiently small
$\varepsilon_0, \varepsilon_1$,
$\alpha_{\varepsilon_0,\varepsilon'}$ and
$\alpha_{\varepsilon_1,\varepsilon'}$ are close, and so are
$R_{\varepsilon_0,\varepsilon'}$ and
$R_{\varepsilon_1,\varepsilon'}$.  Hence, the left-hand side of
Equation~\ref{eqn: moser}, evaluated on the Reeb vector field $R_s$,
can be made arbitrarily small. Hence we conclude that $H$ is
arbitrarily close to $1$ near the binding.
\end{proof}

\begin{lemma}\label{lemma: estimation1}
For every $0<\varepsilon_1' <\varepsilon_0'<1$, there exists
$\delta_2 (\varepsilon_1' )>0$ so that, for every $0<\varepsilon_1
<\delta_2 (\varepsilon_1' )$, there exists an isotopy $(\phi_s'
)_{s\in [0,1]}$ of $M$ so that $(\phi_1')^* \alpha_{\varepsilon_1,
\varepsilon_1'} =H'\alpha_{\varepsilon_1,\varepsilon_0'}$, with
${1\over 2}\leq H'\leq 2$.
\end{lemma}

\begin{proof}
We can concentrate our attention on $\Sigma (S,\psi' )$, since
$\alpha_{\varepsilon ,\varepsilon'}$ does not depend on
$\varepsilon'$ on $N(K)$.  Given $0<\varepsilon_0',\varepsilon_1'<1$
and  $s\in [0,1]$, let $\varepsilon_s'=(1-s)\varepsilon_0'
+s\varepsilon_1'$. By Equation~\ref{eqn: moser} applied to the path
$\alpha_{\varepsilon_1, \varepsilon_s'}$, we obtain:
$$\varepsilon_1 \dot{\beta_t} (Y_{\varepsilon_s'} ) = G_s
(1+\varepsilon_1 \beta_t (Y_{\varepsilon_s'})),$$ where the dot is
the derivative in the $s$-variable. As before, we see that if
$\varepsilon_1$ is small enough, then $|1+\varepsilon_1 \beta_t
(Y_{\varepsilon_s'})|$ is bounded below by a positive constant.
Since $\dot\beta_t$ and $Y_{\varepsilon'_s}$ do not depend
on $\varepsilon_1$, $|G_s|$ is bounded above by a small constant,
provided $\delta_2(\varepsilon'_1)$ is sufficiently small. Again,
this implies that ${1\over C}<H_s'<C$ for $C>1$, say $C=2$.
\end{proof}

We are now ready to prove Proposition~\ref{prop: direct limit}.

\begin{proof}[Proof of Proposition~\ref{prop: direct limit}]
First use Theorem~\ref{theorem: nodisks} to choose sequences
$(\varepsilon_i)_{i\in \N}$ and $(\varepsilon_i')_{i\in \N}$ so that
no closed orbit of the Reeb vector field $R_{\varepsilon_i,
\varepsilon_i'}$ which intersects a page $\leq L_i$ times bounds a
finite energy plane. (Here we are using the perturbed
$\alpha_{\varepsilon_i,\varepsilon'_i}$.) Next, after shrinking
$\varepsilon_i$ if necessary, suppose $\varepsilon_0 <\delta_1
(\varepsilon_0' )$ and $(\varepsilon_i )_{i\in \N}$ is a decreasing
sequence which satisfies $\varepsilon_i <\inf \{ \delta_1
(\varepsilon_{i-1}' ), \delta_1 (\varepsilon_i'), \delta_2
(\varepsilon_i' )\}$.  By Lemma~\ref{lemma:uniform}, we may also
assume that $\varepsilon_i$ is sufficiently small so that
$R_{\varepsilon_i,\varepsilon_i'}$ is arbitrarily close to
$\frac{\partial}{\partial t}+ Y_{\varepsilon_i'}$ on $\Sigma
(S,\psi' )$ and the action is almost the same as the number of
intersections with a page.  Hence $C_{\leq
L_i}(\alpha_{\varepsilon_i, \varepsilon_i'} )$ is cylindrical.  Now,
if we compose the two isotopies given by Lemmas~\ref{lemma:
estimation0} and \ref{lemma: estimation1}, we find an isotopy whose
time $1$ map pulls $\alpha_{\varepsilon_{i+1}, \varepsilon_{i+1}'}$
back to $H_i \alpha_{\varepsilon_i, \varepsilon_i'}$, with ${1\over
4}\leq H_i\leq 4$.  The composition of all these isotopies produces
an isotopy whose time $1$ map pulls $\alpha_{\varepsilon_i,
\varepsilon_i'}$ back to $G_i \alpha_{\varepsilon_0,
\varepsilon_0'}$ with ${1\over 4^i} \leq G_i \leq 4^i$.
\end{proof}

\subsection{Proof of Theorem~\ref{main}(1)}
Suppose $\bdry S$ is connected. Let $(S,h)$ be the open book, where
$h$ is freely homotopic to the pseudo-Anosov $\psi$ and has
fractional Dehn twist coefficient $c={k\over n}$.

Suppose $k\geq 2$.  By Proposition~\ref{prop: direct limit} there
exists an exhaustive sequence $\{(\alpha_i, L_i)\}_{i=1}^\infty$
adapted to $(S,h)$, so that each $\Q$-vector space $C_{\leq
L_i}(\alpha_i)$ is cylindrical. There are chain maps
$$\Phi_{\alpha_i\alpha_{i+1}}^{cyl}:C_{\leq L_i}(\alpha_i)\rightarrow C_{\leq
L_{i+1}}(\alpha_{i+1}),$$ which count rigid holomorphic cylinders in
the symplectization from $\alpha_i$ to $\alpha_{i+1}$. Let
$\displaystyle\lim_{i\rightarrow \infty} HC_{\leq L_i}(\alpha_i)$ be
the direct limit.

Next consider the chain maps
$$\Phi_{\alpha_i\alpha_{i+1}}: \A(\alpha_i)\rightarrow
\A(\alpha_{i+1})$$ which count rigid punctured rational curves in
the symplectization from $\alpha_i$ to $\alpha_{i+1}$. We claim that
no orbit $\gamma$ of $\mathcal{P}_{\alpha_i}^{\leq L_i}$ bounds a
finite energy plane in the symplectization from $\alpha_i$ to
$\alpha_{i+1}$. The argument is identical to that of
Theorem~\ref{theorem: nodisks}, by observing that the almost complex
structure $J$ can be chosen so that $\R$ times the binding
$\gamma_0$ is $J$-holomorphic. (This is possible since we can make
the binding an orbit of the Reeb vector field for each
$f(t_0,x)\alpha$ with $t_0$ fixed.) Therefore, under the maps
$\Phi_{\alpha_i\alpha_{i+1}}$, the trivial augmentation
$\varepsilon_{i+1}$ on $\A_{\leq L_{i+1}}(\alpha_{i+1})$ pulls back
to the trivial augmentation $\varepsilon_i$ of $\A_{\leq
L_i}(\alpha_i)$.

We now prove that $\A(\alpha)$ admits an augmentation $\varepsilon$.
Define $\Phi_{\alpha\alpha_1}$ in the same way as
$\Phi_{\alpha_i\alpha_{i+1}}$. Take $\gamma\in \A(\alpha)$. If we
let
$$\Phi_{\alpha_i}= \Phi_{\alpha_{i-1}\alpha_{i}}\circ
\dots\circ\Phi_{\alpha_1\alpha_2}\circ \Phi_{\alpha\alpha_1},$$
then, for sufficiently large $i$, each term of
$\Phi_{\alpha_i}(\gamma)$ has $\alpha_i$-action $\leq L_i$ by the
exhaustive condition. (Here
$A_{\alpha_i}(a\gamma_1\dots\gamma_m)=\sum_j
A_{\alpha_i}(\gamma_j)$, where $a\in\Q$.) We then define
$\varepsilon(\gamma)= \varepsilon_i\circ \Phi_{\alpha_i}(\gamma)$.
The definition of $\varepsilon(\gamma)$ does not depend on the
choice of sufficiently large $i$, due to the fact that
$\varepsilon_{i+1}$ pulls back to $\varepsilon_i$ under
$\Phi_{\alpha_i\alpha_{i+1}}:\A_{\leq L_i}(\alpha_i)\rightarrow
\A_{\leq L_{i+1}}(\alpha_{i+1})$.

It remains to see that $HC^\varepsilon(\alpha)\simeq
\displaystyle\lim_{i\rightarrow \infty} HC_{\leq L_i}(\alpha_i)$. We
use the same argument as in Proposition~\ref{prop: invariance of
direct limit}.
Given $L_i$, there exist $N_i$ and $i'$ so that there are maps
$\Psi_i:\A_{\leq L_i}(\alpha_i)\rightarrow \A_{\leq N_i}(\alpha)$
and $\Phi_{i'}: \A_{\leq N_i}(\alpha)\rightarrow \A_{\leq
L_{i'}}(\alpha_{i'})$ so that
$\Phi_{i'}^*\varepsilon_{i'}=\varepsilon$.  Hence we have
$$HC^{\Psi_i^*\varepsilon}_{\leq L_i}(\alpha_i)\stackrel{\Psi_i}\rightarrow
HC^\varepsilon_{\leq N_i}(\alpha) \stackrel{\Phi_{i'}}\rightarrow
HC_{\leq L_{i'}}(\alpha_{i'}),$$ whose composition is the map
$HC_{\leq
L_i}(\alpha_i)\stackrel{\Phi_{\alpha_i\alpha_i'}}\longrightarrow
HC_{\leq L_{i'}}(\alpha_{i'})$ by Theorem~\ref{thm: aug}, since
$\Psi_i^*\varepsilon$ is homotopic to the trivial augmentation and
$HC_{\leq L_i}(\alpha_i)\simeq HC^{\Psi_i^*\varepsilon}_{\leq
L_i}(\alpha_i)$. The direct limit of the right-hand side yields
$$\Phi:HC^\varepsilon(\alpha)\rightarrow
\displaystyle\lim_{i\rightarrow \infty} HC_{\leq L_i}(\alpha_i).$$
As before, we have a right inverse of $\Phi$ and a left inverse
exists similarly.

\section{Exponential growth of contact homology} \label{section:
nonvanishing}

\subsection{Periodic points of pseudo-Anosov homeomorphisms}
We collect some known facts about the dynamics of pseudo-Anosov
homeomorphisms. Let $\Sigma$ be a {\em closed} oriented surface and
$\psi$ be a pseudo-Anosov homeomorphism on $\Sigma$. The
homeomorphism $\psi$ is smooth away from the singularities of the
stable/unstable foliations.

A pseudo-Anosov homeomorphism $\psi$ admits a Markov partition
$\{R_1,\dots, R_l\}$ of $\Sigma$, where $R_i=[0,1]\times[0,1]$ are
``birectangles'' with coordinates $(x,y)$, where $y=const$ are
leaves of the unstable foliation $\F^u$ and $x=const$ are leaves of
the stable foliation $\F^s$. (See \cite[Expos\'e~10]{FLP} for
details, including the definition of a Markov partition.)

The Markov partition gives rise to a graph $G$ as follows: the set
of vertices is $\{R_1,\dots, R_l\}$ and there is a directed edge
from $R_i$ to $R_j$ if $int(\psi(R_i))\cap int(R_j)\not=\emptyset$.
The periodic orbits of $\psi$ of order $m$ are in 1-1 correspondence
with cycles of $G$ of length $m$.  (The singular points of $\F^s$ or
$\F^u$ are omitted from this consideration.)  In particular, the
orbits which multiply cover a simple orbit correspond to cycles of
$G$ which multiply cover a ``simple'' cycle of $G$.  As a corollary,
we have the following exponential growth property:

\begin{thm} \label{thm: 1}
There exist constants $A,B>0$ so that the number of periodic orbits
of $\psi$ of period $m$ is greater than $Ae^{Bm}$.  The same is true
for simple periodic orbits or good periodic orbits, i.e., orbits
which are not even multiple covers of hyperbolic orbits with
negative eigenvalues.
\end{thm}

Next we transfer this property to an arbitrary diffeomorphism $h$ of
$\Sigma$ which is homotopic to $\psi$, using {\em Nielsen classes}.
Let $f,g$ be homotopic homeomorphisms of $\Sigma$.  If $x$ is a
periodic point of $f$ and $y$ is a periodic point of $g$, both of
order $m$, then we write $(f,x)\sim(g,y)$ if there exist lifts
$\tilde x$, $\tilde y$ of $x$, $y$ and lifts $\tilde f$, $\tilde g$
of $f,g$ to the universal cover $\wt\Sigma$ such that $d(\tilde
f^k(\tilde x), \tilde g^k (\tilde y))\leq K$ for all $k\in \Z$. Here
$K>0$ is a constant and $d$ is some equivariant metric on
$\wt\Sigma$. Elements $(f,x)$ and $(g,y)$ which satisfy $(f,x)\sim
(g,y)$ are said to be in the same {\em Nielsen class}. Since the
periodic points of $\psi$ belong to different Nielsen classes, we
have the following:

\begin{thm}\label{thm: 2}
For each periodic $(\psi,x)$, there exists at least one $(h,y)$ in
the same Nielsen class. Hence the number of periodic points $h$ of
period $m$ is greater than or equal to the number of periodic points
of $\psi$ of period $m$.
\end{thm}

The above theorem is stated by Thurston in \cite{Th}. A proof can be
found in \cite{Hn}.

Given a diffeomorphism $f:\Sigma\rightarrow \Sigma$ and $x$ a
nondegenerate fixed point of $\Sigma$, its $\pm 1$ contribution to
the Lefschetz fixed point formula is calculated by the sign of
$\det(df(x)-id)$. More precisely, if $df(x)$ is of hyperbolic type
with positive eigenvalues then $\det(df(x)-id)<0$ and the
contribution is $-1$; if $df(x)$ is of hyperbolic type with negative
eigenvalues or of elliptic type, then $\det(df(x)-id)>0$ and the
contribution is $+1$. They correspond to even and odd parity,
respectively, in the contact homology setting. For a pseudo-Anosov
$\psi$, the sum of contributions (in the Lefschetz fixed point
theorem) of a periodic orbit $x$ of period $m$ in a particular
Nielsen class is $\pm 1$, if the orbit does not pass through a
singular point of the stable/unstable foliation, since there is only
one orbit in its Nielsen class.  On the other hand, if the orbit
passes through a singular point, then the sum of contributions is
still nonzero, but the orbit is counted with multiplicity.  Now, the
same holds for the sum of contributions from all the $(h,y)$ that
are in the same Nielsen class as $(\psi,x)$.

\s So far the discussion has been for $\Sigma$ closed. In our case,
the surface $S$ has nonempty boundary. Let $f:S\rightarrow S$ be the
first return map of the Reeb vector field
$R=R_{\varepsilon,\varepsilon'}$ constructed above. By construction,
$f|_{\bdry S}=id$ and is homotopic to $h$ and $\psi$. We cap off $S$
by attaching disks $D_i$ to obtain a closed surface $\Sigma$ and
extend $f$ to $\Sigma$ by extending by the identity map.  Since we
want to compare $f$ on $\Sigma$ to $\psi$ on $\Sigma$, we extend
$\psi$ to $\Sigma$ (as well as the stable and unstable foliations).
The extension of $\psi$ to $\Sigma$ is pseudo-Anosov, provided the
number of prongs on the boundary is not $n=1$. (Boundary monogons
could exist, although interior monogons do not.) We can avoid
monogons by passing to a ramified cover which is ramified at the
singular point of the monogon.  The nondegeneracy of $R$ implies the
nondegeneracy on the cover. Also, there is at most a finite-to-one
correspondence between periodic points on the cover of $S$ and the
periodic points on $S$. Hence, the number of Nielsen classes of $f$
with period $m$ grows exponentially with respect to $m$.  We will
then discard the fixed points of $f$ in the same Nielsen class as
$(f,x)$, where $x\in \bdry S$.

\subsection{Proof of Theorem~\ref{main}(2)}

Suppose $k\geq 3$. We prove that the direct limit
$\lim_{i\rightarrow\infty} HC_{\leq L_i}(\alpha_i)$ has exponential
growth with respect to the action. Recall we already proved the
isomorphism between $HC^\varepsilon(\alpha)$ and
$\displaystyle\lim_{i\rightarrow \infty} HC_{\leq L_i}(\alpha_i)$,
during the proof of Theorem~\ref{main}(1).

Let $C'=C'_{\leq L_i}(\alpha_i)$ be the subspace of $C=C_{\leq
L_i}(\alpha_i)$ generated by the orbits that are not covers of the
binding. Also let $C''=C''_{\leq L_i}(\alpha_i)$ be the subspace of
$C$ generated by the orbits which are covers of the binding. Then
$C=C'\oplus C''$ and $\bdry=\bdry'+\bdry''$, where
$\bdry':C'\rightarrow C'$ and $\bdry'':C''\rightarrow C''$, in view
of Theorem~\ref{theorem: cylinder}. Here $\bdry, \bdry', \bdry''$
only count holomorphic cylinders.  Also $C'$ is filtered by the open
book filtration (i.e., the number of times an orbit intersects a
page). Let $\F_j$ be the subspace of $C'$ generated by orbits which
intersect a page exactly $j$ times, and let $\F_{j,(\psi,x)}$ be the
subspace of $\F_j$ generated by orbits in the same Nielsen class as
$(\psi,x)$. Suppose $(\psi,x)$ is {\em good}, i.e., it is not an
even multiple of an orbit which has negative eigenvalues. The set of
such good Nielsen classes grows exponentially with respect to $j$,
provided $j< L_i$. (Recall that we can take the contact form so that
the action is arbitrarily close to the number of intersections with
the binding.) By Theorem~\ref{theorem: cylinder}, the boundary map
$\bdry:C\rightarrow C$, restricted to $\F_{j,(\psi,x)}$, has image
in $\F_{j,(\psi,x)}$. Since $\F_{j,(\psi,x)}$ can be split into even
and odd parity subspaces, and they have dimensions that differ by
one by Euler characteristic reasons, it follows that the homology of
$(\F_{j,(\psi,x)},\bdry)$ has dimension at least one. This proves
the exponential growth of $HC_{\leq L_i}(\alpha_i)$ with respect to
the action, provided we stay with orbits of action $\leq L_i$.
(Alternatively, one can say that the $E_1$-term of the spectral
sequence given by the open book filtration which converges to
$HC_{\leq L_i}(\alpha_i)$ has exponential growth with respect to the
action, and, moreover, the higher differentials of the spectral
sequence vanish.)

Let $\F_{j,(\psi,x)}(\alpha_i)$ be $\F_{j,(\psi,x)}$ for $\alpha_i$.
Suppose $(\psi,x)$ is good.  We claim that the map
$$\Phi_{\alpha_i\alpha_{i+1}}: C_{\leq L_i}(\alpha_i)\rightarrow
C_{\leq L_{i+1}}(\alpha_{i+1})$$ sends $\F_{j,(\psi,x)}(\alpha_i)$
to $\F_{j,(\psi,x)}(\alpha_{i+1})$.  In other words, no holomorphic
cylinder from a generator $\gamma$ of $\F_{j,(\psi,x)}(\alpha_i)$ to
a generator $\gamma'$ of $\F_{j,(\psi,x)}(\alpha_{i+1})$ intersects
$\R\times\gamma_0$. This follows from applying the same argument as
in the proof of Theorem~\ref{theorem: cylinder}.

Finally we show that, by choosing sufficiently large $L_i$, there is
a sequence $N_i\rightarrow \infty$ so that the map
\begin{equation}
\label{eqn: chain map on small portion} \Phi_{\alpha_i\alpha_{i+1}}:
H(\F_{j,(\psi,x)}(\alpha_i))\rightarrow
H(\F_{j,(\psi,x)}(\alpha_{i+1}))
\end{equation}
on the level of homology is injective, if $j\leq N_i$.  This is
sufficient to guarantee the exponential growth for the direct limit.
Recall that the orbits of $R_{\alpha_i}$ of action $K$ map to orbits
of $R_{\alpha_{i+1}}$ of action $\leq M_iM_{i+1} K$ under
$\Phi_{\alpha_i\alpha_{i+1}}$, and the orbits of $R_{\alpha_{i+1}}$
of action $K'$ map to orbits of $R_{\alpha_i}$ of action $\leq M_i
M_{i+1}K'$ under $\Phi_{\alpha_{i+1}\alpha_i}$.  Hence, in order to
compose
$\Phi_{\alpha_{i+1}\alpha_i}\circ\Phi_{\alpha_i\alpha_{i+1}}$ in the
cylindrical regime, we need $j\leq {L_i\over (M_iM_{i+1})^2}$.
Provided this holds, the usual chain homotopy proof shows that
$\Phi_{\alpha_i\alpha_{i+1}}$, restricted to
$H(\F_{j,(\psi,x)}(\alpha_i))$, has a left inverse and hence is
injective. Therefore, we choose $L_i$ so that, in addition to the
exhaustive condition, $N_i={L_i\over (M_iM_{i+1})^2}$ is strictly
increasing to $\infty$.

This completes the proof of Theorem~\ref{main}(2).

\subsection{Proof of Corollary~\ref{cor: infinitely many simple orbits}}
\label{subsection: proof of corollary}

Suppose $\alpha$ is nondegenerate. By Theorem~\ref{main}(1), there
is a linearized contact homology for any nondegenerate $\alpha$.
Observe that, if $R_\alpha$ only has finitely many simple orbits,
then $HC^\varepsilon(M,\alpha)$ will have at most polynomial growth
for any augmentation $\varepsilon$. The corollary then follows from
Theorem~\ref{main}(2).

Suppose $\alpha$ is degenerate and has a finite number of simple
orbits $\gamma_1,\dots,\gamma_l$.  Then, according to
Lemma~\ref{lemma: perturb finite number of orbits}, for any $N\gg 0$
there exists a $C^\infty$-small perturbation $\alpha_N$ of $\alpha$
so that the only periodic orbits of action $\leq N$ are isotopic to
multiple covers of $\gamma_i$. This means that the only free
homotopy classes which could possibly have generators in the
linearized contact homology group $\displaystyle\lim_{i\rightarrow
\infty} HC_{\leq L_i}(\alpha_i)$ are multiples of the $l$ simple
orbits. This contradicts the fact, sketched in the next two
paragraphs, that there are infinitely many simple free homotopy
classes in $M$ which have generators in
$\displaystyle\lim_{i\rightarrow \infty} HC_{\leq L_i}(\alpha_i)$.

We now sketch the proof, following
Gabai-Oertel~\cite[Lemma~2.7]{GO}. If $\gamma$ and $\gamma'$ are
closed orbits of the suspension flow of $\psi$, then $\gamma$ and
$\gamma'$ are both tangent to the suspension lamination
$\mathcal{L}$, which is an essential lamination if $c>{1\over n}$.
Let $u: \R\times S^1\rightarrow M$ be an immersed cylinder from
$\gamma$ to $\gamma'$. Then the lamination on $\R\times S^1$,
induced by pulling back $\mathcal{L}$ via $u$, cannot have any
$0$-gons or monogons, after normalizing/simplifying as in
\cite[Lemma~2.7]{GO}. Since, by Euler characteristic reasons, an
$m$-gon with $m>2$ implies the existence of a $0$-gon or a monogon,
it implies that $m$-gons with $m>2$ also do not exist. Hence the
only complementary regions of $u^{-1}(\mathcal{L})$ are annuli
$S^1\times[0,1]$ and $\R\times[0,1]$.

Now, if $c>{2\over n}$, then it is possible to replace $u$ by $u'$
which does not intersect the binding $\gamma_0$: Let $v: S^1\times
[0,1]\rightarrow M$ be an immersion whose interior maps to the
connected component $V$ of $M-\mathcal{L}$ that contains $\gamma_0$
and such that $S^1\times \{0,1\}$ maps to $\mathcal{L}$.  The map
$v$ is the restriction of $u$ to the closure of one connected
component of $u^{-1}(M-\mathcal{L})$. It is not hard to see that $v$
can be replaced by $v'$ so that they agree on $\bdry
(S^1\times[0,1])$ and $v'$ is disjoint from $\gamma_0$.  The same
technique works for $v:\R\times[0,1]\rightarrow M$. Therefore,
$\gamma$ and $\gamma'$ are freely homotopic in $M$ if and only they
correspond to the same Nielsen class.

\begin{rmk}
The above argument gives an easy proof of Theorem~\ref{main} if
$\psi$ is realized as a first return map of a Reeb vector field.
\end{rmk}

\s\n {\em Acknowledgements.} First and foremost, we thank Tobias
Ekholm's help throughout this project. We are also extremely
grateful to Francis Bonahon, Fr\'ed\'eric Bourgeois, Dragomir
Dragnev, Yasha Eliashberg, Fran\-\c cois Laudenbach, and Bob Penner
for invaluable discussions.  We thank Paolo Ghiggini and Michael
Hutchings for encouraging us to understand the relationship between
periodic monodromy and $S^1$-invariant contact structures. KH
wholeheartedly thanks Will Kazez and Gordana Mati\'c
--- this work would not exist without their collaboration in
\cite{HKM}.   KH also thanks l'universit\'e de Nantes for its
hospitality during his visit in the summer of 2005.

\end{document}